\documentclass[11pt]{article}
\usepackage[tbtags]{amsmath}
\usepackage{amssymb}
\usepackage{amsthm}
\usepackage[misc]{ifsym}
\usepackage{cases}
\usepackage{mathrsfs}
\usepackage[colorlinks,linkcolor=black,anchorcolor=black,citecolor=black]{hyperref}
\usepackage{CJK}
\numberwithin{equation}{section}
\normalsize
\title{\bf Relationship between General MP and DPP for the Stochastic Recursive Optimal Control Problem With Jumps: Viscosity Solution Framework \thanks{This work is supported by National Key R\&D Program of China (2022YFA1006104), National Natural Science Foundations of China (12271304, 11971266, 11831010), and Shandong Provincial Natural Science Foundations (ZR2022JQ01, ZR2020ZD24, ZR2019ZD42).}}
\author{\normalsize Bin Wang\thanks{\it School of Mathematics, Shandong University, Jinan 250100, P.R. China, E-mail: 202112005@mail.sdu.edu.cn} , Jingtao Shi\thanks{Corresponding Author. \it School of Mathematics, Shandong University, Jinan 250100, P.R. China, E-mail: shijingtao@sdu.edu.cn}}

\newtheorem{mythm}{Theorem}[section]
\newtheorem{mydef}{Definition}[section]
\newtheorem{mylem}{Lemma}[section]
\newtheorem{Remark}{Remark}[section]
\newtheorem{mycor}{Corollary}[section]
\begin{document}
\begin{CJK}{GBK}{kai}
\maketitle

\noindent{\bf Abstract:}\quad This paper is concerned with the relationship between general maximum principle and dynamic programming principle for the stochastic recursive optimal control problem with jumps, where the control domain is not necessarily convex. Relations among the adjoint processes, the generalized Hamiltonian function and the value function are proved, under the assumption of a smooth value function and within the framework of viscosity solutions, respectively. Some examples are given to illustrate the theoretical results.
\vspace{2mm}

\noindent{\bf Keywords:}\quad Backward stochastic differential equation with jumps, recursive optimal control, maximum principle, dynamic programming principle, viscosity solution

\vspace{2mm}

\section{Introduction}

\qquad Nonlinear {\it backward stochastic differential equation} (BSDE) was first introduced by Pardoux and Peng \cite{PardouxPeng1990}. Since then, the theory of BSDE has garnered significant attention from scholars and has been widely used in stochastic control, partial differential equations, differential games and financial mathematics. Duffie and Epstein  \cite{Duffie1992} pioneered of recursive utilities in continuous time, a specific type of BSDE. Recursive utility is an extension of the standard additive utility, where the instantaneous utility not only depend on the instantaneous consumption rate but also on the future utility. The {\it forward-backward stochastic differential equation} (FBSDE) was first initially explored by Antonelli \cite{Antonelli1993}, and subsequently underwent rapid theoretical and practical achievement. Optimal control problems of FBSDEs, among others, can be seen in Peng \cite{Peng1992, Peng1993}, El Karoui et al. \cite{EPQ1997}, Wu and Yu \cite{WuYu2008}, Yong \cite{Yong2010}, Wu \cite{Wu2013}, Hu \cite{Hu2017}, Hu et al. \cite{HuJiXue2018, HuJiXue2019}.

In literatures, there are two commonly use approaches to solve optimal control problems: Pontryagin's {\it maximum principle} (MP) and Bellman's {\it dynamic programming principle} (DPP). Peng \cite{Peng1990} proved the general MP for the stochastic control problem, and there are many works about DPP. Here, the word {\it general} means that the control domain is not necessarily convex and the diffusion term is control dependent. In recent years, many scholars have conducted researches on their relationship. This topic was initially explored by Bismut \cite{Bismut1978}, but for some special cases. Zhou \cite{Zhou1990} was the first one to establish the relationship between general MP and DPP for the stochastic optimal control problem, without assuming the smoothness of the value function, using the viscosity solution and the second-order adjoint equation (see also Yong and Zhou \cite{YongZhou1999}). Chen and Lv \cite{ChenLv2023} investigated the relationship between the general MP and DPP for control systems governed by stochastic evolution equations in infinite dimensional space, with the control variables entering into both the drift and the diffusion terms.

The connection between the MP and DPP of the optimal control problem of FBSDE was first obtained by Shi and Yu \cite{ShiYu2013} in the case that the value functions is smooth, where the recursive utility functional is given by the solution to a controlled FBSDE. Nie et al. \cite{NieShiWu2017} studied the connection between general MP and DPP for stochastic recursive optimal control problems under the viscosity solution's framework. Hu et al. \cite{HuJiXue2020} studied the relationship between general MP and DPP for the fully coupled forward-backward stochastic controlled system within the framework of viscosity solution. Li \cite{Li2023} researched the relationship between MP and DPP for stochastic recursive optimal control problems under volatility uncertainty.

In the recent decades, more and more research attentions have been drawn towards the optimal control problem for discontinuous stochastic systems. Models that incorporate jumps have become increasingly popular in finance and various areas of science and engineering. This trend underscores the heightened attention to optimal control problems involving systems of {\it stochastic differential equations with Poisson jumps} (SDEPs), i.e., jump diffusions. Optimal control problems of jump diffusions, can be seen in Situ \cite{Situ1991}, Tang and Li \cite{TangLi1994}, Pham \cite{Pham1998}, Li and Peng \cite{LiPeng2009}, \O ksendal and Sulem \cite{OksendalSulem2006, OksendalSulem2009}, Shi and Wu \cite{ShiWu2010}, Shi \cite{Shi2012, Shi2012-}, Song et al. \cite{SongTangWu2020}, Moon and Basar \cite{MoonBasar2022}, Song and Wu \cite{SongWu2023} and the references therein.

Framstad et al. \cite{FOS2004} first investigated the relationship between MP and DPP for the optimal control problem of jump diffusions, on the assumption that the value function is smooth. Shi and Wu \cite{ShiWu2011} established the relationship between the general MP and DPP within the framework of viscosity solutions. Chighoub and Mezerdi \cite{ChighoubMezerdi2014} found the relationship between MP and DPP in singular control of jump diffusions. Shi \cite{Shi2014} explored the relationship between MP and DPP for stochastic recursive optimal control problems of jump diffusions. Sun et al. \cite{SunGuoZhang2018} investigated the relationship between MP and DPP for Markov regime-switching forward-backward stochastic controlled system with jumps. However, to the best of our knowledge, there is no results about the relationship between general MP and DPP for stochastic recursive optimal control problems of jump diffusions. So, in this paper, we will research this topic and fill in the gap in the literatures.

The contribution of this paper can be summarized as follows.

(1) This paper delves into the stochastic recursive optimal control problem with Poisson jumps within the framework of viscosity solutions. Notably, the control domain is permitted to be non-convex, and the control variables are incorporated into the coefficients of the drift, diffusion and jumps. We extend the work of \cite{NieShiWu2017} to incorporate jumps, generalize the findings of \cite{ShiWu2011} to a forward-backward system, and broaden the scope of \cite{Shi2014} to encompass more general scenarios. Overall, our consideration of a more intricate system enhances its practical applicability in scenarios involving both continuous variations and abrupt jumps.

(2) Since Zheng and Shi \cite{ZhengShi2023}'s emphasis was primarily on the partially observed system, the adjoint equations introduced in it was insufficient to derive the relationship of the adjoint process $(p,q,\tilde{q})$ and the value function $V$. However, inspired by \cite{NieShiWu2017}, a novel second-order adjoint equation (\ref{second-order adjoint equation}) and Hamiltonian function $H$ (\ref{WangShi's Hamilton function}) are introduced in this paper, differentiating it from the work of \cite{ZhengShi2023}. Additionally, we establish the relationships (\ref{relation of adjoint equation}) between the adjoint equations and Hamiltonian functions in both the above mentioned papers. Moreover, by (\ref{WangShi's Hamilton function}), (\ref{relation of adjoint equation}) and \cite{ZhengShi2023}, a general MP (Theorem \ref{general MP}) is given for our problem, which is new of its own.

(3) While the relationship under the smooth assumption is investigated, as expressed in (\ref{relation of p and V}), by leveraging the results in (\ref{relation of adjoint equation}) of this paper, we fill in a gap of Remark 3.1 in \cite{Shi2014}.

(4) During the derivation of Theorem \ref{nonsmooth time theorem}, we discovered that (63) in \cite{ShiWu2011} lacks a term $\int_{\mathcal{E}}\int_s^\tau \big\langle \tilde{Q}_{(r,e)}\bar{f}(r,e),\bar{f}(r,e) \big\rangle \nu(de)dr$ compared to (\ref{estimate of PXX with tau}). By adding this missing term, a new result (\ref{revised result in ShiWu2011}) is obtained, which indicates there is no need to construct a new function $\mathcal{G}$ in \cite{ShiWu2011}.

The remaining sections of this paper is outlined as follows. In section 2, we state our problem and give some preliminary results about the general MP and DPP. Section 3 and 4 exhibit the main results of this paper. We give the relationship of the general MP and DPP under the assumption of a smooth value function in section 3, and delve into the new relation under the framework of viscosity solutions in section 4. In section 5, we give some examples to explain the theoretical results. Finally, in section 6, some concluding remarks are given.

{\it Notations.}\quad In this paper, $\mathbf{R}^n$ denotes the $n$-dimensional Euclidean space with scalar product $\langle\cdot, \cdot\rangle$ and norm $|\cdot|$, $\mathcal{S}^n$ denotes the $n \times n$ symmetric matrix space, $D g, D^2 g$ denotes the gradient and the Hessian matrix of the differentiable function $g$, respectively, $\top$ appearing as a superscript denotes the transpose of a matrix, and $C>0$ denotes a generic constant which may take different values in different places. We denote by $f\in C^{i,j}([0,T]\times\mathbf{R}^n)$ that the real-valued function $f:[0,T]\times\mathbf{R}^n\rightarrow\mathbf{R}$ is $i$-order continuously differentiable with $t$ and $j$-order continuously differentiable with $x$, for $i,j=0,1,2,\cdots$ (we denote $C\equiv C^{0,0}$ when $i=j=0$).

\section{\bf Problem statement, general MP and DPP}

\qquad Let $T>0$ be fixed and $\mathbf{U} \subset \mathbf{R}^k$ be nonempty. Let $(\mathcal{E}, \mathcal{B}(\mathcal{E}))$ be a Polish space with the $\sigma$-finite measure $\nu(\cdot)$ on $\mathcal{E}$. We denote by $\mathcal{L}^2(\mathcal{E},\mathcal{B}(\mathcal{E}),\nu;\mathbf{R}^n)$ or $\mathcal{L}^2$ the set of square-integrable functions $f(\cdot):\mathcal{E}\rightarrow\mathbf{R}^n$ such that $\left\|f(\cdot)\right\|_{\mathcal{L}^2}^2:=\int_{\mathcal{E}}|f(e)|^2\nu(de)<\infty $.

Given $t\in[0,T)$, we denote $\mathcal{U}^w[t,T]$ the set of all 6-tuples $(\Omega,\mathcal{F},\mathbf{P},W_\cdot,\tilde{N}(\cdot,\cdot);u_\cdot)$ satisfying the following conditions: (i) $(\Omega,\mathcal{F},\mathbf{P})$ is a complete probability space, and $\mathbf{E}$ denotes the expectation under the probability measure $\mathbf{P}$; (ii) $\{W_s\}_{s\geqslant t}$ is a one-dimensional standard Brownian motion defined on $(\Omega,\mathcal{F},\mathbf{P})$ over $[t,T]$ (with $W_t=0$, $\mathbf{P}\text{-a.s.}$); (iii) $\{N(d e, d s)\}_{s\geqslant t}$ is a Poisson random measure on $\left(\mathbf{R}^+ \times \mathcal{E}, \mathcal{B}\left(\mathbf{R}^+ \right)\times \mathcal{B}\left(\mathcal{E}\right)\right)$ on $(\Omega,\mathcal{F},\mathbf{P})$ over $[t,T]$ and for any $E \in \mathcal{B}\left(\mathcal{E}\right)$, $\nu\left(E\right)<\infty$, then the compensated Poisson random measure can be defined by $\tilde{N}(d e, d s):=N(d e, d s)-\nu(d e) d s$. (iv) The filtration $\left\{\mathcal{F}_s^t;t\leqslant s\leqslant T\right\}$ is generated as the following $\mathcal{F}_s^t:=\mathcal{F}_s^{W} \vee \mathcal{F}_s^{N} \vee \mathcal{N}$, where $W, N$ are mutually independent under $\mathbf{P}$, $\mathcal{F}_s^{W}, \mathcal{F}_s^{N}$ are the $\mathbf{P}$-completed natural filtrations generated by $W, N$ respectively, and $\mathcal{N}$ denotes the totality of $\mathbf{P}$-null sets. In particular, if $t=0$ we write $\mathcal{F}_s\equiv \mathcal{F}_s^t$.  (v) $u:[t,T]\times\Omega\rightarrow\mathbf{U}$ is an $\{\mathcal{F}_s^t\}_{s\geq t}$-predictable process on $(\Omega,\mathcal{F},\mathbf{P})$ such that $\mathbf{E}\big[\sup\limits_{t\leqslant s\leqslant T}|u_s|^k\big]^{1/k}<\infty$ for $k>1$.

For given $t\in[0,T)$, we write $(\Omega,\mathcal{F},\mathbf{P},W_\cdot,N(\cdot,\cdot);u_\cdot)\in\mathcal{U}^w[t,T]$, but occasionally we write only $u_\cdot\in\mathcal{U}^w[t,T]$ if no ambiguity exists. Any $u_\cdot\in\mathcal{U}^w[t,T]$ is called an admissible control. We denote by $L_{\mathcal{F}}^2([t,T];\mathbf{R}^n)$ the space of all $\mathbf{R}^n$-valued $\mathcal{F}_s^t$-adapted processes $\phi(\cdot)$ such that $\mathbf{E}\int_t^T |\phi(s)|^2ds<\infty$, by $L_{\mathcal{F},p}^2([t,T];\mathbf{R})$ the space of all $\mathbf{R}$-valued $\mathcal{F}_s^t$-predictable processes $\varphi(\cdot)$ such that $\mathbf{E}\int_t^T |\varphi(s)|^2ds<\infty$, and by $F_p^2([t,T]\times\mathcal{E};\mathbf{R})$ the space of all $\mathbf{R}$-valued $\mathcal{F}_s^t$-predictable processes $\tilde{\varphi}(\cdot,\cdot)$ such that $\mathbf{E}\int_t^T \int_{\mathcal{E}} |\tilde{\varphi}(s,e)|^2\nu(de)ds<\infty$.

For any initial time and state $(t, x) \in[0, T) \times \mathbf{R}^n$, consider the state process $\{X_s^{t, x ; u};t\leqslant s\leqslant T\} \in \mathbf{R}^n$ given by the following controlled SDEP:
\begin{equation}
\left\{\begin{aligned}
d X_s^{t, x ; u}= & \ b\left(s, X_s^{t, x ; u}, u_s\right) d s+\sigma\left(s, X_s^{t, x ; u}, u_s\right) d W_s+ \int_{\mathcal{E}}f(s, X_{s-}^{t,x; u}, u_s, e)\tilde{N}(d e,d s), \\
X_t^{t, x ; u}= & \ x .
\end{aligned}\right.
 \label{SDEP equation}
\end{equation}
Here $\{u_s;t\leqslant s\leqslant T\}\in \mathbf{U}$ is the control process, $b:[0, T] \times \mathbf{R}^n \times \mathbf{U} \rightarrow \mathbf{R}^n, \sigma:[0, T] \times \mathbf{R}^n \times \mathbf{U} \rightarrow \mathbf{R}^{n }$ and $f:[0, T] \times \mathbf{R}^n \times \mathbf{U} \times \mathcal{E} \rightarrow \mathbf{R}^n$ are given functions.

For given $u_\cdot\in \mathcal{U}^w[t, T]$ and $x \in \mathbf{R}^n$, any $X^{t, x ; u}_\cdot\in L_{\mathcal{F}}^2([t,T];\mathbf{R}^n)$ is called a solution to ($\ref{SDEP equation}$) if ($\ref{SDEP equation}$) holds. Note that $X^{t, x ; u}_\cdot$, in general, is supposed to be RCLL (i.e., right-continuous with left-hand limits). We refer to such $(X^{t, x ; u}_\cdot, u_\cdot)$ as an admissible pair.

We make the following assumption. \\
$(\mathbf{H1})\quad b, \sigma, f$ are uniformly continuous in $(s, x, u)$, and there exists a constant $C>0$, such that
$$
\left\{\begin{array}{l}
|b(s, x, u)-b(s, \hat{x}, \hat{u})|+|\sigma(s, x, u)-\sigma(s, \hat{x}, \hat{u})| \leqslant C(|x-\hat{x}|+|u-\hat{u}|), \\
\|f(s, x, u, \cdot)-f(s, \hat{x}, \hat{u}, \cdot)\|_{\mathcal{L}^2}\leqslant C(|x-\hat{x}|+|u-\hat{u}|), \\
|b(s, x, u)|+|\sigma(s, x, u)|+\|f(s, x, u, \cdot)\|_{\mathcal{L}^2} \leqslant C(1+|x|+|u|),
\end{array}\right.
$$
for all $s \in[0, T], x, \hat{x} \in \mathbf{R}^n, u, \hat{u} \in \mathbf{U}$. For any  $u_\cdot \in \mathcal{U}^w[t, T]$, under (H1), it is obvious that SDEP (\ref{SDEP equation}) has a unique solution $X_\cdot^{t,x;u}\in L_{\mathcal{F}}^2([t,T];\mathbf{R}^n)$ (see Ikeda and Watanabe \cite{IkedaWatanabe1981}).

Next, we introduce the following controlled {\it backward stochastic differential equation with jumps} (BSDEP) which is coupled with (\ref{SDEP equation}):
\begin{equation}
\left\{\begin{aligned}
-d Y_s^{t, x ; u}= & \ g\left(s, X_s^{t, x ; u}, Y_s^{t, x ; u}, Z_s^{t, x ; u}, \tilde{Z}_{(s,\cdot)}^{t,x;u}, u_s\right) d s -Z_s^{t, x ; u} d W_s -\int_{\mathcal{E}}\tilde{Z}_{(s,e)}^{t,x;u} \tilde{N}(d e, d s), \\
Y_T^{t, x ; u}= & \ \phi\left(X_T^{t, x ; u}\right),
\end{aligned}\right.
 \label{BSDEP equation}
\end{equation}
whose solution is a process triple $\left\{\left(Y_s^{t, x ; u}, Z_s^{t, x ; u}, \tilde{Z}_{(s,\cdot)}^{t, x ; u}\right);t\leqslant s\leqslant T\right\}$. Here $g:[0, T] \times \mathbf{R}^n \times \mathbf{R} \times \mathbf{R} \times \mathcal{L}^2 \times \mathbf{U} \rightarrow \mathbf{R}, \ \phi:\mathbf{R}^n \rightarrow \mathbf{R}$ are given functions.
We assume that\\
$(\mathbf{H2}) \quad g,\phi$ are uniformly continuous in $(s,x,y,z,\tilde{z},u)$, and there exists a constant $C>0$, such that
$$
\left\{\begin{array}{l}
|g(s, x, y, z,\tilde{z} , u)-g(s, \hat{x}, \hat{y}, \hat{z}, \hat{\tilde{z}}, \hat{u}))| \leqslant C\left(|x-\hat{x}|+|y-\hat{y}|+|z-\hat{z}|+\| \tilde{z}-\hat{\tilde{z}}\|_{\mathcal{L}^2}+|u-\hat{u}|\right), \\
|g(s, x, 0,0,0, u)|+|\phi(x)| \leqslant C(1+|x|), \quad|\phi(x)-\phi(\hat{x})| \leqslant C|x-\hat{x}|,
\end{array}\right.
$$
for all $s \in[0, T], x, \hat{x} \in \mathbf{R}^n, y, \hat{y} \in \mathbf{R}, z, \hat{z} \in \mathbf{R}, \tilde{z}, \hat{\tilde{z}} \in  \mathcal{L}^2, u, \hat{u} \in \mathbf{U}$.

Then, for any $u_\cdot \in \mathcal{U}^w[t, T]$ and the unique solution $X_\cdot^{t, x ; u}\in L_{\mathcal{F}}^2([t,T];\mathbf{R}^n)$ to (\ref{SDEP equation}), under (H2), it is classical that BSDEP (\ref{BSDEP equation}) admits a unique solution $\left(Y_\cdot^{t, x ; u}, Z_\cdot^{t, x ; u}, \tilde{Z}_{(\cdot,\cdot)}^{t, x ; u} \right)\in L_{\mathcal{F}}^2([t,T];\mathbf{R})\times L_{\mathcal{F},p}^2([t,T];\mathbf{R})\times F_p^2([t,T]\times\mathcal{E};\mathbf{R})$ (see Tang and Li \cite{TangLi1994} or Barles et al. \cite{BBP1997}).

Given $u_\cdot \in \mathcal{U}^w[t, T]$, we introduce the cost functional
\begin{equation}
J(t, x ; u_\cdot):=-\left.Y_s^{t, x ; u}\right|_{s=t}, \quad(t, x) \in[0, T] \times \mathbf{R}^n .
 \label{cost functional}
\end{equation}

Our stochastic recursive optimal control problem with jumps is the following.

\noindent{\bf Problem (SROCPJ).}\quad For given $(t, x) \in[0, T) \times \mathbf{R}^n$, to minimize ($\ref{cost functional}$) subject to (\ref{SDEP equation}), (\ref{BSDEP equation}) over $\mathcal{U}^w[t, T]$. We define the value function as
\begin{equation}
V(t, x):=\inf _{u_\cdot \in \, \mathcal{U}^w[t, T]} J(t, x ; u_\cdot), \quad(t, x) \in[0, T] \times \mathbf{R}^n.
 \label{value function}
\end{equation}
Any $\bar{u}_\cdot$ satisfied (\ref{value function}) is called an optimal control, and the corresponding solution $\left(\bar{X}_\cdot^{t, x ; \bar{u}}, \bar{Y}_\cdot^{t, x ; \bar{u}}\right.$, $\left.\bar{Z}_\cdot^{t, x ; \bar{u}},\bar{\tilde{Z}}_{(\cdot,\cdot)}^{t,x;\bar{u}}\right)\in L_{\mathcal{F}}^2([t,T];\mathbf{R}^n)\times L_{\mathcal{F}}^2([t,T];\mathbf{R})\times L_{\mathcal{F},p}^2([t,T];\mathbf{R})\times F_p^2([t,T]\times\mathcal{E};\mathbf{R})$ to (\ref{SDEP equation}) and (\ref{BSDEP equation}) is called an optimal trajectory.

\begin{Remark}
From their Proposition 3.1 of \cite{LiPeng2009}, we know that under (H1), (H2), the above value function is a deterministic function, then it is well-defined.
\end{Remark}

In the following, we introduce the DPP and general MP approaches for {\bf Problem (SROCPJ)} in the literatures, respectively, to characterize the optimal control and the value function. We first introduce the following generalized HJB equation:
\begin{equation}
\left\{\begin{array}{l}
-v_t(t, x)+ \underset{u \in \mathbf{U}}\sup \ G\left(t, x,-v(t, x),-v_x(t, x),-v_{x x}(t, x), u\right)=0, \quad (t, x) \in[0, T) \times \mathbf{R}^n, \\
v(T, x)=-\phi(x), \quad \forall x \in \mathbf{R}^n,
\end{array}\right.
 \label{HJB equation}
\end{equation}
where the generalized Hamiltonian function $G:[0, T] \times$ $\mathbf{R}^n \times \mathbf{R} \times \mathbf{R}^n \times \mathcal{S}^n \times \mathbf{U} \rightarrow \mathbf{R}$ is defined as
\begin{equation}
\begin{aligned}
& G(t, x, -v(t,x), -v_x(t,x), -v_{x x}(t,x), u) \\
:=& -\frac{1}{2} \operatorname{tr}\left(\sigma(t, x,u)^\top v_{x x}(t,x) \sigma(t, x,u)\right)-\left\langle v_x(t,x), b(t, x, u)\right\rangle \\
& -\int_{\mathcal{E}}\left[v\left(t,x+f(t,x,u,e)\right)-v(t,x)-\left\langle v_x(t,x),f(t,x,u,e)\right\rangle\right]\nu (de) \\
& +g\left(t, x, -v(t,x), -\sigma(t, x,u)^\top v_x(t,x),-\int_{\mathcal{E}}\left[v(t,x+f(t,x,u,e))-v(t,x)\right]\nu(de),u\right).
\end{aligned}
 \label{generalized Hamiltonian function G}
\end{equation}

The following result can be inferred by their Theorem 4.1 of \cite{LiPeng2009}.

\begin{mylem}
Let $(H1), (H2)$ hold and $(t, x) \in[0, T) \times \mathbf{R}^n$ be fixed. Suppose $V \in C^{1,2}\left([0, T] \times \mathbf{R}^n\right)$. Then $V$ is a solution to (\ref{HJB equation}).
\end{mylem}

In convenient to state the general MP, for given $(t, x) \in[0, T) \times \mathbf{R}^n$, we regard the above coupled ($\ref{SDEP equation}$) and ($\ref{BSDEP equation}$) as a controlled {\it forward-backward SDEP} (FBSDEP):
\begin{equation}
\left\{\begin{aligned}
d X_s^{t, x ; u}= & \ b\left(s, X_s^{t, x ; u}, u_s\right) d s+\sigma\left(s, X_s^{t, x ; u}, u_s\right) d W_s+\int_{\mathcal{E}}f(s,X_{s-}^{t,x;u},u_s,e)\tilde{N}(d e,d s), \\
-d Y_s^{t, x ; u}= & \ g\left(s, X_s^{t, x ; u}, Y_s^{t, x ; u}, Z_s^{t, x ; u},\tilde{Z}_{(s,\cdot)}^{t,x;u}, u_s\right) d s-Z_s^{t, x ; u} d W_s-\int_{\mathcal{E}}\tilde{Z}_{(s,e)}^{t,x;u}\tilde{N}(d e,d s), \\
X_t^{t, x ; u}= & \ x, \ \ \ \ \ Y_T^{t, x ; u}=\phi\left(X_T^{t, x ; u}\right),
\end{aligned}\right.
 \label{FBSDEP equation}
\end{equation}
whose solution is a process quadruple $\left(\bar{X}_\cdot^{t, x ; \bar{u}}, \bar{Y}_\cdot^{t, x ; \bar{u}}\right.$, $\left.\bar{Z}_\cdot^{t, x ; \bar{u}},\bar{\tilde{Z}}_{(\cdot,\cdot)}^{t,x;\bar{u}}\right)\in L_{\mathcal{F}}^2([t,T];\mathbf{R}^n)\times L_{\mathcal{F}}^2([t,T];\mathbf{R})\\\times L_{\mathcal{F},p}^2([t,T];\mathbf{R})\times F_p^2([t,T]\times\mathcal{E};\mathbf{R})$.

In addition, we need the following assumption.\\
$(\mathbf{H3})\quad (1)\ b,\sigma,f,\phi$ are twice continuously differentiable in $x$, and their derivatives $b_x, b_{x x}, \sigma_x$, $\sigma_{x x}, \|f_x(\cdot)\|_{\mathcal{L}^2}, \|f_{x x}(\cdot)\|_{\mathcal{L}^2}, \phi_x, \phi_{x x}$ are continuous in $(s,x,u)$ and uniformly bounded. Moreover there exists a modulus of continuity $\varpi:[0,\infty)\rightarrow[0,\infty)$ and a constant $C>0$ such that
$$
\left\{\begin{array}{l}
\left|b_x(s, x, u)-b_x(s, \hat{x}, \hat{u})\right|+\left|\sigma_x(s, x, u)-\sigma_x(s, \hat{x}, \hat{u})\right| \leqslant C(|x-\hat{x}|+|u-\hat{u}|), \\
\left\|f_x(s, x, u, \cdot)-f_x(s, \hat{x}, \hat{u}, \cdot)\right\|_{\mathcal{L}^2} \leqslant C(|x-\hat{x}|+|u-\hat{u}|), \\
\left|b_{x x}(s, x, u)-b_{x x}(s, \hat{x}, \hat{u})\right|+\left|\sigma_{x x}(s, x, u)-\sigma_{x x}(s, \hat{x}, \hat{u})\right| \leqslant \varpi(|x-\hat{x}|+|u-\hat{u}|), \\
\left\|f_{x x}(s, x, u, \cdot)-f_{x x}(s, \hat{x}, \hat{u}, \cdot)\right\|_{\mathcal{L}^2} \leqslant \varpi(|x-\hat{x}|+|u-\hat{u}|), \\
\left|\phi_x(x)-\phi_x(\hat{x})\right| \leqslant C(|x-\hat{x}|),\quad \quad \left|\phi_{x x}(x)-\phi_{x x}(\hat{x})\right| \leqslant \varpi(|x-\hat{x}|),
\end{array}\right.
$$
for all $s \in[0, T], x, \hat{x} \in \mathbf{R}^n, u, \hat{u} \in \mathbf{U}$.

(2)\ $g$ is twice continuously differentiable with respect to $(x,y,z,\tilde{z})$, $g,Dg,D^2g$ are continuously in $(s,x,y,z,\tilde{z},u)$, $Dg,D^2g$ are uniformly bounded and there exists a modulus of continuity $\varpi:[0,\infty)\rightarrow[0,\infty)$ and a constant $C>0$ such that
$$
\left\{\begin{array}{l}
\left|Dg(s,x,y,z,\tilde{z},u)-Dg(s,\hat{x},\hat{y},\hat{z},\hat{\tilde{z}},\hat{u})\right|\\
\leqslant C\left(|x-\hat{x}|+|y-\hat{y}|+|z-\hat{z}|+\| \tilde{z}-\hat{\tilde{z}}\|_{\mathcal{L}^2}+|u-\hat{u}|\right), \\
\left|D^2g(s,x,y,z,\tilde{z},u)-D^2g(s,\hat{x},\hat{y},\hat{z},\hat{\tilde{z}},\hat{u})\right|\\
\leqslant \varpi\left(|x-\hat{x}|+|y-\hat{y}|+|z-\hat{z}|+\| \tilde{z}-\hat{\tilde{z}}\|_{\mathcal{L}^2}+|u-\hat{u}|\right),
\end{array}\right.
$$
for all $s \in[0, T], x, \hat{x} \in \mathbf{R}^n, y, \hat{y} \in \mathbf{R}, z, \hat{z} \in \mathbf{R}, \tilde{z}, \hat{\tilde{z}} \in  \mathcal{L}^2, u, \hat{u} \in \mathbf{U}$.

\cite{ZhengShi2023} studied a partially observed progressive optimal control problem of FBSDEP, where the control domain is not necessarily convex, and the control variable enter into all the coefficients. A general MP was obtained (their Theorem 3.6 in \cite{ZhengShi2023}). If we consider the completely observed case, that is, let $\sigma_2, f_2,z_\cdot^{2,u}, \tilde{z}_{(\cdot,\cdot)}^{2,u}\equiv 0$ in (2.1) of \cite{ZhengShi2023} so that the state equation is consistent with (\ref{FBSDEP equation}), and let $b_2,\sigma_3,f_3\equiv0$ to disappear the observation equation (2.3) in \cite{ZhengShi2023}. Moreover, we let $l,\Phi\equiv 0, \Gamma(y_0^u)\equiv y_0^u$, so that the cost functional (2.4) in \cite{ZhengShi2023} is reduced to the recursive case (\ref{cost functional}). Thus, we can state the general MP for our {\bf Problem (SROCPJ)}.

\indent Let $\bar{u}_\cdot$ be an optimal control, and $\left(\bar{X}_\cdot^{t, x ; \bar{u}}, \bar{Y}_\cdot^{t, x ; \bar{u}}, \bar{Z}_\cdot^{t, x ; \bar{u}},\bar{\tilde{Z}}_{(\cdot,\cdot)}^{t,x;\bar{u}}\right)\in L_{\mathcal{F}}^2([t,T];\mathbf{R}^n)\times L_{\mathcal{F}}^2([t,T];\mathbf{R})\times L_{\mathcal{F},p}^2([t,T];\mathbf{R})\times F_p^2([t,T]\times\mathcal{E};\mathbf{R})$ be the corresponding optimal trajectory. For all $s \in[0, T]$, we denote
$$
\begin{aligned}
\bar{b}(s) & :=b\left(s, \bar{X}_s^{t, x ; \bar{u}}, \bar{u}_s\right), \quad \bar{\sigma}(s) :=\sigma\left(s, \bar{X}_s^{t, x ; \bar{u}}, \bar{u}_s\right), \\
\bar{g}(s) & :=g\left(s, \bar{X}_s^{t, x ; \bar{u}}, \bar{Y}_s^{t, x ; \bar{u}}, \bar{Z}_s^{t, x ; \bar{u}}, \bar{\tilde{Z}}_{(s,\cdot)}^{t,x;\bar{u}},\bar{u}_s\right), \quad
\bar{f}(s,e) :=f\left(s,\bar{X}_s^{t, x ; \bar{u}}, \bar{u}_s,e\right),
\end{aligned}
$$
and similar notations used for all their derivatives.

For given $(t, x) \in[0, T) \times \mathbf{R}^n$, we introduce the following first- and second-order adjoint BSDEPs:
\begin{equation}
\left\{\begin{aligned}
-dp_s=& \ \bigg\{\bar{b}_x(s)^\top p_s+\bar{\sigma}_x(s)^\top q_s+\bar{g}_x(s)+\bar{g}_y(s)p_s+\bar{g}_z(s)\left[\bar{\sigma}_x(s)^\top p_s+q_s\right] \\
&\ +\bar{g}_{\tilde{z}}(s)\int_{\mathcal{E}}\left[\bar{f}_x(s,e)^\top p_s+\tilde{q}_{(s,e)}+\bar{f}_x(s,e)^\top \tilde{q}_{(s,e)}\right]\nu(de) \\
&\ +\int_{\mathcal{E}}\bar{f}_x(s,e)^\top \tilde{q}_{(s,e)}\nu(de)\bigg\}ds-q_s d W_s-\int_{\mathcal{E}}\tilde{q}_{(s,e)}\tilde{N}(d e,d s),\\
p_T=& \ \phi_x(\bar{X}_T^{t,x;\bar{u}}),
\end{aligned}\right.
 \label{first-order adjoint equation}
\end{equation}
whose solution is the process triple $\left(p_\cdot, q_\cdot, \tilde{q}_{(\cdot,\cdot)}\right)\in L_{\mathcal{F}}^2([t,T];\mathbf{R}^n)\times L_{\mathcal{F},p}^2([t,T];\mathbf{R}^n)\times F_p^2([t,T]\times\mathcal{E};\mathbf{R}^n)$,
\begin{equation}
\left\{\begin{aligned}
-dP_s&=\Bigg\{P_s \bar{b}_x(s)+\bar{b}_x(s)^\top P_s+\bar{b}_{x x}(s)^\top p_s+\bar{\sigma}_x(s)^\top Q_s+Q_s \bar{\sigma}_x(s)+\bar{\sigma}_{x x}(s)^\top q_s \\
&\qquad +\bar{\sigma}_x(s)^\top P_s\bar{\sigma}_x(s)+\bar{g}_y(s)P_s+\bar{g}_z(s)\bar{\sigma}_x(s)^\top P_s+{P}_s\bar{g}_z(s)\bar{\sigma}_x(s) \\
&\qquad +\bar{\sigma}_{x x}(s)^\top \bar{g}_z(s)p_s+\bar{g}_z(s)Q_s+\Psi(s) D^2\bar{g}(s) \Psi(s)^\top \\
&\qquad +\int_{\mathcal{E}}\bigg\{\bar{f}_x(s,e)^\top \tilde{Q}_{(s,e)} +\tilde{Q}_{(s,e)}\bar{f}_x(s,e)+\bar{f}_x(s,e)^\top \tilde{Q}_{(s,e)}\bar{f}_x(s,e) \\
&\qquad +\bar{f}_x(s,e)^\top P_sf_x(s,e)+\bar{f}_{x x}(s,e)^\top \tilde{q}_{(s,e)}+\bar{f}_{x x}(s,e)^\top \bar{g}_{\tilde{z}}(s) p_s \\
&\qquad +\bar{f}_{x x}(s,e)^\top \bar{g}_{\tilde{z}}(s) \tilde{q}_{(s,e)}+\bar{g}_{\tilde{z}}(s)\bar{f}_x(s,e)^\top P_s+P_s\bar{g}_{\tilde{z}}(s)\bar{f}_x(s,e) \\
&\qquad +\bar{g}_{\tilde{z}}(s)\tilde{Q}_{(s,e)}\bar{f}_x(s,e)+\bar{g}_{\tilde{z}}(s)\bar{f}_x(s,e)^\top \tilde{Q}_{(s,e)}\bar{f}_x(s,e) \\
&\qquad +\bar{g}_{\tilde{z}}(s)\bar{f}_x(s,e)^\top P_s\bar{f}_x(s,e)+\bar{g}_{\tilde{z}}(s)\bar{f}_x(s,e)^\top \tilde{Q}_{(s,e)}+\bar{g}_{\tilde{z}}(s)\tilde{Q}_{(s,e)}\bigg\}\nu(de) \Bigg\}ds \\
&\quad -Q_s d W_s-\int_{\mathcal{E}}\tilde{Q}_{(s,e)}\tilde{N}(d e,d s),\\
P_T=& \ \phi_{x x}(\bar{X}_T^{t,x;\bar{u}}),
\end{aligned}\right.
 \label{second-order adjoint equation}
\end{equation}
whose solution is the matrix-valued process triple $\left(P_\cdot, Q_\cdot, \tilde{Q}_{(\cdot,\cdot)}\right)\in L_{\mathcal{F}}^2([t,T];\mathcal{S}^n)\times L_{\mathcal{F},p}^2([t,T];\\\mathcal{S}^n)\times F_p^2([t,T]\times\mathcal{E};\mathcal{S}^n)$,
\begin{equation}
\quad \left\{\begin{aligned}
-dP^*_s&= \bigg\{h_s\Psi^*(s)D^2\bar{g}(s)\Psi^*(s)^\top+\bar{b}_{x x}(s)^\top p^*_s+\bar{b}_x(s)^\top P^*_s+P^*_s\bar{b}_x(s) \\
&\qquad +\bar{\sigma}_{x x}(s)^\top q^*_s+\bar{\sigma}_x(s)^\top Q^*_s+Q^*_s\bar{\sigma}_x(s)+\bar{\sigma}_x(s)^\top P^*_s\bar{\sigma}_x(s) \\
&\qquad +\int_{\mathcal{E}}\Big[\bar{f}_{x x}(s,e)^\top \tilde{q}^*_{(s,e)}+\bar{f}_x(s,e)^\top \tilde{Q}^*_{(s,e)}+\tilde{Q}^*_{(s,e)}\bar{f}_x(s,e) \\
&\qquad +\bar{f}_x(s,e)^\top \tilde{Q}^*_{(s,e)}\bar{f}_x(s,e)+\bar{f}_x(s,e)^\top P^*_s\bar{f}_x(s,e)\Big]\nu(de)\bigg\}d s\\
&\quad -Q^*_s d W_s-\int_{\mathcal{E}}\tilde{Q}^*_{(s,e)}\tilde{N}(d e,d s), \\
P^*_T=& \ h_T \phi_{x x}(\bar{X}_T),
\end{aligned}\right.
 \label{second-order dual equation}
\end{equation}
whose solution is the matrix-valued process triple $\left(P_\cdot^*, Q_\cdot^*, \tilde{Q}_{(\cdot,\cdot)}^*\right)\in L_{\mathcal{F}}^2([t,T];\mathcal{S}^n)\times L_{\mathcal{F},p}^2([t,T];\\\mathcal{S}^n)\times F_p^2([t,T]\times\mathcal{E};\mathcal{S}^n)$, where
$$
\begin{aligned}
\Psi(s):=&\bigg[I_{n\times n},p_s,\bar{\sigma}_x(s)^\top p_s+q_s,\int_{\mathcal{E}}\left(\bar{f}_x(s,e)^\top p_s+\tilde{q}_{(s,e)}+\bar{f}_x(s,e)^\top \tilde{q}_{(s,e)}\right)\nu(de)\bigg], \\
\Psi^*(s):=&\bigg[I_{n\times n},p_s^*,\bar{\sigma}_x(s)^\top p_s^*+q_s^*,\int_{\mathcal{E}}\left(\bar{f}_x(s,e)^\top p_s^*+\tilde{q}_{(s,e)}^*+\bar{f}_x(s,e)^\top \tilde{q}_{(s,e)}^*\right)\nu(de)\bigg],
\end{aligned}
$$
and the following adjoint FBSDEP:
\begin{equation}
\left\{\begin{aligned}
-dp^*_s& =\left[h_s\bar{g}_x(s)+\bar{b}_x(s)^\top p^*_s+\bar{\sigma}_x(s)^\top q^*_s+\int_{\mathcal{E}}\bar{f}_x(s,e)^\top \tilde{q}^*_{(s,e)}\nu(de)\right]ds \\
& \quad \ -q^*_sd W_s-\int_{\mathcal{E}}\tilde{q}^*_{(s,e)}\tilde{N}(d e,d s),s\in[t,T], \\
dh_s& =h_s\bar{g}_y(s)ds+h_s\bar{g}_z(s)d W_s+\int_{\mathcal{E}}h_s\bar{g}_{\tilde{z}}(s)\tilde{N}(d e,d s),\\
h_t& =1,\ \ \ \ \ p^*_T=h_T\phi_x(\bar{X}_T),
\end{aligned}\right.
 \label{first-order dual equation}
\end{equation}
whose solution is the process quadruple $\left(h_\cdot, p_\cdot^*, q_\cdot^*, \tilde{q}_{(\cdot,\cdot)}^*\right)\in L_{\mathcal{F}}^2([t,T];\mathbf{R})\times L_{\mathcal{F}}^2([t,T];\mathbf{R}^n)\times L_{\mathcal{F},p}^2([t,T];\mathbf{R}^n)\times F_p^2([t,T]\times\mathcal{E};\mathbf{R}^n)$.

We also introduce the following two Hamiltonian functions $H^*:[0,T]\times \mathbf{R}^n \times \mathbf{R} \times \mathbf{R} \times \mathcal{L}^2 \times \mathbf{R}^n \times \mathbf{R} \times \mathbf{R}^n \times \mathbf{R}^n \times \mathcal{S}^n \times \mathbf{U} \rightarrow \mathbf{R}$ (see (3.25) of \cite{ZhengShi2023}) as
\begin{equation}
\begin{aligned}
&  H^*(s,x,y,z,\tilde{z};p,h,p^*,q^*,P^*,u) \\
 :=&\left\langle-h g_z p+q^*,\sigma(s,x,u)\right\rangle+\frac {1}{2}\sigma(s,x,u)^\top P^*\sigma(s,x,u)-\sigma(s,x,u)^\top P^*\bar{\sigma}(s) \\
&  +hg\left(s,x,y,z+\langle p,\sigma(s,x,u)-\bar{\sigma}(s)\rangle,\tilde{z},u\right)+\langle p^*,b(s,x,u)\rangle,
\end{aligned}
 \label{ZhengShi's Hamilton function: special}
\end{equation}
and $H:[0,T]\times \mathbf{R}^n \times \mathbf{R} \times \mathbf{R} \times \mathcal{L}^2 \times \mathbf{R}^n \times \mathbf{R}^n \times \mathcal{S}^n \times \mathbf{U}\rightarrow \mathbf{R} $ as
\begin{equation}
\begin{aligned}
&  H(s,x,y,z,\tilde{z};p, q, P, u) \\
 :=&\left\langle p,b(s,x,u) \right\rangle+\langle q,\sigma(s,x,u) \rangle+\frac {1}{2}\sigma(s,x,u)^\top P\sigma(s,x,u)-\sigma(s,x,u)^\top P\bar{\sigma}(s) \\
&  +g\left(s,x,y,z+\langle p,\sigma(s,x,u)-\bar{\sigma}(s)\rangle,\tilde{z},u\right).
\end{aligned}
 \label{WangShi's Hamilton function}
\end{equation}

\begin{Remark}
We can get the adjoint equation (\ref{first-order adjoint equation}), (\ref{first-order dual equation}), (\ref{second-order dual equation}) from (3.9), (3.21), (3.23) in \cite{ZhengShi2023}, respectively. However, there are three important differences among those adjoint equations. Firstly, all the solutions to the adjoint equations in \cite{ZhengShi2023} are one-dimensional, and we expand it to multi-dimensional case. Secondly, due to some technological difficulty, we only consider the predictable frame, not the progressive structure. So we can see that terms, like $\int_{\mathcal{E}}\bar{f}_x(s,e)\tilde{q}_{(s,e)}\nu(de)$ in (\ref{first-order adjoint equation}), are different with those, like $\int_{\mathcal{E}} \mathbb{E}\left[f_{x}(s, e) | \mathcal{P} \otimes \mathcal{B}\left(\mathcal{E}\right)\right] \tilde{q}_{(s, e)} \nu(d e)$, in (3.9) of \cite{ZhengShi2023}. As stated in the \cite{SongWu2023}, \cite{ZhengShi2023}, in fact, $\mathbb{E}$ is not an expectation (since $m$ is not a probability measure), but it owns similar properties to the expectation. And in the predictable frame, we have  $\mathbb{E}\left[f_{x}(s, e) | \mathcal{P} \otimes \mathcal{B}\left(\mathcal{E}\right)\right]={f}_x(s,e)$. Thirdly, we introduce a new second-order adjoint BSDEP (\ref{second-order adjoint equation}) involves $\left(p_\cdot,q_\cdot,\tilde{q}_{(\cdot,\cdot)}\right)$ in (\ref{first-order adjoint equation}), and a new Hamiltonian function $H$ in (\ref{WangShi's Hamilton function}). They not appear in \cite{ZhengShi2023}.
\end{Remark}

\begin{Remark}
We can also see that BSDEPs (\ref{first-order adjoint equation}), (\ref{second-order adjoint equation}) are the generalization of BSDEs (15), (16) in \cite{Hu2017} to jumps, and (\ref{first-order dual equation}) is the one of FBSDE (4.10)-(4.11) in \cite{Peng1993} to jumps. In addition, we can give some relations among these adjoint processes $p,q,\tilde{q},p^*,q^*,\tilde{q}^*,P,Q,\tilde{Q},P^*,Q^*,\tilde{Q}^*,h$ and Hamiltonian functions $H,H^*$. In fact, for all $s\in[t,T]$, we have
\begin{equation}
\left\{\begin{aligned}
p^*_s=& \ h_sp_s, \\
q^*_s=& \ h_s\bar{g}_z(s)p_s+h_sq_s, \\
\tilde{q}^*_{(s,e)}=& \ h_s\bar{g}_{\tilde{z}}(s)p_s+h_s\tilde{q}_{(s,e)}+h_s\bar{g}_{\tilde{z}}(s)\tilde{q}_{(s,e)}, \\
P^*_s=& \ h_sP_s, \\
Q^*_s=& \ h_s\bar{g}_z(s)P_s+h_sQ_s, \\
\tilde{Q}^*_{(s,e)}=& \ h_s\bar{g}_{\tilde{z}}(s)P_s+h_s\tilde{Q}_{(s,e)}+h_s\bar{g}_{\tilde{z}}(s)\tilde{Q}_{(s,e)}, \\
H^*(s,x,y,z,\tilde{z})=&\ h_sH(s,x,y,z,\tilde{z}).
\end{aligned}\right.
 \label{relation of adjoint equation}
\end{equation}
\end{Remark}

The following result is the general MP for {\bf Problem (SROCPJ)}, which is new of its own.
\begin{mythm}\label{general MP}
Let (H1)-(H3) hold, and $(t, x) \in[0, T) \times \mathbf{R}^n$ be fixed. Suppose that $\bar{u}_\cdot$ is an optimal control for {\bf Problem (SROCPJ)}, and $\left(\bar{X}^{t, x ; \bar{u}}_\cdot, \bar{Y}^{t, x ; \bar{u}}_\cdot, \bar{Z}^{t, x ; \bar{u}}_\cdot,\bar{\tilde{Z}}^{t, x ; \bar{u}}_{(\cdot,\cdot)}\right)\in L_{\mathcal{F}}^2([t,T];\mathbf{R}^n)\times L_{\mathcal{F}}^2([t,T];\mathbf{R})\times L_{\mathcal{F},p}^2([t,T];\mathbf{R})\times F_p^2([t,T]\times\mathcal{E};\mathbf{R})$ is the corresponding optimal trajectory. Let $(p_\cdot,q_\cdot,\tilde{q}_{(\cdot,\cdot)})\in L_{\mathcal{F}}^2([t,T];\mathbf{R}^n)\times L_{\mathcal{F},p}^2([t,T];\mathbf{R}^n)\times F_p^2([t,T]\times\mathcal{E};\mathbf{R}^n)$ and $(P_\cdot,Q_\cdot,\tilde{Q}_{(\cdot,\cdot)})\in L_{\mathcal{F}}^2([t,T];\\\mathcal{S}^n)\times L_{\mathcal{F},p}^2([t,T];\mathcal{S}^n)\times F_p^2([t,T]\times\mathcal{E};\mathcal{S}^n)$  be the solution to adjoint BSDEPs (\ref{first-order adjoint equation}), (\ref{second-order adjoint equation}), respectively. Then
\begin{equation}
\begin{aligned}
 &H\left(s,\bar{X}_s,\bar{Y}_s,\bar{Z}_s,\bar{\tilde{Z}}_{(s,\cdot)};p_s,q_s,P_s,u\right)- H\left(s,\bar{X}_s,\bar{Y}_s,\bar{Z}_s,\bar{\tilde{Z}}_{(s,\cdot)};p_s,q_s,P_s,\bar{u}_s\right)\geqslant 0,\\
 &\qquad\qquad\qquad\qquad \text {for all } u\in\mathbf{U},\quad \text { a.e. } s \in[t, T],\quad \mathbf{P}\text{-a.s.}.
\end{aligned}
 \label{general MP maximum condition}
\end{equation}
\end{mythm}

\begin{proof}
From their theorem 3.6 of $\cite{ZhengShi2023}$, we get
\begin{equation*}
\begin{aligned}
 &H^*\left(s,\bar{X}_s,\bar{Y}_s,\bar{Z}_s,\bar{\tilde{Z}}_{(s,\cdot)};p_s,h_s,p_s^*,q_s^*,P_s^*,u\right)- H^*\left(s,\bar{X}_s,\bar{Y}_s,\bar{Z}_s,\bar{\tilde{Z}}_{(s,\cdot)};p_s,h_s,p_s^*,q_s^*,P_s^*,\bar{u}_s\right)\geqslant 0,\\
 &\qquad\qquad\qquad\qquad \text {for all } u\in\mathbf{U},\quad \text { a.e. } s \in[t, T],\quad \mathbf{P}\text{-a.s.}.
\end{aligned}
\end{equation*}
By (\ref{relation of adjoint equation}), we can obtain that (\ref{general MP maximum condition}) holds. The proof is complete.
\end{proof}

\section{Relationship between general MP and DPP: Smooth case }

\qquad In this section, we investigate the relationship between the above general MP (Theorem \ref{general MP}) and DPP (the generalized HJB equation (\ref{HJB equation})) for {\bf Problem (SROCPJ)}, under the assumption of a smooth value function. Specifically, we obtain the connection among the value function $V$, the generalized Hamiltonian function $G$ and adjoint variables $p, q, \tilde{q},P,Q,\tilde{Q}$. Our main result is the following theorem.

\begin{mythm}\label{smooth theorem}	
Let (H1)-(H3) hold, and $(t, x) \in[0, T) \times \mathbf{R}^n$ be fixed. Suppose that $\bar{u}_\cdot$ is an optimal control for {\bf Problem (SROCPJ)}, and $\left(\bar{X}^{t, x ; \bar{u}}_\cdot, \bar{Y}^{t, x ; \bar{u}}_\cdot, \bar{Z}^{t, x ; \bar{u}}_\cdot,\bar{\tilde{Z}}^{t, x ; \bar{u}}_{(\cdot,\cdot)}\right)\in L_{\mathcal{F}}^2([t,T];\mathbf{R}^n)\times L_{\mathcal{F}}^2([t,T];\mathbf{R})\times L_{\mathcal{F},p}^2([t,T];\mathbf{R})\times F_p^2([t,T]\times\mathcal{E};\mathbf{R})$ is the corresponding optimal trajectory. If $V \in C^{1,2}\left([0, T] \times \mathbf{R}^n\right)$, then
\begin{equation}
\begin{aligned}
V_s\left(s, \bar{X}^{t, x ; \bar{u}}_s\right)& =G\left(s, \bar{X}^{t, x ; \bar{u}}_s,-V\left(s, \bar{X}^{t, x ; \bar{u}}_s\right),-V_x\left(s, \bar{X}^{t, x ; \bar{u}}_s\right),-V_{x x}\left(s, \bar{X}^{t, x ; \bar{u}}_s\right), \bar{u}_s\right) \\
& =\max _{u \in \mathbf{U}} G\left(s, \bar{X}^{t, x ; \bar{u}}_s,-V\left(s, \bar{X}^{t, x ; \bar{u}}_s\right),-V_x\left(s, \bar{X}^{t, x ; \bar{u}}_s\right),-V_{x x}\left(s, \bar{X}^{t, x ; \bar{u}}_s\right), u\right),\\
& \hspace{3cm} \text{ a.e. } s \in[t, T],\quad \mathbf{P}\text{-a.s.},
\end{aligned}
 \label{relation of V and G}
\end{equation}
where $G$ is defined as (\ref{generalized Hamiltonian function G}). Moreover, if $V \in C^{1,3}([0, T] \times$ $\left.\mathbf{R}^n\right)$ and $V_{s x}(\cdot, \cdot)$ is continuous, then
\begin{equation}
\left\{\begin{array}{lr}
p_s=-V_x\left(s, \bar{X}^{t, x ; \bar{u}}_s\right), & \forall \  s \in[t, T],\ \mathbf{P}\text{-a.s.}, \\
q_s=-V_{x x}\left(s, \bar{X}^{t, x ; \bar{u}}_s\right) {\sigma}\left(s, \bar{X}_s^{t, x ; \bar{u}}, \bar{u}_s\right), & \text { a.e. } s \in[t, T],\ \mathbf{P}\text{-a.s.},\\
\tilde{q}_{(s,e)}=-\Big[V_x\left(s, \bar{X}^{t, x ; \bar{u}}_{s-}+f(s, \bar{X}^{t, x ; \bar{u}}_{s-},\bar{u}_s,e)\right)-V_x\left(s, \bar{X}^{t, x ; \bar{u}}_{s-}\right)\Big], & \text { a.e. } s \in[t, T],\ \mathbf{P}\text{-a.s.},
\end{array}\right.
 \label{relation of p and V}
\end{equation}
where $(p_\cdot,q_\cdot,\tilde{q}_{(\cdot,\cdot)})\in L_{\mathcal{F}}^2([t,T];\mathbf{R}^n)\times L_{\mathcal{F},p}^2([t,T];\mathbf{R}^n)\times F_p^2([t,T]\times\mathcal{E};\mathbf{R}^n)$ satisfies (\ref{first-order adjoint equation}). Furthermore, if $V \in C^{1,4}\left([0, T] \times \mathbf{R}^n\right)$ and $V_{s x x}(\cdot, \cdot)$ is continuous, then
\begin{equation}
-V_{x x}\left(s, \bar{X}_s^{t, x ; \bar{u}}\right) \geqslant P_s, \quad \forall s \in[t, T],\ \mathbf{P} \text {-a.s., }
 \label{relation of P and V}
\end{equation}
where $(P_\cdot,Q_\cdot,\tilde{Q}_{(\cdot,\cdot)})\in L_{\mathcal{F}}^2([t,T];\mathcal{S}^n)\times L_{\mathcal{F},p}^2([t,T];\mathcal{S}^n)\times F_p^2([t,T]\times\mathcal{E};\mathcal{S}^n)$ satisfies (\ref{second-order adjoint equation}).
\end{mythm}

\begin{proof}
By the generalized DPP (see Theorem 3.1 of \cite{LiPeng2009}, or Theorem 1 of \cite{MoonBasar2022}), it is easy to obtain that
\begin{equation}
\begin{aligned}
V\left(s, \bar{X}_s^{t, x ; \bar{u}}\right) & =-\bar{Y}_s^{t, x ; \bar{u}}
 =-\mathbf{E}\left[\int_s^T \bar{g}(r) d r+\phi\left(\bar{X}_T^{t, x ; \bar{u}}\right) \Big| \mathcal{F}_s^t\right],\ \forall s \in[t, T],\ \mathbf{P} \text {-a.s.. }
\end{aligned}
 \label{martingale of V}
\end{equation}
In fact, because
$$
\begin{aligned}
V(t, x) & =J(t, x ; \bar{u}_\cdot)=-\bar{Y}_t^{t, x ; \bar{u}}
 =-\mathbf{E}\left\{\int_t^s \bar{g}(r) d r+\mathbf{E}\left[\int_s^T \bar{g}(r) d r+\phi\left(\bar{X}_T^{t, x ; \bar{u}}\right) \Big| \mathcal{F}_s^t\right]\right\} \\
& =-\mathbf{E} \int_t^s \bar{g}(r) d r+\mathbf{E} J\left(s, \bar{X}_s^{t, x ; \bar{u}} ; \bar{u}_\cdot \right)
 \geqslant-\mathbf{E} \int_t^s \bar{g}(r) d r+\mathbf{E} V\left(s, \bar{X}_s^{t, x ; \bar{u}}\right) \geqslant V(t, x),
\end{aligned}
$$
where the last inequality is due to the property of backward semigroup (see Theorem 3.1 of \cite{LiPeng2009} or (7) of \cite{MoonBasar2022}), noting that $t \in[0, T)$ is fixed), and all the inequalities in the aforementioned become equalities. In particular,
$$
\mathbf{E} J\left(s, \bar{X}_s^{t, x ; \bar{u}} ; \bar{u}_\cdot \right)=\mathbf{E} V\left(s, \bar{X}_s^{t, x ; \bar{u}}\right) .
$$
However, by definition $V\left(s, \bar{X}_s^{t, x ; \bar{u}}\right) \leqslant J\left(s, \bar{X}_s^{t, x ; \bar{u}} ; \bar{u}_\cdot \right)$,\ $\mathbf{P}$-a.s. Thus,
$$
V\left(s, \bar{X}_s^{t, x ; \bar{u}}\right)=J\left(s, \bar{X}_s^{t, x ; \bar{u}} ; \bar{u}_\cdot \right), \ \mathbf{P}\text{-a.s.} ,
$$
which gives ($\ref{martingale of V}$). For $s \in[t, T]$, define
$$
m_s:=-\mathbf{E}\left[\int_t^T \bar{g}(r) d r+\phi\left(\bar{X}_T^{t, x ; \bar{u}}\right) \mid \mathcal{F}_s^t\right].
$$
Clearly, $m_\cdot$ is a square integrable $\mathcal{F}_s^t$-martingale. Thus, by the martingale representation theorem (see their Lemma 2.3 of \cite{TangLi1994}), there exists unique $M_\cdot$ and $\tilde{M}_{(\cdot,\cdot)}$ such that
$$
m_s=m_t+\int_t^s M_r d W_r+ \int_t^s \int_{\mathcal{E}}\tilde{M}_{(r,e)} \tilde{N}(d e, d r).
$$
So for $s \in[t, T]$,
$$
m_s=m_T-\int_s^T M_r d W_r- \int_s^T \int_{\mathcal{E}}\tilde{M}_{(r,e)} \tilde{N}(d e,d r).
$$
Then, by ($\ref{martingale of V}$), we have for $s \in[t, T]$,
\begin{equation}
\begin{aligned}
V\left(s, \bar{X}_s^{t, x ; \bar{u}}\right) & =m_s+\int_t^s \bar{g}(r) d r
 =m_T+\int_t^s \bar{g}(r) d r-\int_s^T M_r d W_r- \int_s^T\int_{\mathcal{E}} \tilde{M}_{(r,e)} \tilde{N}(d e,d r) \\
& =-\int_s^T \bar{g}(r) d r-\int_s^T M_r d W_r- \int_s^T\int_{\mathcal{E}} \tilde{M}_{(r,e)} \tilde{N}(d e,d r)+V\left(T, \bar{X}_T^{t, x ; \bar{u}}\right).
\end{aligned}
 \label{martingale representation of V}
\end{equation}
Then, applying It\^{o}'s formula of jump diffusions (see, for example, \cite{OksendalSulem2006}) to $V\left(s, \bar{X}_s^{t, x ; \bar{u}}\right)$, we have
\begin{equation}
\begin{aligned}
dV\left(s, \bar{X}_s^{t,x;\bar{u}}\right)& =\bigg\{V_s\left(s, \bar{X}_s^{t,x; \bar{u}}\right)+\left\langle V_x\left(s, \bar{X}_s^{t,x;\bar{u}}\right), \bar{b}(s)\right\rangle+\frac{1}{2} \operatorname{tr}\left(\bar{\sigma}(s)^\top V_{x x}\left(s, \bar{X}_s^{t,x;\bar{u}} \right) \bar{\sigma}(s)\right) \\
&\quad +\int_{\mathcal{E}}\Big[V\left(s, \bar{X}_s^{t,x; \bar{u}}+f(s, \bar{X}_{s}^{t,x; \bar{u}}, \bar{u}_s, e)\right)-V\left(s, \bar{X}_s^{t,x; \bar{u}}\right) \\
&\quad -\left\langle V_{x}(s, \bar{X}_s^{t,x; \bar{u}}),f(s, \bar{X}_{s}^{t,x; \bar{u}}, \bar{u}_s, e)\right\rangle \Big]\nu(de) \bigg\}ds+\bar{\sigma}(s)^\top V_x\left(s, \bar{X}_s^{t,x;\bar{u}}\right) d W_s \\
&\quad +\int_{\mathcal{E}}\left[V\left(s, \bar{X}_{s-}^{t, x; \bar{u}}+f(s, \bar{X}_{s-}^{t,x; \bar{u}}, \bar{u}_s, e)\right)-V\left(s, \bar{X}_{s-}^{t,x; \bar{u}}\right)\right] \tilde{N}(d e, d s).
\end{aligned}
 \label{ito for V}
\end{equation}
\indent Comparing this with (\ref{martingale representation of V}), we conclude that, for all $s\in[t,T]$,
\begin{equation}
\left\{\begin{aligned}
& V_s\left(s, \bar{X}_s^{t,x; \bar{u}}\right)+\left\langle V_x\left(s, \bar{X}_s^{t,x;\bar{u}}\right), \bar{b}(s)\right\rangle+\frac{1}{2} \operatorname{tr}\left(\bar{\sigma}(s)^\top V_{x x}\left(s, \bar{X}_s^{t,x;\bar{u}} \right) \bar{\sigma}(s)\right)\\
& +\int_{\mathcal{E}}\Big[V\left(s, \bar{X}_s^{t,x; \bar{u}}+f(s, \bar{X}_{s}^{t,x; \bar{u}}, \bar{u}_s, e)\right)-V\left(s, \bar{X}_s^{t,x; \bar{u}}\right)\\
&\quad -\left\langle V_{x}(s, \bar{X}_s^{t,x; \bar{u}}),f(s, \bar{X}_{s}^{t,x; \bar{u}}, \bar{u}_s, e)\right\rangle \Big]\nu(de)\\
&=\bar{g}(s)\equiv g\left(s,\bar{X}_s^{t,x;\bar{u}},\bar{Y}_s^{t,x;\bar{u}},\bar{Z}_s^{t,x;\bar{u}},\bar{\tilde{Z}}_{(s,\cdot)}^{t,x;\bar{u}},\bar{u}_s \right),\\
& \sigma\left(s, \bar{X}_s^{t,x;\bar{u}}, \bar{u}_s\right)^\top V_x\left(s, \bar{X}_s^{t,x;\bar{u}}\right) ={M}_s,\\
& V\left(s, \bar{X}_{s-}^{t, x; \bar{u}}+f(s, \bar{X}_{s-}^{t,x; \bar{u}}, \bar{u}_s, e)\right)-V\left(s, \bar{X}_{s-}^{t,x; \bar{u}}\right)={\tilde{M}}_{(s,e)}.
\end{aligned}\right.
 \label{relationship of V and martingale}
\end{equation}

However, by the uniqueness of solution to BSDEP (\ref{BSDEP equation}), we have for all $s\in[t,T]$,
\begin{equation}
\left\{\begin{aligned}
& \bar{Y}_s^{t, x ; \bar{u}}=-V\left(s, \bar{X}_s^{t,x;\bar{u}}\right), \\
& \bar{Z}_s^{t, x ; \bar{u}}=-\sigma\left(s, \bar{X}_s^{t,x;\bar{u}}, \bar{u}_s\right)^\top V_x\left(s, \bar{X}_s^{t,x;\bar{u}}\right) , \\
& \bar{\tilde{Z}}_{(s,e)}^{t, x ; \bar{u}}=-\left[V\left(s, \bar{X}_{s-}^{t, x; \bar{u}}+f(s, \bar{X}_{s-}^{t,x; \bar{u}}, \bar{u}_s, e)\right)-V\left(s, \bar{X}_{s-}^{t,x; \bar{u}}\right)\right].
\end{aligned}\right.
 \label{relationship of V and BSDEP}
\end{equation}
It follows from (\ref{relationship of V and BSDEP}) that the first equality of (\ref{relation of V and G}) holds.

Since $V \in C^{1,2}\left([0, T] \times \mathbf{R}^n\right)$, it satisfies the generalized HJB equation (\ref{HJB equation}), which implies the second equality of ($\ref{relation of V and G}$). Also, by (\ref{HJB equation}) we have
\begin{equation}
\begin{aligned}
0=- & V_s\left(s, \bar{X}_s^{t, x ; \bar{u}}\right)+G\left(s, \bar{X}_s^{t, x ; \bar{u}},-V\left(s, \bar{X}_s^{t, x ; \bar{u}}\right),-V_x\left(s, \bar{X}_s^{t, x ; \bar{u}}\right),-V_{x x}\left(s, \bar{X}_s^{t, x ; \bar{u}}\right), \bar{u}_s\right) \\
\geqslant- & V_s(s, x)+G\left(s, x,-V(s, x),-V_x(s, x),-V_{x x}(s, x), \bar{u}_s\right), \quad \forall x \in \mathbf{R}^n .
\end{aligned}
 \label{by HJB equation}
\end{equation}
Consequently, if $V \in C^{1,3}\left([0, T] \times \mathbf{R}^n\right)$ , then we have
\begin{equation}
\begin{aligned}
\left.\frac{\partial}{\partial x} \bigg\{-V_s(s, x)+G\left(s, x,-V(s, x),-V_x(s, x),-V_{x x}(s, x), \bar{u}_s\right) \bigg\}\right|_{x=\bar{X}_s^{t, x ; \bar{u}}}=0, \ \forall s \in[t, T],
\end{aligned}
 \label{first-order maximum condition}
\end{equation}
which is the first-order maximum condition. Furthermore, if $V \in C^{1,4}([0, T] \times$ $\left.\mathbf{R}^n\right)$, the following second-order maximum condition holds:
\begin{equation}
\left.\frac{\partial^2}{\partial x^2}\bigg\{-V_s(s, x)+G\left(s, x,-V(s, x),-V_x(s, x),-V_{x x}(s, x), \bar{u}_s\right)\bigg\}\right|_{x=\bar{X}^{t, x ; \bar{u}}_s} \leqslant 0, \ \forall s \in[t, T].
 \label{second-order maximum condition}
\end{equation}
On the one hand, (\ref{first-order maximum condition}) yields that (recall (\ref{generalized Hamiltonian function G})), for all $s \in[t, T]$,
\begin{equation}
\begin{aligned}
0=& -V_{s x}\left(s, \bar{X}_s^{t, x ; \bar{u}}\right)-V_{x x}\left(s, \bar{X}_s^{t, x ; \bar{u}}\right) \bar{b}(s)-\bar{b}_x(s)^\top V_x\left(s, \bar{X}_s^{t, x ; \bar{u}}\right) \\
& -\frac{1}{2} \operatorname{tr}\left(\bar{\sigma}(s)^\top V_{x x x}\left(s, \bar{X}_s^{t, x ; \bar{u}}\right)\bar{\sigma}(s))\right)
 -\bar{\sigma}_x(s)^\top V_{x x}\left(s, \bar{X}_s^{t, x ; \bar{u}}\right) \bar{\sigma}(s)+\bar{g}_x(s) \\
& -\bar{g}_y(s) V_x\left(s, \bar{X}_s^{t, x ; \bar{u}}\right)-\bar{g}_z(s){\left[V_{x x}\left(s, \bar{X}_s^{t, x ; \bar{u}}\right) \bar{\sigma}(s)
 +\bar{\sigma}_x(s)^\top V_x\left(s, \bar{X}_s^{t, x ; \bar{u}}\right) \right] } \\
& -\bar{g}_{\tilde{z}}(s)\int_{\mathcal{E}}\Big[f_x(s,\bar{X}_{s}^{t,x; \bar{u}}, \bar{u}_s, e)^\top V_x\left(s, \bar{X}_s^{t,x; \bar{u}}+f(s, \bar{X}_{s}^{t,x; \bar{u}}, \bar{u}_s, e)\right) \\
& +V_x\left(s, \bar{X}_s^{t,x; \bar{u}}+f(s, \bar{X}_{s}^{t,x; \bar{u}}, \bar{u}_s, e)\right)-V_x\left(s, \bar{X}_{s}^{t,x; \bar{u}}\right)\Big]\nu(de) \\
& -\int_{\mathcal{E}}\Big\{V_x\left(s, \bar{X}_s^{t, x; \bar{u}}+f(s, \bar{X}_{s}^{t,x; \bar{u}}, \bar{u}_s, e)\right)-V_x\left(s, \bar{X}_s^{t,x; \bar{u}}\right)\\
& -V_{x x}\left(s, \bar{X}_s^{t,x; \bar{u}}\right)f(s, \bar{X}_{s}^{t,x; \bar{u}}, \bar{u}_s, e) \\
& +f_x(s, \bar{X}_{s}^{t,x; \bar{u}}, \bar{u}_s, e)^\top \Big[V_x\left(s, \bar{X}_s^{t, x; \bar{u}}+f(s, \bar{X}_{s}^{t,x; \bar{u}}, \bar{u}_s, e)\right)-V_x\left(s, \bar{X}_s^{t,x; \bar{u}}\right)\Big]\Big\}\nu(de).
\end{aligned}
 \label{eq:3.12}
\end{equation}
On the other hand, applying It\^{o}'s formula to $V_x\left(s, \bar{X}_s^{t, x ; \bar{u}}\right)$, we get
\begin{equation}
\begin{aligned}
&dV_x\left(s, \bar{X}_s^{t,x;\bar{u}}\right) =\bigg\{V_{x s}\left(s, \bar{X}_s^{t,x; \bar{u}}\right)+ V_{x x}\left(s, \bar{X}_s^{t,x;\bar{u}}\right) \bar{b}(s)
 +\frac{1}{2} \operatorname{tr}\left(\bar{\sigma}(s)^\top V_{x x x}\left(s, \bar{X}_s^{t,x;\bar{u}} \right) \bar{\sigma}(s)\right)\bigg. \\
&\quad +\int_{\mathcal{E}}\Big[V_x\left(s, \bar{X}_s^{t,x; \bar{u}}+f(s, \bar{X}_{s}^{t,x; \bar{u}}, \bar{u}_s, e)\right)-V_x\left(s, \bar{X}_s^{t,x; \bar{u}}\right) \\
&\quad -V_{x x}(s, \bar{X}_s^{t,x; \bar{u}}) f(s, \bar{X}_{s}^{t,x; \bar{u}}, \bar{u}_s, e)\Big]\nu(de) \bigg\}ds +V_{x x}\left(s, \bar{X}_s^{t,x;\bar{u}}\right) \sigma\left(s, \bar{X}_s^{t,x;\bar{u}}, \bar{u}_s\right) d W_s \\
&\quad +\int_{\mathcal{E}}\left[V_x\left(s, \bar{X}_{s-}^{t, x; \bar{u}}+f(s, \bar{X}_{s-}^{t,x; \bar{u}}, \bar{u}_s, e)\right)-V_x\left(s, \bar{X}_{s-}^{t,x; \bar{u}}\right)\right] \tilde{N}(d e, d s). \\
\end{aligned}
 \label{eq:3.13}
\end{equation}
Note that (\ref{eq:3.12}), we get
\begin{equation}
\begin{aligned}
&-dV_x\left(s, \bar{X}_s^{t,x;\bar{u}_s}\right) =\bigg\{\bar{b}_x(s)^\top V_x\left(s, \bar{X}_s^{t, x ; \bar{u}}\right)+\bar{\sigma}_x(s)^\top V_{x x}\left(s, \bar{X}_s^{t, x ; \bar{u}}\right) \bar{\sigma}(s)-\bar{g}_x(s) \\
&\quad +\bar{g}_y(s) V_x\left(s, \bar{X}_s^{t, x ; \bar{u}}\right)+\bar{g}_z(s)\left[V_{x x}\left(s, \bar{X}_s^{t, x ; \bar{u}}\right) \bar{\sigma}(s)
 +\bar{\sigma}_x(s)^\top V_x\left(s, \bar{X}_s^{t, x ; \bar{u}}\right) \right] \\
&\quad +\bar{g}_{\tilde{z}}(s)\int_{\mathcal{E}}\Big[V_x\left(s, \bar{X}_s^{t, x; \bar{u}}+f(s, \bar{X}_{s}^{t,x; \bar{u}}, \bar{u}_s, e)\right)-V_x\left(s, \bar{X}_s^{t,x; \bar{u}}\right) \\
&\quad +f_x(s, \bar{X}_{s}^{t,x; \bar{u}}, \bar{u}_s, e)^\top V_x\left(s, \bar{X}_s^{t, x; \bar{u}}+f(s, \bar{X}_{s}^{t,x; \bar{u}}, \bar{u}_s, e)\right)\Big]\nu(de)\bigg. \\
&\quad +\int_{\mathcal{E}}f_x(s, \bar{X}_{s}^{t,x; \bar{u}}, \bar{u}_s, e)^\top \left[V_x\left(s, \bar{X}_s^{t, x; \bar{u}}+f(s, \bar{X}_{s}^{t,x; \bar{u}}, \bar{u}_s, e)\right)
 -V_x\left(s, \bar{X}_s^{t,x; \bar{u}}\right)\right]\nu(de)\bigg\}ds \\
&\quad +V_{x x}\left(s, \bar{X}_s^{t,x;\bar{u}}\right) \sigma\left(s, \bar{X}_s^{t,x;\bar{u}}, \bar{u}_s\right) d W_s \\
&\quad +\int_{\mathcal{E}}\left[V_x\left(s, \bar{X}_{s-}^{t, x; \bar{u}}+f(s, \bar{X}_{s-}^{t,x; \bar{u}}, \bar{u}_s, e)\right)-V_x\left(s, \bar{X}_{s-}^{t,x; \bar{u}}\right)\right] \tilde{N}(d e, d s).
\end{aligned}
 \label{eq:3.14}
\end{equation}
Note that $V$ solves (\ref{HJB equation}), and thus $V_x\left(T, \bar{X}^{t, x ; \bar{u}}_T\right)=-\phi_x\left(\bar{X}^{t, x ; \bar{u}}_T\right)$. Then by the uniqueness of the solutions to (\ref{first-order adjoint equation}), we obtain (\ref{relation of p and V}).

Moreover, (\ref{second-order maximum condition}) yields that, for any $s \in[t, T]$,
\begin{equation*}
\begin{aligned}
& -V_{s x x}\left(s, \bar{X}_s^{t, x ; \bar{u}}\right)-V_{x x x}\left(s, \bar{X}_s^{t, x ; \bar{u}}\right) \bar{b}(s)-V_{x x}\left(s, \bar{X}_s^{t, x ; \bar{u}}\right) \bar{b}_x(s)
 -\bar{b}_{x x}(s)^\top V_x\left(s, \bar{X}_s^{t, x ; \bar{u}}\right) \\
& -\bar{b}_{x}(s)^\top V_{x x}\left(s, \bar{X}_s^{t, x ; \bar{u}}\right)-\frac{1}{2} \operatorname{tr}\left(\bar{\sigma}(s)^\top V_{x x x x}\left(s, \bar{X}_s^{t, x ; \bar{u}}\right)\bar{\sigma}(s)\right) \\
& -2\bar{\sigma}_x(s)^\top V_{x x x}\left(s, \bar{X}_s^{t, x ; \bar{u}}\right) \bar{\sigma}(s)-\bar{\sigma}_{x x}(s)^\top V_{x x}\left(s, \bar{X}_s^{t, x ; \bar{u}}\right) \bar{\sigma}(s)
 -\bar{\sigma}_x(s)^\top V_{x x}\left(s, \bar{X}_s^{t, x ; \bar{u}}\right) \bar{\sigma}_x(s) \\
& -\bar{g}_y(s) V_{x x}\left(s, \bar{X}_s^{t, x ; \bar{u}}\right)-\bar{g}_z(s)\bar{\sigma}_x(s)^\top V_{x x}\left(s, \bar{X}_s^{t, x ; \bar{u}}\right)
 -\bar{g}_z(s)V_{x x}\left(s, \bar{X}_s^{t, x ; \bar{u}}\right)\bar{\sigma}_x(s)\\
& -\bar{g}_z(s)\bar{\sigma}_{x x}(s)^\top V_x\left(s, \bar{X}_s^{t, x ; \bar{u}}\right)-\bar{g}_z(s)V_{x x x}\left(s, \bar{X}_s^{t, x ; \bar{u}}\right)\bar{\sigma}(s) \\
& +\bigg[I_{n\times n},-V_x\left(s, \bar{X}^{t, x ; \bar{u}}_s\right),-\bar{\sigma}_x(s)^\top V_x\left(s, \bar{X}^{t, x ; \bar{u}}_s\right)-V_{x x}\left(s, \bar{X}^{t, x ; \bar{u}}_s\right) \bar{\sigma}(s), \\
&\quad \int_{\mathcal{E}}\left[-f_x(s, \bar{X}_{s}^{t,x; \bar{u}}, \bar{u}_s, e)^\top V_x\left(s, \bar{X}^{t, x ; \bar{u}}_s\right)
 -\big[V_x\left(s, \bar{X}^{t, x ; \bar{u}}_s+f(s, \bar{X}^{t, x ; \bar{u}}_{s},\bar{u}_s,e)\right)-V_x\left(s, \bar{X}^{t, x ; \bar{u}}_s\right)\big]\right. \\
&\quad \left.-f_x(s, \bar{X}_{s}^{t,x; \bar{u}}, \bar{u}_s, e)^\top \big[V_x\left(s, \bar{X}^{t, x ; \bar{u}}_s+f(s, \bar{X}^{t, x ; \bar{u}}_{s},\bar{u}_s,e)\right)
 -V_x\left(s, \bar{X}^{t, x ; \bar{u}}_s\right)\big]\right]\nu(de)\bigg]D^{2}\bar{g}(s) \\
\end{aligned}
\end{equation*}
\begin{equation}
\begin{aligned}
&\quad \cdot\bigg[I_{n\times n},-V_x\left(s, \bar{X}^{t, x ; \bar{u}}_s\right),-\bar{\sigma}_x(s)^\top V_x\left(s, \bar{X}^{t, x ; \bar{u}}_s\right)-V_{x x}\left(s, \bar{X}^{t, x ; \bar{u}}_s\right) \bar{\sigma}(s), \\
&\quad \int_{\mathcal{E}}\left[-f_x(s, \bar{X}_{s}^{t,x; \bar{u}}, \bar{u}_s, e)^\top V_x\left(s, \bar{X}^{t, x ; \bar{u}}_s\right)
 -\big[V_x\left(s, \bar{X}^{t, x ; \bar{u}}_s+f(s, \bar{X}^{t, x ; \bar{u}}_{s},\bar{u}_s,e)\right)-V_x\left(s, \bar{X}^{t, x ; \bar{u}}_s\right)\big]\right. \\
&\quad \left.-f_x(s, \bar{X}_{s}^{t,x; \bar{u}}, \bar{u}_s, e)^\top \big[V_x\left(s, \bar{X}^{t, x ; \bar{u}}_s+f(s, \bar{X}^{t, x ; \bar{u}}_{s},\bar{u}_s,e)\right)
 -V_x\left(s, \bar{X}^{t, x ; \bar{u}}_s\right)\big]\right]\nu(de)\bigg]^\top \\
& -\int_{\mathcal{E}}\bigg\{\bar{g}_{\tilde{z}}(s) f_{x x}(s, \bar{X}_{s}^{t,x; \bar{u}}, \bar{u}_s, e)^\top V_x\left(s, \bar{X}^{t, x ; \bar{u}}_s+f(s, \bar{X}^{t, x ; \bar{u}}_{s},\bar{u}_s,e)\right) \\
& +\bar{g}_{\tilde{z}}(s) f_x(s, \bar{X}_{s}^{t,x; \bar{u}}, \bar{u}_s, e)^\top V_{x x}\left(s, \bar{X}^{t, x ; \bar{u}}_s+f(s, \bar{X}^{t, x ; \bar{u}}_{s},\bar{u}_s,e)\right) \\
& +\bar{g}_{\tilde{z}}(s)  V_{x x}\left(s, \bar{X}^{t, x ; \bar{u}}_s+f(s, \bar{X}^{t, x ; \bar{u}}_{s},\bar{u}_s,e)\right)f_x(s, \bar{X}_{s}^{t,x; \bar{u}}, \bar{u}_s, e) \\
& +\bar{g}_{\tilde{z}}(s) f_x(s, \bar{X}_{s}^{t,x; \bar{u}}, \bar{u}_s, e)^\top V_{x x}\left(s, \bar{X}^{t, x ; \bar{u}}_s
 +f(s, \bar{X}^{t, x ; \bar{u}}_{s},\bar{u}_s,e)\right) f_x(s, \bar{X}_{s}^{t,x; \bar{u}}, \bar{u}_s, e) \\
& +\bar{g}_{\tilde{z}}(s)\left[V_{x x}\left(s, \bar{X}_s^{t, x; \bar{u}}+f(s, \bar{X}_{s}^{t,x; \bar{u}}, \bar{u}_s, e)\right)-V_{x x}\left(s, \bar{X}_s^{t,x; \bar{u}}\right)\right] \\
& +V_{x x}\left(s, \bar{X}_s^{t, x; \bar{u}}+f(s, \bar{X}_{s}^{t,x; \bar{u}}, \bar{u}_s, e)\right)-V_{x x}\left(s, \bar{X}_s^{t,x; \bar{u}}\right)
 -V_{x x x}\left(s, \bar{X}_s^{t,x; \bar{u}}\right)f(s, \bar{X}_{s}^{t,x; \bar{u}}, \bar{u}_s, e) \\
& +f_{x x}(s, \bar{X}_{s}^{t,x; \bar{u}}, \bar{u}_s, e)^\top \left[V_x\left(s, \bar{X}_s^{t, x; \bar{u}}+f(s, \bar{X}_{s}^{t,x; \bar{u}}, \bar{u}_s, e)\right)-V_x\left(s, \bar{X}_s^{t,x; \bar{u}}\right)\right] \\
& +f_x(s, \bar{X}_{s}^{t,x; \bar{u}}, \bar{u}_s, e)^\top \left[V_{x x}\left(s, \bar{X}_s^{t, x; \bar{u}}+f(s, \bar{X}_{s}^{t,x; \bar{u}}, \bar{u}_s, e)\right)-V_{x x}\left(s, \bar{X}_s^{t,x; \bar{u}}\right)\right] \\
& +\left[V_{x x}\left(s, \bar{X}_s^{t, x; \bar{u}}+f(s, \bar{X}_{s}^{t,x; \bar{u}}, \bar{u}_s, e)\right)-V_{x x}\left(s, \bar{X}_s^{t,x; \bar{u}}\right)\right]f_x(s, \bar{X}_{s}^{t,x; \bar{u}}, \bar{u}_s, e) \\
& +f_x(s, \bar{X}_{s}^{t,x; \bar{u}}, \bar{u}_s, e)^\top V_{x x}\left(s, \bar{X}_s^{t, x; \bar{u}}+f(s, \bar{X}_{s}^{t,x; \bar{u}}, \bar{u}_s, e)\right)f_x(s, \bar{X}_{s}^{t,x; \bar{u}}, \bar{u}_s, e)\bigg\}\nu(de)\leq0.
\end{aligned}
 \label{eq:3.15}
\end{equation}
In the above and what follows, the notation of partial derivatives has its own definitions which we will not clarify on by one, because of limited space. (For simplicity, we can verify the calculus just using $n=1$, i.e., $x$ is one-dimensional.)

Applying It\^{o}'s formula again, to $V_{x x}\left(s, \bar{X}_s^{t, x ; \bar{u}}\right)$, we obtain
\begin{equation}
\begin{aligned}
&dV_{x x}\left(s, \bar{X}_s^{t,x;\bar{u}}\right) =\bigg\{V_{xx  s}\left(s, \bar{X}_s^{t,x; \bar{u}}\right)+ V_{x x x}\left(s, \bar{X}_s^{t,x;\bar{u}}\right) \bar{b}(s)
 +\frac{1}{2} \operatorname{tr}\left(\bar{\sigma}(s)^\top V_{x x x x}\left(s, \bar{X}_s^{t,x;\bar{u}} \right) \bar{\sigma}(s)\right) \\
&\quad +\int_{\mathcal{E}}\Big[V_{x x}\left(s, \bar{X}_s^{t,x; \bar{u}}+f(s, \bar{X}_{s}^{t,x; \bar{u}}, \bar{u}_s, e)\right)-V_{x x}\left(s, \bar{X}_s^{t,x; \bar{u}}\right) \\
&\quad -V_{x x x}(s, \bar{X}_s^{t,x; \bar{u}}) f(s, \bar{X}_{s}^{t,x; \bar{u}}, \bar{u}_s, e)\Big]\nu(de) \bigg\}ds
 +V_{x x x}\left(s, \bar{X}_s^{t,x;\bar{u}}\right) \sigma\left(s, \bar{X}_s^{t,x;\bar{u}}, \bar{u}_s\right) d W_s \\
&\quad +\int_{\mathcal{E}}\left[V_{x x}\left(s, \bar{X}_{s-}^{t, x; \bar{u}}+f(s, \bar{X}_{s-}^{t,x; \bar{u}}, \bar{u}_s, e)\right)-V_{x x}\left(s, \bar{X}_{s-}^{t,x; \bar{u}}\right)\right] \tilde{N}(d e, d s). \\
\end{aligned}
 \label{eq:3.16}
\end{equation}
For all $s\in[t,T]$, define
$$
\begin{aligned}
& \mathcal{P}_s:=-V_{x x}\left(s, \bar{X}_s^{t, x ; \bar{u}}\right), \\
& \mathcal{Q}_s:=-V_{x x x}\left(s, \bar{X}_s^{t, x ; \bar{u}}\right) \bar{\sigma}(s), \\
& \tilde{\mathcal{Q}}_{(s,e)}:=-\left[V_{x x}\left(s, \bar{X}_{s-}^{t, x; \bar{u}}+f(s, \bar{X}_{s-}^{t,x; \bar{u}}, \bar{u}_s, e)\right)-V_{x x}\left(s, \bar{X}_{s-}^{t,x; \bar{u}}\right)\right],
\end{aligned}
$$
and we will have
\begin{equation}
\begin{aligned}
-d \mathcal{P}_s=& \bigg\{V_{xx  s}\left(s, \bar{X}_s^{t,x; \bar{u}}\right)+ V_{x x x}\left(s, \bar{X}_s^{t,x;\bar{u}}\right) \bar{b}(s)
 +\frac{1}{2} \operatorname{tr}\left(\bar{\sigma}(s)^\top V_{x x x x}\left(s, \bar{X}_s^{t,x;\bar{u}} \right) \bar{\sigma}(s)\right)\bigg. \\
& +\int_{\mathcal{E}}\Big[V_{x x}\left(s, \bar{X}_s^{t,x; \bar{u}}+f(s, \bar{X}_{s}^{t,x; \bar{u}}, \bar{u}_s, e)\right)-V_{x x}\left(s, \bar{X}_s^{t,x; \bar{u}}\right) \\
& -V_{x x x}(s, \bar{X}_s^{t,x; \bar{u}}) f(s, \bar{X}_{s}^{t,x; \bar{u}}, \bar{u}_s, e)\Big]\nu(de) \bigg\}ds
 -\mathcal{Q}_s d W_s -\int_{\mathcal{E}}\tilde{\mathcal{Q}}_{(s,e)} \tilde{N}(d e, d s).
\end{aligned}
 \label{BSDEP with P}
\end{equation}

From ($\ref{eq:3.15}$) and the continuity of $V_{s x x}(\cdot, \cdot)$, as well as (\ref{relation of p and V}), we have
\begin{equation}
\begin{aligned}
&\quad V_{s x x}\left(s, \bar{X}_s^{t, x ; \bar{u}}\right)+V_{x x x}\left(s, \bar{X}_s^{t, x ; \bar{u}}\right) \bar{b}(s)
 +\frac{1}{2} \operatorname{tr}\left(\bar{\sigma}(s)^\top V_{x x x x}\left(s, \bar{X}_s^{t, x ; \bar{u}}\right)\bar{\sigma}(s)\right) \\
&\quad +\int_{\mathcal{E}}\Big\{V_{x x}\left(s, \bar{X}_s^{t, x; \bar{u}}+f(s, \bar{X}_{s}^{t,x; \bar{u}}, \bar{u}_s, e\right)-V_{x x}\left(s, \bar{X}_s^{t,x; \bar{u}}\right)\\
&\quad -V_{x x x}\left(s, \bar{X}_s^{t,x; \bar{u}}\right)f(s, \bar{X}_{s}^{t,x; \bar{u}}, \bar{u}_s, e)\Big\}\nu(de) \\
& \geqslant \mathcal{P}_s \bar{b}_x(s)+\bar{b}_x(s)^\top \mathcal{P}_s+\bar{b}_{x x}(s)^\top p_s+\bar{\sigma}_x(s)^\top \mathcal{Q}_s+\mathcal{Q}_s \bar{\sigma}_x(s)+\bar{\sigma}_{x x}(s)^\top q_s \\
&\quad +\bar{\sigma}_x(s)^\top \mathcal{P}_s\bar{\sigma}_x(s)+\bar{g}_y(s)\mathcal{P}_s+\bar{g}_z(s)\bar{\sigma}_x(s)^\top \mathcal{P}_s+\bar{g}_z(s)\mathcal{P}_s\bar{\sigma}_x(s) \\
&\quad +\bar{g}_z(s)\bar{\sigma}_{x x}(s)^\top p_s+\bar{g}_z(s)\mathcal{Q}_s+\Psi(s) D^2\bar{g}(s) \Psi(s)^\top +\int_{\mathcal{E}}\Big\{\bar{f}_x(s,e)^\top \tilde{\mathcal{Q}}_{(s,e)}\\
&\quad +\tilde{\mathcal{Q}}_{(s,e)}\bar{f}_x(s,e)+\bar{f}_x(s,e)^\top \tilde{\mathcal{Q}}_{(s,e)}\bar{f}_x(s,e)+\bar{f}_x(s,e)^\top \mathcal{P}_s\bar{f}_x(s,e)+\bar{f}_{x x}(s,e)^\top \tilde{q}_{(s,e)} \\
&\quad +\bar{f}_{x x}(s,e)^\top \bar{g}_{\tilde{z}}(s) p_s+\bar{f}_{x x}(s,e)^\top \bar{g}_{\tilde{z}}(s) \tilde{q}_{(s,e)}+\bar{g}_{\tilde{z}}(s)\bar{f}_x(s,e)^\top \mathcal{P}_s+\bar{g}_{\tilde{z}}(s)\mathcal{P}_s\bar{f}_x(s,e) \\
&\quad +\bar{g}_{\tilde{z}}(s)\bar{f}_x(s,e)^\top \tilde{\mathcal{Q}}_{(s,e)}+\bar{g}_{\tilde{z}}(s)\tilde{\mathcal{Q}}_{(s,e)}\bar{f}_x(s,e)+\bar{g}_{\tilde{z}}(s)\bar{f}_x(s,e)^\top \tilde{\mathcal{Q}}_{(s,e)}\bar{f}_x(s,e) \\
&\quad +\bar{g}_{\tilde{z}}(s)\bar{f}_x(s,e)^\top\mathcal{P}_s\bar{f}_x(s,e)+\bar{g}_{\tilde{z}}(s)\tilde{\mathcal{Q}}_{(s,e)}\Big\}\nu(de).
\end{aligned}
 \label{eq:3.18}
\end{equation}
Note that $P_T=\mathcal{P}_T=-V_{x x}\left(T, \bar{X}_T^{t, x ; \bar{u}}\right)=\phi_{x x}\left(\bar{X}_T^{t, x ; \bar{u}}\right)$. Using the above relation, one can applying the comparison theorem (see Zhu \cite{Zhu2010}) for the matrix-valued BSDEPs (\ref{second-order adjoint equation}) and (\ref{BSDEP with P}), and thus we obtain (\ref{relation of P and V}). The proof is complete.
\end{proof}

\begin{Remark}
Theorem \ref{smooth theorem} can cover the existing results in some special cases. If we do not contain the random jumps, i.e., $f\equiv0$, Theorem \ref{smooth theorem} reduces to Corollary 3.4 of \cite{NieShiWu2017}. Moreover, by  (\ref{relation of adjoint equation}), we can fill the gap in Remark 3.1 of \cite{Shi2014}.
\end{Remark}

\section{Relationship between general MP and DPP: Nonsmooth case}

\qquad In this section, we investigate the relationship between the above general MP and DPP for {\bf Problem (SROCPJ)}, where the value function is not smooth enough. To make this paper self-contained, we will present the definitions of viscosity solutions and semi-jets. At first, we give some lemmas.

\subsection{Some preliminary results}

\begin{mylem}\label{Lemma 4.1}
(\cite{ShiWu2011}) Let (H1) hold. For all $k=2,4$, there exists $C>0$ such that for any $0 \leqslant t, \hat{t} \leqslant T, x, \hat{x} \in \mathbf{R}^n, u_\cdot \in \mathcal{U}^w[s, T]$,
$$
\begin{aligned}
\mathbf{E}\left|X^{t, x; u}_s\right|^k & \leqslant C\left(1+|x|^k\right), \quad s \in[t, T], \\
\mathbf{E}\left|X^{t, x; u}_s-x\right|^k & \leqslant C\left(1+|x|^k\right)(T-t)^{\frac{k}{2}}, \quad s \in[t, T], \\
\mathbf{E}\left[\sup _{t \leqslant s \leqslant T}\left|X^{t, x; u}_s-x\right|\right]^k & \leqslant C\left(1+|x|^k\right)(T-t)^{\frac{k}{2}}, \\
\mathbf{E}\left|X^{t, x; u}_s-X^{t, \hat{x}; u}_s\right|^k & \leqslant C|x-\hat{x}|^k, \quad s \in[t, T], \\
\mathbf{E}\left|X^{t, x; u}_s-X^{\hat{t}, x; u}_s\right|^k & \leqslant C(1+|x|)|t-\hat{t}|^{\frac{k}{2}}, \quad s \in[t \vee \hat{t}, T].
\end{aligned}
$$
\end{mylem}

\begin{mylem}\label{Lemma 4.2}
Let (H1) hold. For all $k=2,4$, there exists $C>0$ such that $0 \leq t \leq T, x, \hat{x}\in \mathbf{R}^n, u_\cdot \in \mathcal{U}^w[s, T]$,
\begin{equation}
\mathbf{E}\left[\sup _{t \leqslant s \leqslant T}\left|X^{t, x; u}_s-X^{t, \hat{x}; u}_s\right|^k\right] \leqslant C|x-\hat{x}|^k, \quad s \in[t, T].
 \label{estimate of sup X-X}
\end{equation}
\end{mylem}

\begin{mycor}\label{Corollary 4.1}
Let (H1) hold. For all $k=2^i, i=1,2,\cdots$, Lemma \ref{Lemma 4.1} and Lemma \ref{Lemma 4.2} also hold.
\end{mycor}

\begin{proof}
We only prove (\ref{estimate of sup X-X}) also holds. In fact, since for $x, \hat{x} \in \mathbf{R}^n$,
$$
\begin{aligned}
X^{t, x; u}_s-X^{t, \hat{x}; u}_s &= x-\hat{x} + \int_t^s\left(b(\alpha,X_\alpha^{t,x;u},u_\alpha)-b(\alpha,X_\alpha^{t,\hat{x};u},u_\alpha )\right)d\alpha \\
&\quad +\int_t^s\left(\sigma(\alpha,X_\alpha^{t,x;u},u_\alpha)-\sigma(\alpha,X_\alpha^{t,\hat{x};u},u_\alpha )\right)dW_\alpha \\
&\quad +\int_t^s \int_\mathcal{E}\left(f(\alpha,X_{\alpha-}^{t,x;u},u_\alpha,e)-f(\alpha,X_{\alpha-}^{t,\hat{x};u},u_\alpha,e)\right)\tilde{N}(d e,d\alpha),
\end{aligned}
$$
we have
$$
\begin{aligned}
\sup _{t \leqslant s \leqslant T}\left|X^{t, x; u}_s-X^{t, \hat{x}; u}_s\right|^k
&\leqslant C\bigg[|x-\hat{x}|^k + \sup _{t \leqslant s \leqslant T}\left|\int_t^s\left|b(\alpha,X_\alpha^{t,x;u},u_\alpha)-b(\alpha,X_\alpha^{t,\hat{x};u},u_\alpha )\right|d\alpha\right|^k \\
&\quad +\sup _{t \leqslant s \leqslant T}\left|\int_t^s\left|\sigma(\alpha,X_\alpha^{t,x;u},u_\alpha)-\sigma(\alpha,X_\alpha^{t,\hat{x};u},u_\alpha )\right|dW_\alpha\right|^k \\
&\quad +\sup _{t \leqslant s \leqslant T}\left|\int_t^s \int_\mathcal{E}\left|f(\alpha,X_{\alpha-}^{t,x;u},u_\alpha,e)-f(\alpha,X_{\alpha-}^{t,\hat{x};u},u_\alpha,e)\right|\tilde{N}(d e,d\alpha)\right|^k\bigg].
\end{aligned}
$$
We denote
$$
\begin{aligned}
& \Delta b(\alpha):=b(\alpha,X_\alpha^{t,x;u},u_\alpha)-b(\alpha,X_\alpha^{t,\hat{x};u},u_\alpha ), \\
& \Delta \sigma(\alpha):=\sigma(\alpha,X_\alpha^{t,x;u},u_\alpha)-\sigma(\alpha,X_\alpha^{t,\hat{x};u},u_\alpha ), \\
& \Delta f(\alpha,e):=f(\alpha,X_\alpha^{t,x;u},u_\alpha,e)-f(\alpha,X_\alpha^{t,\hat{x};u},u_\alpha,e ).
\end{aligned}
$$
By B-D-G inequality, H\"{o}lder's inequality, Fubini's theorem and (H1), for all $k=2^i, i=1,2,\cdots$, we have
$$
\begin{aligned}
&\quad \mathbf{E}\sup _{t \leqslant s \leqslant T}\left|X^{t, x; u}_s-X^{t, \hat{x}; u}_s\right|^k \\
&\leqslant C\bigg[|x-\hat{x}|^k + \mathbf{E}\sup _{t \leq s \leq T}\left|\int_t^s\left| \Delta b(\alpha)\right|d\alpha\right|^k
 + \mathbf{E}\sup _{t \leq s \leq T}\left|\int_t^s\left| \Delta \sigma(\alpha)\right|dW_\alpha\right|^k \\
&\quad + \mathbf{E}\sup _{t \leq s \leq T}\left|\int_t^s \int_\mathcal{E}\left| \Delta f(\alpha,e)\right|\tilde{N}(d e,d\alpha)\right|^k\bigg] \\
&\leqslant C\bigg[|x-\hat{x}|^k + \mathbf{E}\sup _{t \leq s \leq T}\left|\int_t^s\left| \Delta b(\alpha)\right|d\alpha\right|^k + \mathbf{E}\left|\int_t^T\left| \Delta \sigma(\alpha)\right|^2d\alpha\right|^{\frac{k}{2}} \\
&\quad + \mathbf{E}\left|\int_t^T \int_\mathcal{E}\left| \Delta f(\alpha,e)\right|^2{N}(d e,d\alpha)\right|^{\frac{k}{2}}\bigg] \\
&\leqslant C\bigg[|x-\hat{x}|^k + \mathbf{E}\sup _{t \leq s \leq T}\left|\int_t^s\left| \Delta b(\alpha)\right|d\alpha\right|^k + \mathbf{E}\left|\int_t^T\left| \Delta \sigma(\alpha)\right|^2d\alpha\right|^{\frac{k}{2}} \\
&\quad + \mathbf{E}\left|\int_t^T \left\| \Delta f(\alpha,e)\right\|_{\mathcal{L}^2}^2d\alpha\right|^{\frac{k}{2}}\bigg]
 + \mathbf{E}\left|\int_t^T \int_\mathcal{E}\left| \Delta f(\alpha,e)\right|^2\tilde{N}(d e,d\alpha)\right|^{\frac{k}{2}}\bigg] \\
&\leqslant C\bigg[|x-\hat{x}|^k + \mathbf{E}\sup _{t \leq s \leq T}\left|\int_t^s\left|X^{t, x; u}_\alpha-X^{t, \hat{x}; u}_\alpha\right|d\alpha\right|^k \\
&\quad +\mathbf{E}\left|\int_t^T\left|X^{t, x; u}_\alpha-X^{t, \hat{x}; u}_\alpha\right|^2d\alpha\right|^{\frac{k}{2}}
 + \mathbf{E}\sup _{t \leq s \leq T}\left|\int_t^s \int_\mathcal{E}\left| \Delta f(\alpha,e)\right|^2\tilde{N}(d e,d\alpha)\right|^{\frac{k}{2}}\bigg] \\
&\leqslant C\bigg[|x-\hat{x}|^k + \mathbf{E}\sup _{t \leq s \leq T}\int_t^s\left|X^{t, x; u}_\alpha-X^{t, \hat{x}; u}_\alpha\right|^kd\alpha+\mathbf{E}\int_t^T\left|X^{t, x; u}_\alpha-X^{t, \hat{x}; u}_\alpha\right|^kd\alpha \\
&\quad + \mathbf{E}\left|\int_t^T \left\| \Delta f(\alpha,e)\right\|_{\mathcal{L}^2}^4d\alpha\right|^{\frac{k}{4}}
 + \mathbf{E}\left|\int_t^T \int_\mathcal{E}\left| \Delta f(\alpha,e)\right|^4\tilde{N}(d e,d\alpha)\right|^{\frac{k}{4}}\bigg] \\
&\leqslant \cdots \quad \quad \quad \quad  \\
&\leqslant C\bigg[|x-\hat{x}|^k + \mathbf{E}\int_t^T\left|X^{t, x; u}_\alpha-X^{t, \hat{x}; u}_\alpha\right|^kd\alpha\bigg] \\
&\leqslant C\bigg[|x-\hat{x}|^k + \mathbf{E}\int_t^T \sup _{t \leq r \leq \alpha}\left|X^{t, x; u}_r-X^{t, \hat{x}; u}_r \right|^kd\alpha\bigg] \\
&= C\bigg[|x-\hat{x}|^k + \int_t^T \mathbf{E} \sup _{t \leq r \leq \alpha}\left|X^{t, x; u}_r-X^{t, \hat{x}; u}_r\right|^kd\alpha\bigg].
\end{aligned}
$$
By Gronwall's inequality, (\ref{estimate of sup X-X}) holds. The proof is complete.
\end{proof}

Using the notation introduced in (\ref{BSDEP equation}), we now suppose that, for $i=1,2$ and any $s\in[t,T]$,
$$
g_i(s,y_s^i,z_s^i,\tilde{z}_s^i):=g(s,y_s^i,z_s^i,\tilde{z}_s^i)+\varphi^i(s),
$$
where $\varphi^i(\cdot)\in L_{\mathcal{F}}^2([t,T];\mathbf{R})$ and $g$ is defined in (\ref{BSDEP equation}). Then we have the following result, which belongs to \cite{BBP1997} and \cite{LiPeng2009}.

\begin{mylem}\label{Lemma 4.3}
Let (H1)-(H2) hold. For given $(t,x)\in [0,T)\times \mathbf{R}^n, u_\cdot\in \mathcal{U}^w[t,T]$, the difference of two solutions $(Y_\cdot^{t, x ; u ; 1},Z_\cdot^{t, x ; u ; 1},\tilde{Z}_{(\cdot,\cdot)}^{t, x ; u ; 1})\in L_{\mathcal{F}}^2([t,T];\mathbf{R})\times L_{\mathcal{F},p}^2([t,T];\mathbf{R})\times F_p^2([t,T]\times\mathcal{E};\mathbf{R})$ and $(Y_\cdot^{t, x ; u ; 2},Z_\cdot^{t, x ; u ; 2},\tilde{Z}_{(\cdot,\cdot)}^{t, x ; u ; 2})\in L_{\mathcal{F}}^2([t,T];\mathbf{R})\times L_{\mathcal{F},p}^2([t,T];\mathbf{R})\times F_p^2([t,T]\times\mathcal{E};\mathbf{R})$ of BSDEP (\ref{BSDEP equation}) with the data $(\phi^1,g^1)$ and $(\phi^2,g^2)$, respectively, satisfies the following estimate:
$$
\begin{aligned}
&\quad \left|Y_s^{t, x ; u ; 1}-Y_s^{t, x ; u ; 2}\right|^2+\frac{1}{2} \mathbf{E}\left[\int_s^T e^{\beta(r-s)}\left(\left|Y_r^{t, x ; u ; 1}-Y_r^{t, x ; u ; 2}\right|^2
 +\left|Z_r^{t, x ; u ; 1}-Z_r^{t, x ; u ; 2}\right|^2\right) d r \Big| \mathcal{F}_s^t\right] \\
&\quad +\frac{1}{2} \mathbf{E}\left[\int_s^T \int_\mathcal{E} e^{\beta(r-s)}\left|\tilde{Z}_{(r,e)}^{t, x ; u ; 1}-\tilde{Z}_{(r,e)}^{t, x ; u ; 2}\right|^2 \nu(de) dr \Big| \mathcal{F}_s^t\right] \\
& \leqslant \mathbf{E}\left[e^{\beta(T-s)}\left|\phi^1-\phi^2\right|^2 \Big| \mathcal{F}_s^t\right]+\mathbf{E}\left[\int_s^T e^{\beta(r-s)}\left|\varphi^1(r)-\varphi^2(r)\right|^2 d r \Big| \mathcal{F}_s^t\right],\quad \mathbf{P}\text{-a.s.},
\end{aligned}
$$
for all $t \leqslant s \leqslant T$, where $\beta \geqslant 2+2 C+4 C^2$. Moreover, the following estimate holds:
$$
\begin{aligned}
\mathbf{E}  {\left[\sup _{t \leqslant s \leqslant T}\left|Y^{t, x ; u}_s\right|^2 \Big| \mathcal{F}_s^t\right] }
 \leqslant C \mathbf{E}\left[\left|\phi\left(X^{t, x ; u}_T\right)\right|^2+\int_s^T\left|g\left(r, X^{t, x ; u}_r, 0,0,0, u_r\right)\right|^2 d r \Big| \mathcal{F}_s^t\right] .
\end{aligned}
$$
\end{mylem}

Let us recall the definitions of viscosity solution to (\ref{HJB equation}), and of semi-jets of a continuous function on $[0, T] \times \mathbf{R}^n$ (see Yong and Zhou \cite{YongZhou1999} or Barles and Imbert \cite{BarlesImbert2008}).

\begin{mydef}
(i) A function $v \in C\left([0, T] \times \mathbf{R}^n\right)$ is called a viscosity subsolution of (\ref{HJB equation}) if $v(T, x) \leq -\phi(x), \forall x \in \mathbf{R}^n$, and for any test function $\psi \in C^{1,2}\left([0, T] \times \mathbf{R}^n\right)$, whenever $v-\psi$ attains a global maximum at $(t, x) \in[0, T) \times \mathbf{R}^n$, then
$$
-\psi_t(t, x)+\sup _{u \in \mathbf{U}} G\left(t, x,-\psi(t, x),-\psi_x(t, x),-\psi_{x x}(t, x), u\right) \leqslant 0.
$$
(ii) A function $v \in C\left([0, T] \times \mathbf{R}^n\right)$ is called a viscosity supersolution of (\ref{HJB equation}) if $v(T, x) \geq -\phi(x), \forall x \in \mathbf{R}^n$, and for any test function $\psi \in C^{1,2}\left([0, T] \times \mathbf{R}^n\right)$, whenever  $v-\psi$ attains a global minimum at $(t, x) \in[0, T) \times \mathbf{R}^n$, then
$$
-\psi_t(t, x)+\sup _{u \in \mathbf{U}} G\left(t, x,-\psi(t, x),-\psi_x(t, x),-\psi_{x x}(t, x), u\right) \geqslant 0.
$$
(iii) If $v \in C\left([0, T] \times \mathbf{R}^n\right)$ is both a viscosity subsolution and viscosity supersolution of (\ref{HJB equation}), then it is called a viscosity solution of (\ref{HJB equation}).
\end{mydef}

For $v \in C\left([0, T] \times \mathbf{R}^n\right)$ and $(\hat{t}, \hat{x}) \in[0, T) \times \mathbf{R}^n$, the right parabolic superjet of $v$ at $(\hat{t}, \hat{x})$ is the set triple
$$
\begin{aligned}
\mathcal{D}_{t+, x}^{1,2,+} v(\hat{t}, \hat{x}) &:=\bigg\{(q, p, P) \in \mathbf{R} \times \mathbf{R}^n \times \mathcal{S}^n \Big| v(t, x) \leq v(\hat{t}, \hat{x})+q(t-\hat{t})+\langle p, x-\hat{x}\rangle \\
&\qquad +\frac{1}{2}(x-\hat{x})^\top P(x-\hat{x})+o\left(|t-\hat{t}|+|x-\hat{x}|^2\right), \text { as } t \downarrow \hat{t}, x \rightarrow \hat{x}\bigg\},
\end{aligned}
$$
and the right parabolic subjet of $v$ at $(\hat{t}, \hat{x})$ is the set triple
$$
\begin{aligned}
\mathcal{D}_{t+, x}^{1,2,-} v(\hat{t}, \hat{x}) &:=\bigg\{(q, p, P) \in \mathbf{R} \times \mathbf{R}^n \times \mathcal{S}^n \Big| v(t, x) \geq v(\hat{t}, \hat{x})+q(t-\hat{t})+\langle p, x-\hat{x}\rangle \\
&\qquad +\frac{1}{2}(x-\hat{x})^\top P(x-\hat{x})+o\left(|t-\hat{t}|+|x-\hat{x}|^2\right), \text { as } t \downarrow \hat{t}, x \rightarrow \hat{x}\bigg\} .
\end{aligned}
$$
From the above definitions, we see immediately that
$$
\left\{\begin{array}{l}
\mathcal{D}_{t+, x}^{1,2,+} v(\hat{t}, \hat{x})+[0, \infty) \times\{0\} \times \mathcal{S}_{+}^n=\mathcal{D}_{t+, x}^{1,2,+} v(\hat{t}, \hat{x}), \\
\mathcal{D}_{t+, x}^{1,2,-} v(\hat{t}, \hat{x})-[0, \infty) \times\{0\} \times \mathcal{S}_{+}^n=\mathcal{D}_{t+, x}^{1,2,-} v(\hat{t}, \hat{x}),
\end{array}\right.
$$
where $\mathcal{S}_{+}^n:=\left\{S \in \mathcal{S}^n | S \geqslant 0\right\}$, and $A \pm B:=\{a \pm b \mid a \in A, b \in B\}$ for any subsets $A$ and $B$ in a same Euclidean space.

We will also make use of the partial super-subjets with respect to one of the variables $t$ and $x$. Therefore, we need the following definitions.
\begin{equation}
\left\{\begin{aligned}
\mathcal{D}_x^{2,+} v(\hat{t}, \hat{x}):= & \bigg\{(p, P) \in  \mathbf{R}^n \times \mathcal{S}^n \Big| v(\hat{t}, x) \leq v(\hat{t}, \hat{x})+\langle p, x-\hat{x}\rangle \\
& +\frac{1}{2}(x-\hat{x})^\top P(x-\hat{x})+o\left(|x-\hat{x}|^2\right), \text { as } x \rightarrow \hat{x}\bigg\}, \\
\mathcal{D}_x^{2,-} v(\hat{t}, \hat{x}):= & \bigg\{(p, P) \in  \mathbf{R}^n \times \mathcal{S}^n \Big| v(\hat{t}, x) \geq v(\hat{t}, \hat{x})+\langle p, x-\hat{x}\rangle \\
& +\frac{1}{2}(x-\hat{x})^\top P(x-\hat{x})+o\left(|x-\hat{x}|^2\right), \text { as } x \rightarrow \hat{x}\bigg\},
\end{aligned}\right.
 \label{definition of partial super-subjets to x}
\end{equation}
and
\begin{equation}
\left\{\begin{array}{l}
\mathcal{D}_{t+}^{1,+} v(\hat{t}, \hat{x}) :=\Big\{q \in \mathbf{R} \big| v(t, \hat{x}) \leq v(\hat{t}, \hat{x})+q(t-\hat{t})+o(|t-\hat{t}|), \text { as } t \downarrow \hat{t}\Big\}, \\
\mathcal{D}_{t+}^{1,-} v(\hat{t}, \hat{x}) :=\Big\{q \in \mathbf{R} \big| v(t, \hat{x}) \geq v(\hat{t}, \hat{x})+q(t-\hat{t})+o(|t-\hat{t}|), \text { as } t \downarrow \hat{t}\Big\}.
\end{array}\right.
 \label{definition of partial super-subjets to t}
\end{equation}

The following results are also useful (\cite{YongZhou1999}).
\begin{mydef}\label{Definition 4.2}
Let $Z$ be a Banach space and let $z:[a, b] \rightarrow Z$ be a measurable function that is Bochner integrable. We say that $t$ is a right Lesbesgue point of $z$ if
$$
\lim _{h \rightarrow 0^{+}} \frac{1}{h} \int_t^{t+h}|z(r)-z(t)|_Z d r=0 .
$$
\end{mydef}

\begin{mylem}\label{Lebesgue point}
Let $z:[a, b] \rightarrow Z$ be as in Definition \ref{Definition 4.2}. Then the set of right Lesbesgue points of $z$ is of full measure in $[a, b]$.
\end{mylem}

\subsection{Main results}

\qquad The following theorem shows that the adjoint variables $p, P$ and the value function $V$ relate to each other within the framework of the superjet and the subjet in the state variable $x$ along an optimal trajectory.

\begin{mythm}\label{nonsmooth state theorem}
Suppose (H1)-(H3) hold and let $(t,x)\in [0,T)\times \mathbf{R}^n$ be fixed. Let $(\bar{X}_\cdot^{t,x;\bar{u}},\bar{u}_\cdot)$ be an optimal pair of {\bf Problem (SROCPJ)}. Let $(p_\cdot,q_\cdot,\tilde{q}_{(\cdot,\cdot)})\in L_{\mathcal{F}}^2([t,T];\mathbf{R}^n)\times L_{\mathcal{F},p}^2([t,T];\mathbf{R}^n)\times F_p^2([t,T]\times\mathcal{E};\mathbf{R}^n)$ and $(P_\cdot,Q_\cdot,\tilde{Q}_{(\cdot,\cdot)})\in L_{\mathcal{F}}^2([t,T];\mathcal{S}^n)\times L_{\mathcal{F},p}^2([t,T];\mathcal{S}^n)\times F_p^2([t,T]\times\mathcal{E};\mathcal{S}^n)$ be the first-order and second-order adjoint processes satisfying (\ref{first-order adjoint equation}) and (\ref{second-order adjoint equation}), respectively. Then
\begin{equation}
\begin{array}{lr}
\{-p_s\}\times [-P_s,\infty) \subseteq \mathcal{D}_x^{2,+}V(s,\bar{X}_s^{t,x;\bar{u}}), & \forall  s \in[t, T],\ \mathbf{P}\text{-a.s.}, \\
\mathcal{D}_x^{2,-}V(s,\bar{X}_s^{t,x;\bar{u}}) \subseteq \{-p_s\}\times (-\infty,-P_s], & \forall  s \in[t, T],\ \mathbf{P}\text{-a.s.}.
\end{array}
 \label{relation of p,P and V}
\end{equation}
We also have
\begin{equation}
\mathcal{D}_x^{1,-}V(s,\bar{X}_s^{t,x;\bar{u}}) \subseteq \{-p_s\} \subseteq \mathcal{D}_x^{1,+}V(s,\bar{X}_s^{t,x;\bar{u}}),\quad \forall  s \in[t, T],\ \mathbf{P}\text{-a.s.}.
 \label{relation of p and DV}
\end{equation}
\end{mythm}

\begin{proof}\quad We split the proof into several steps.

Step 1: Variational equation for the SDEP.

Fix an $s\in [t,T]$. For any $z \in \mathbf{R}^n$, denote by $X_\cdot^{s,z;\bar{u}}$ the solution to the following SDEP on $[s,T]$:
\begin{equation}
\begin{aligned}
X_r^{s,z;\bar{u}}& =z+\int_s^r b\left(\alpha,X_\alpha^{s,z;\bar{u}},\bar{u}_\alpha \right)d\alpha+\int_s^r \sigma\left(\alpha,X_\alpha^{s,z;\bar{u}},\bar{u}_\alpha \right)dW_\alpha\\
&\quad +\int_s^r \int_\mathcal{E}f\left(\alpha,X_{\alpha-}^{s,z;\bar{u}},\bar{u}_\alpha,e \right)\tilde{N}(d e,d\alpha).
\end{aligned}
 \label{SDEP with z}
\end{equation}
It is clear that (\ref{SDEP with z}) can be regarded as an SDEP on $(\Omega,\mathcal{F},\{\mathcal{F}_r^t\}_{r\geq t},\mathbf{P}(\cdot \mid \mathcal{F}_s^t)(\omega))$ for $\mathbf{P}$-a.s. $\omega$, where $\mathbf{P}(\cdot | \mathcal{F}_s^t)(\omega)$ is the regular conditional probability given $\mathcal{F}_s^t$ defined on $(\Omega,\mathcal{F})$ (see \cite{IkedaWatanabe1981}). For any $s \leq r \leq T$, set $\hat{X}_r:=X_r^{s,z;\bar{u}}-\bar{X}_r^{t,x;\bar{u}}$. Thus by Corollary \ref{Corollary 4.1}, we have for any  $k=2^i,i=1,2,\cdots$,
\begin{equation}
\mathbf{E}\left[\sup _{s \leqslant r \leqslant T}\left|\hat{X}_r\right|^k|\mathcal{F}_s^t\right] \leqslant C\left|z-\bar{X}_s^{t,x;\bar{u}}\right|^k, \quad s \in[t, T].
 \label{estimate of sup X}
\end{equation}
Now we rewrite the equation for the RCLL process $\hat{X}_\cdot$ in two different ways based on different orders of expansion, which called the first-order and second-order variational equations, respectively:
\begin{equation}
\left\{\begin{aligned}
d \hat{X}_r& = \left[\bar{b}_x(r) \hat{X}_r+\varepsilon_{z 1}(r)\right] d r+\left[\bar{\sigma}_x(r) \hat{X}_r+\varepsilon_{z 2}(r)\right] d W_r \\
&\quad +\int_{\mathcal{E}} \left[\bar{f}_x(r, e) \hat{X}_{r-}+\varepsilon_{z 3}(r, e)\right] \tilde{N}(d e, d r) , \quad r \in[s, T], \\
\hat{X}_s& =z-\bar{X}_s^{t, x ; \bar{u}},
\end{aligned}\right.
 \label{first-order variational equation}
\end{equation}
where
\begin{equation}
\left\{\begin{aligned}
\varepsilon_{z 1}(r):=& \int_0^1\left[b_x\left(r, \bar{X}^{t, x ; \bar{u}}_r+\theta \hat{X}_r, \bar{u}_r\right)-\bar{b}_x(r)\right] \hat{X}_r d \theta, \\
\varepsilon_{z 2}(r):=& \int_0^1\left[\sigma_x\left(r, \bar{X}^{t, x ; \bar{u}}_r+\theta \hat{X}_r, \bar{u}_r\right)-\bar{\sigma}_x(r)\right] \hat{X}_r d \theta, \\
\varepsilon_{z 3}(r, e):=& \int_0^1\left[f_x\left(r, \bar{X}^{t, x ; \bar{u}}_{r-}+\theta \hat{X}_{r-}, \bar{u}_r, e\right)-\bar{f}_x(r,e)\right] \hat{X}_{r-} d \theta,
\end{aligned}\right.
 \label{definiton of z 1-3}
\end{equation}
and
\begin{equation}
\left\{\begin{aligned}
d \hat{X}_r & = \left[\bar{b}_x(r) \hat{X}_r+\frac{1}{2} \hat{X}_r^\top \bar{b}_{x x}(r) \hat{X}_r+\varepsilon_{z 4}(r)\right] d r \\
&\quad +\left[\bar{\sigma}_x(r) \hat{X}_r+\frac{1}{2} \hat{X}_r^\top \bar{\sigma}_{x x}(r) \hat{X}_r+\varepsilon_{z 5}(r)\right] d W_r \\
&\quad +\int_{\mathcal{E}}\left[\bar{f}_x(r, e) \hat{X}_{r-}+\frac{1}{2} \hat{X}_{r-}^\top \bar{f}_{x x}(r, e) \hat{X}_{r-}+\varepsilon_{z 6}(r, e)\right] \tilde{N}(d e,d r) , \quad r \in[s, T], \\
\hat{X}_s & = z -\bar{X}^{t, x ; \bar{u}}_s,
\end{aligned}\right.
\label{second-order variational equation}
\end{equation}
where
\begin{equation}
\left\{\begin{aligned}
\varepsilon_{z 4}(r):= & \int_0^1(1-\theta)\hat{X}_r^\top \left[b _ { x x } \left(r, \bar{X}^{t, x ; \bar{u}}_r+\theta \hat{X}_r, \bar{u}_r\right)-\bar{b}_{x x}(r)\right] \hat{X}_r d \theta, \\
\varepsilon_{z 5}(r):= & \int_0^1(1-\theta) \hat{X}_r^\top \left[\sigma _ { x x } \left(r, \bar{X}^{t, x ; \bar{u}}_r+\theta \hat{X}_r, \bar{u}_r\right)-\bar{\sigma}_{x x}(r)\right] \hat{X}_r d \theta, \\
\varepsilon_{z 6}(r, e):= & \int_0^1(1-\theta) \hat{X}_{r-}^\top \left[f _ { x x } \left(r, \bar{X}^{t, x ; \bar{u}}_{r-}+\theta \hat{X}_{r-}, \bar{u}_r, e\right)-\bar{f}_{x x}(r, e)\right] \hat{X}_{r-} d \theta .
\end{aligned}\right.
 \label{definition of z 4-6}
\end{equation}

Step 2. Estimates of remainder terms of SDEPs.

We are going to show that, there exists a deterministic continuous and increasing function $\delta:[0, \infty) \rightarrow[0, \infty)$, independent of $z \in \mathbf{R}^n$, with $\frac{\delta(\alpha)}{\alpha} \rightarrow 0$ as $\alpha \rightarrow 0$, such that
\begin{equation}
\left\{\begin{aligned}
\mathbf{E}\left[\int_s^T\left|\varepsilon_{z 1}(r)\right|^2 d r \Big| \mathcal{F}_s^t\right] \leqslant \delta\left(\left|z-\bar{X}^{t, x ; \bar{u}}_s\right|^2\right), \quad \mathbf{P}\text {-a.s.},  \\
\mathbf{E}\left[\int_s^T\left|\varepsilon_{z 2}(r)\right|^2 d r \Big| \mathcal{F}_s^t\right] \leqslant \delta\left(\left|z-\bar{X}^{t, x ; \bar{u}}_s\right|^2\right), \quad \mathbf{P}\text {-a.s.}, \\
\mathbf{E}\left[\int_s^T\left\|\varepsilon_{z 3}(r, e)\right\|_{\mathcal{L}^2}^2 d r \Big| \mathcal{F}_s^t\right] \leqslant \delta\left(\left|z-\bar{X}^{t, x ; \bar{u}}_s\right|^2\right), \quad \mathbf{P}\text {-a.s.}, \\
\mathbf{E}\left[\int_s^T\left|\varepsilon_{z 4}(r)\right|^2 d r \Big| \mathcal{F}_s^t\right] \leq \delta\left(\left|z-\bar{X}^{t, x ; \bar{u}}_s\right|^4\right), \quad \mathbf{P}\text {-a.s.}, \\
\mathbf{E}\left[\int_s^T\left|\varepsilon_{z 5}(r)\right|^2 d r \Big| \mathcal{F}_s^t\right] \leqslant \delta\left(\left|z-\bar{X}^{t, x ; \bar{u}}_s\right|^4\right), \quad \mathbf{P}\text {-a.s.},  \\
\mathbf{E}\left[\int_s^T\left\|\varepsilon_{z 6}(r, e)\right\|_{\mathcal{L}^2}^2 d r \Big| \mathcal{F}_s^t\right] \leqslant \delta\left(\left|z-\bar{X}^{t, x ; \bar{u}}_s\right|^4\right), \quad \mathbf{P}\text {-a.s.}.
\end{aligned}\right.
\label{estimate of z 1-6}
\end{equation}
Here $\mathbf{E}\left[\int_s^T\left|\varepsilon_{z 1}(r)\right|^{2} d r \big| \mathcal{F}_s^t\right]=\delta(|z-\bar{X}^{t, x ; \bar{u}}_s|^{2})$, $\mathbf{P}$-a.s., means that for $\mathbf{P}$-a.s. $\omega$ fixed, $\\\mathbf{E}\left[\int_s^T\left|\varepsilon_{z 1}(r)\right|^{2} d r \big| \mathcal{F}_s^t\right](\omega)=\delta(|z-\bar{X}^{t, x ; \bar{u}}_{(s, \omega)}|^{2})$, where $\delta(\cdot)$ is almost surely a deterministic function under the regular conditional probability $\mathbf{P}\left(\cdot | \mathcal{F}_s^t\right)(\omega)$. Moreover, $\delta(|z-\bar{X}^{t, x ; \bar{u}}_{(s, \omega)}|^{2})$ depends only on the size of $\left|z-\bar{X}^{t, x ; \bar{u}}_{(s, \omega)}\right|$, and it is independent of $z$. Such notation has a similar meaning for other estimates in (\ref{estimate of z 1-6}) as well as in what follows in the paper.
To prove conveniently, we denote
$$
\begin{aligned}
b_x(r, \theta):=&\ b_x\left(r, \bar{X}^{t, x ; \bar{u}}_r+\theta \hat{X}_r, \bar{u}_r\right), \\
\sigma_x(r, \theta):=&\ \sigma_x\left(r, \bar{X}^{t, x ; \bar{u}}_r+\theta \hat{X}_r, \bar{u}_r\right), \\
f_x(r, \theta,e):=&\ f_x\left(r, \bar{X}^{t, x ; \bar{u}}_r+\theta \hat{X}_r, \bar{u}_r,e\right),
\end{aligned}
$$
and similar notations used for all their derivatives.

Now, we start to prove (\ref{estimate of z 1-6}). By the continuity and uniformly boundedness of $b_x, b_{x x}, \sigma_x, \sigma_{x x}$, $f_x, f_{x x}$ as well as (\ref{estimate of sup X}), we have
$$
\begin{aligned}
&\mathbf{E}\left[\int_s^T\left|\varepsilon_{z 1}(r)\right|^2 d r \mid \mathcal{F}_s^t\right]
  \leqslant \int_s^T \mathbf{E}\left\{\int_0^1\left|b_x(r, \theta)-\bar{b}_x(r)\right|^2 d \theta \cdot\big|\hat{X}_r\big|^2 \Big| \mathcal{F}_s^t\right\} d r \\
& \leqslant C \int_s^T \mathbf{E}\left[\big|\hat{X}_r\big|^4 \Big| \mathcal{F}_s^t\right] d r \leqslant C\left|z-\bar{X}^{t, x ; \bar{u}}_s\right|^4 = \delta_1 \left(\left|z-\bar{X}^{t, x ; \bar{u}}_s\right|^2\right).
\end{aligned}
$$
Thus the first inequality in (\ref{estimate of z 1-6}) holds and similarly for the second one if we choose $\delta_1(r) \equiv Cr^2, r \geq 0$. And we have
$$
\begin{aligned}
& \mathbf{E}\left[\int_s^T\left\|\varepsilon_{z 3}(r, e)\right\|_{\mathcal{L}^2}^2 d r \Big| \mathcal{F}_s^t\right]
 \leqslant \int_s^T \mathbf{E}\left\{\int_0^1\left\|f_x(r, \theta, e)-\bar{f}_x(r, e)\right\|_{\mathcal{L}^2}^2 d \theta \cdot\big|\hat{X}_{r-}\big|^2 \Big| \mathcal{F}_s^t\right\} d r \\
& \leqslant C \int_s^T \mathbf{E}\left[\big|\hat{X}_{r-}\big|^4 \Big| \mathcal{F}_s^t\right] d r = C \int_s^T \mathbf{E}\left[\big|\hat{X}_r\big|^4 \Big| \mathcal{F}_s^t\right] d r
 \leqslant C\left|z-\bar{X}^{t, x ; \bar{u}}_s\right|^4 = \delta_3 \left(\left|z-\bar{X}^{t, x ; \bar{u}}_s\right|^2\right),
\end{aligned}
$$
where first equality holds because the discontinuous points of $\hat{X}_\cdot$ are at most countable. Thus, the third inequality in (\ref{estimate of z 1-6}) follows. Moreover, from the modulus continuity of $b_{x x}$ (see (H3)), we can also show that
$$
\begin{aligned}
& \mathbf{E}  {\left[\int_s^T\left|\varepsilon_{z 4}(r)\right|^2 d r \Big| \mathcal{F}_s^t\right] }
 \leqslant \int_s^T \mathbf{E}\left[\int_0^1\left|b_{x x}(r, \theta)-\bar{b}_{x x}(r)\right|^2 d \theta \cdot\big|\hat{X}_r\big|^4 \Big| \mathcal{F}_s^t\right] d r \\
& \leqslant \int_s^T\left\{\mathbf{E}\left[\varpi\left(\big|\hat{X}_r\big|^4\right) \Big| \mathcal{F}_s^t\right]\right\}^{\frac{1}{2}}
 \left\{\mathbf{E}\left[\big|\hat{X}_r\big|^8 \Big| \mathcal{F}_s^t\right]\right\}^{\frac{1}{2}} d r \\
& \leqslant C \int_s^T\left\{\mathbf{E}\left[\varpi\left(\big|\hat{X}_r\big|^4\right) \Big| \mathcal{F}_s^t\right]\right\}^{\frac{1}{2}} d r \cdot\left|z-\bar{X}^{t, x ; \bar{u}}_s\right|^4
 \leqslant \delta_4 \left(\left|z-\bar{X}^{t, x ; \bar{u}}_s\right|^4\right).
\end{aligned}
$$
Thus the fourth inequality in (\ref{estimate of z 1-6}) holds if we choose $\delta_4(r) \equiv Cr\sqrt{\varpi(r)}, r \geq 0$. The fifth inequality in (\ref{estimate of z 1-6}) can be proved similarly. For the last inequality, we can obtain that
$$
\begin{aligned}
& \mathbf{E}\left[\int_s^T\left\|\varepsilon_{z 6}(r, e)\right\|_{\mathcal{L}^2}^2 d r \Big| \mathcal{F}_s^t\right]
 \leqslant \int_s^T \mathbf{E}\left[\int_0^1\left\|f_{x x}(r, \theta,e)-\bar{f}_{x x}(r, e)\right\|_{\mathcal{L}^2}^2 d \theta \cdot\big|\hat{X}_{r-}\big|^4 \Big| \mathcal{F}_s^t\right] d r \\
& \leqslant \int_s^T\left\{\mathbf{E}\left[\varpi\left(\big|\hat{X}_{r-}\big|^4\right) \Big| \mathcal{F}_s^t\right]\right\}^{\frac{1}{2}}
 \left\{\mathbf{E}\left[\left|\hat{X}_{r-}\right|^8 \Big| \mathcal{F}_s^t\right]\right\}^{\frac{1}{2}} d r \\
& \leqslant C \int_s^T\left\{\mathbf{E}\left[\varpi\left(\big|\hat{X}_r\big|\right) \Big| \mathcal{F}_s^t\right]\right\}^{\frac{1}{2}} d r \cdot\left|z-\bar{X}^{t, x ; \bar{u}}_s\right|^4
 \leqslant \delta_6 \left(\left|z-\bar{X}^{t, x ; \bar{u}}_s\right|^4\right).
\end{aligned}
$$
Finally, we can select the largest $\delta(\cdot)$ obtained in the above six calculations. For example, we can choose an enough large constant $C>0$ and define $\delta(r) \equiv Cr(r \vee \sqrt{\varpi(r)}), r \geq 0$. Then (\ref{estimate of z 1-6}) follows with a $\delta(\cdot)$ independent of $z \in \mathbf{R}^n$.

Step 3. Duality relation.

Applying It\^{o}'s formula to $\langle p_\cdot,\hat{X}_\cdot\rangle$, by (\ref{first-order adjoint equation}), (\ref{second-order variational equation}), we have for $r\in[s,T]$,
\begin{equation}
\begin{aligned}
&d\langle p_r,\hat{X}_r\rangle =\bigg\{-\Big\langle \hat{X}_r,\bar{g}_x(r)+\bar{g}_y(r)p_r+\bar{g}_z(r)\left[\bar{\sigma}_x(r)^\top p_r+q_r\right]  \\
& \quad +\bar{g}_{\tilde{z}}(r)\int_{\mathcal{E}}\left[\bar{f}_x(r,e)^\top p_r+\tilde{q}_{(r,e)}+\bar{f}_x(r,e)^\top \tilde{q}_{(r,e)}\right]\nu (d e)\Big\rangle \\
& \quad +\frac{1}{2} \left\langle p_r,\hat{X}_r^\top \bar{b}_{x x}(r)\hat{X}_r \right\rangle +\frac{1}{2} \left\langle q_r,\hat{X}_r^\top \bar{\sigma}_{x x}(r)\hat{X}_r \right\rangle
 +\frac{1}{2} \int_{\mathcal{E}}\left\langle \tilde{q}_{(r,e)},\hat{X}_r^\top \bar{f}_{x x}(r,e)\hat{X}_r \right\rangle \nu(d e) \\
& \quad + \left\langle p_r,\varepsilon_{z 4}(r)\right\rangle +\langle q_r,\varepsilon_{z 5}(r)\rangle+\int_{\mathcal{E}}\langle \tilde{q}_{(r,e)},\varepsilon_{z 6}(r,e)\rangle \nu(d e)\bigg\}d r \\
& \quad +\bigg\{ \left\langle q_r,\hat{X}_r\right\rangle + \left\langle p_r,\bar{\sigma}_x(r)\hat{X}_r+\frac{1}{2}\hat{X}_r^\top \bar{\sigma}_{x x}(r)\hat{X}_r+\varepsilon_{z 5}(r) \right\rangle\bigg\}d W_r \\
& \quad +\int_{\mathcal{E}}\bigg\{\left\langle \tilde{q}_{(r,e)},\hat{X}_{r-}\right\rangle + \Big\langle p_r,\bar{f}_x(r,e)\hat{X}_{r-}
 +\frac{1}{2}\hat{X}_{r-}^\top \bar{f}_{x x}(r,e)\hat{X}_{r-}+\varepsilon_{z 6}(r,e)\Big\rangle  \\
& \quad +\left\langle \tilde{q}_{(r,e)},\bar{f}_x(r,e)\hat{X}_{r-}+\frac{1}{2}\hat{X}_{r-}^\top \bar{f}_{x x}(r,e)\hat{X}_{r-}+\varepsilon_{z 6}(r,e)\right\rangle\bigg\}\tilde{N}(d e,d r).
\end{aligned}
 \label{Ito's formula to pX}
\end{equation}
For $r \in[s, T]$, setting $\Phi_r:=\hat{X}_r\hat{X}_r^\top$ and applying It\^{o}'s formula again, noting (\ref{first-order variational equation}), we get
\begin{equation}
\left\{\begin{aligned}
 d \Phi_r=& \left[\bar{b}_x(r) \Phi_r+\Phi_r \bar{b}_x(r)^\top +\bar{\sigma}_x(r) \Phi_r \bar{\sigma}_x(r)^{\top}+\int_{\mathcal{E}} \bar{f}_x(r, e) \Phi_r \bar{f}_x(r, e)^\top \nu(d e)+\varepsilon_{z 7}(r)\right] d r \\
&+\left[\bar{\sigma}_x(r) \Phi_r+\Phi_r \bar{\sigma}_x(r)^\top +\varepsilon_{z 8}(r)\right] d W_r \\
&+\int_{\mathcal{E}}\left[\bar{f}_x(r, e) \Phi_{r-} \bar{f}_x(r, e)^\top +\bar{f}_x(r, e) \Phi_{r-}+\Phi_{r-} \bar{f}_x(r, e)^\top +\varepsilon_{z 9}(r, e)\right] \tilde{N}(d e, d r), \\
\Phi_s=& \ \hat{X}_s \hat{X}_s^\top,
\end{aligned}\right.
 \label{ito for Phi}
\end{equation}
where
\begin{equation}
\left\{\begin{aligned}
\varepsilon_{z 7}(r)&:=  \varepsilon_{z 1}(r) \hat{X}_r^\top +\hat{X}_r \varepsilon_{z 1}(r)^\top +\bar{\sigma}_x(r) \hat{X}_r \varepsilon_{z 2}(r)^\top +\varepsilon_{z 2}(r) \hat{X}_r^\top \bar{\sigma}_x(r)^\top \\
&\quad +\varepsilon_{z 2}(r) \varepsilon_{z 2}(r)^\top +\int_{\mathcal{E}}\left\{\bar{f}_x(r, e) \hat{X}_r \varepsilon_{z 3}(r, e)^\top \right.\\
&\quad \left.+\varepsilon_{z 3}(r, e) \hat{X}_r^\top \bar{f}_x(r, e)^\top +\varepsilon_{z 3}(r, e) \varepsilon_{z 3}(r, e)^\top \right\} \nu(d e), \\
\varepsilon_{z 8}(r)&:= \varepsilon_{z 2}(r) \hat{X}_r^\top +\hat{X}_r \varepsilon_{z 2}(r)^\top, \\
\varepsilon_{z 9}(r, e)&:= \bar{f}_x(r, e) \hat{X}_{r-} \varepsilon_{z 3}(r, e)^\top +\varepsilon_{z 3}(r, e) \hat{X}_{r-}^\top \bar{f}_x(r, e)^\top +\hat{X}_{r-} \varepsilon_{z 3}(r, e)^\top \\
&\quad +\varepsilon_{z 3}(r, e) \hat{X}_{r-}^\top +\varepsilon_{z 3}(r, e) \varepsilon_{z 3}(r, e)^\top .
\end{aligned}\right.
 \label{definition of z 7-9}
\end{equation}
Once more applying It\^{o}'s formula to $\operatorname{tr}\big\{P_\cdot\Phi_\cdot\big\}$, using (\ref{second-order adjoint equation}), we obtain for $r\in[s,T]$,
\begin{equation}
\begin{aligned}
& d\operatorname{tr}\big\{P_r\Phi_r\big\}
  = \operatorname{tr}\bigg\{-\Phi_r\bar{g}_y(r)P_r-\Phi_r\bar{g}_z(r)\bar{\sigma}_x(r)^\top P_r-\Phi_rP_r\bar{g}_z(r)\bar{\sigma}_x(r)-\Phi_r\bar{g}_z(r)Q_r \\
& \quad -\Phi_r\bar{b}_{x x}(r)^\top p_r-\Phi_r\bar{\sigma}_{x x}(r)^\top \left[\bar{g}_z(r)p_r+q_r\right]+P_r\varepsilon_{z 7}(r)+Q_r\varepsilon_{z 8}(r)  \\
& \quad -\Phi_r \cdot \Big[I_{n\times n},p_r,\bar{\sigma}_x(r)^\top p_r+q_r,\int_{\mathcal{E}}\left(\bar{f}_x(r,e)^\top p_r+\tilde{q}_{(r,e)}+\bar{f}_x(r,e)^\top \tilde{q}_{r,e)}\right)\nu(de)\Big] \\
& \quad \cdot D^2\bar{g}(r)\Big[I_{n\times n},p_r,\bar{\sigma}_x(r)^\top p_r+q_r,\int_{\mathcal{E}}\left(\bar{f}_x(r,e)^\top p_r+\tilde{q}_{(r,e)}+\bar{f}_x(r,e)^\top \tilde{q}_{(r,e)}\right)\nu(de)\Big]^\top \\
& \quad -\int_{\mathcal{E}}\Big[\Phi_r\bar{f}_{x x}(r,e)^\top \tilde{q}_{(r,e)}+\Phi_r\bar{f}_{x x}(r,e)^\top \bar{g}_{\tilde{z}}(r)p_r+\Phi_r\bar{f}_{x x}(r,e)^\top \bar{g}_{\tilde{z}}(r)\tilde{q}_{(r,e)} \\
& \quad +\Phi_r\bar{g}_{\tilde{z}}(r)\bar{f}_x(r,e)^\top P_r+\Phi_rP_r\bar{g}_{\tilde{z}}(r)\bar{f}_x(r,e)+\Phi_r\bar{g}_{\tilde{z}}(r)\bar{f}_x(r,e)^\top \tilde{Q}_{(r,e)} \\
& \quad +\Phi_r\bar{g}_{\tilde{z}}(r)\tilde{Q}_{(r,e)}\bar{f}_x(r,e)+\Phi_r\bar{g}_{\tilde{z}}(r)\bar{f}_x(r,e)^\top \tilde{Q}_{(r,e)}\bar{f}_x(r,e) \\
& \quad +\Phi_r\bar{g}_{\tilde{z}}(r)\bar{f}_x(r,e)^\top P_r\bar{f}_x(r,e)+\Phi_r\bar{g}_{\tilde{z}}(r)\tilde{Q}_{(r,e)}-\tilde{Q}_{(r,e)}\varepsilon_{z 9}(r,e)\Big]\nu(de)\bigg\}dr \\
& \quad +\operatorname{tr}\bigg\{\Phi_rQ_r+P_r\Phi_r\bar{\sigma}_x(r)^\top +P_r\bar{\sigma}_x(r)\Phi_r+P_r\varepsilon_{z 8}(r)\bigg\}dW_r \\
& \quad +\operatorname{tr}\bigg\{\int_{\mathcal{E}}\Big[\Phi_{r-}\tilde{Q}_{(r,e)}+P_r\bar{f}_x(r,e)\Phi_{r-}+P_r\Phi_{r-}\bar{f}_x(r,e)^\top +P_r\bar{f}_x(r,e)\Phi_{r-}\bar{f}_x(r,e)^\top \\
& \quad +\tilde{Q}_{(r,e)}\bar{f}_x(r,e)\Phi_{r-}+\tilde{Q}_{(r,e)}\Phi_{r-}\bar{f}_x(r,e)^\top +\tilde{Q}_{(r,e)}\bar{f}_x(r,e)\Phi_{r-}\bar{f}_x(r,e)^\top \\
& \quad +P_r\varepsilon_{z 9}(r, e)+\tilde{Q}_{(r,e)}\varepsilon_{z 9}(r, e)\Big]\bigg\}\tilde{N}(d e,d r).
\end{aligned}
 \label{Ito's formula to PPhi}
\end{equation}
By (\ref{Ito's formula to pX}) and (\ref{Ito's formula to PPhi}), for $\hat{Y}_r:=\left\langle p_r,\hat{X}_r\right\rangle+\frac{1}{2}\left\langle P_r\hat{X}_r,\hat{X}_r\right\rangle$, we obtain
\begin{equation}
d\hat{Y}_r=C(r)dr+\hat{Z}_rd W_r+\int_{\mathcal{E}}\hat{\tilde{Z}}_{(r,e)}\tilde{N}(d e,d r),\ \ \ \ r\in [s,T],
 \label{definition of hatY}
\end{equation}
where
\begin{equation*}
\begin{aligned}
 C(r)& := -\left\langle \hat{X}_r,\bar{g}_x(r)+\bar{g}_y(r)p_r+\bar{g}_z(r)\left(\bar{\sigma}_x(r)^\top p_r+q_r\right)+\bar{g}_{\tilde{z}}(r)\int_{\mathcal{E}}\Big[\bar{f}_x(r,e)^\top p_r +\tilde{q}_{(r,e)}\right. \\
& \quad +\bar{f}_x(r,e)^\top \tilde{q}_{(r,e)}\Big]\nu (d e)\bigg\rangle+ \langle p_r,\varepsilon_{z 4}(r)\rangle +\langle q_r,\varepsilon_{z 5}(r)\rangle
 +\int_{\mathcal{E}}\langle \tilde{q}_{(r,e)},\varepsilon_{z 6}(r,e)\rangle \nu(d e)\\
& \quad +\operatorname{tr}\bigg\{-\frac{1}{2}\Phi_r\bar{g}_y(r)P_r-\frac{1}{2}\Phi_r\bar{g}_z(r)\bar{\sigma}_x(r)^\top P_r-\frac{1}{2}\Phi_rP_r\bar{g}_z(r)\bar{\sigma}_x(r)-\frac{1}{2}\Phi_r\bar{g}_z(r)Q_r\\
& \quad -\frac{1}{2}\Phi_r\bar{\sigma}_{x x}(r)^\top \bar{g}_z(r)p_r+\frac{1}{2}P_r\varepsilon_{z 7}(r)+\frac{1}{2}Q_r\varepsilon_{z 8}(r)-\frac{1}{2}\Phi_r\Psi(r) D^2\bar{g}(r)\Psi(r)^\top \\
& \quad -\int_{\mathcal{E}}\left[\frac{1}{2}\Phi_r\bar{f}_{x x}(r,e)^\top \bar{g}_{\tilde{z}}(r)p_r+\frac{1}{2}\Phi_r\bar{f}_{x x}(r,e)^\top \bar{g}_{\tilde{z}}(r)\tilde{q}_{(r,e)}
 +\frac{1}{2}\Phi_r\bar{g}_{\tilde{z}}(r)\bar{f}_x(r,e)^\top P_r \right.\\
& \quad +\frac{1}{2}\Phi_rP_r\bar{g}_{\tilde{z}}(r)\bar{f}_x(r,e)+\frac{1}{2}\Phi_r\bar{g}_{\tilde{z}}(r)\bar{f}_x(r,e)^\top \tilde{Q}_{(r,e)}+\frac{1}{2}\Phi_r\bar{g}_{\tilde{z}}(r)\tilde{Q}_{(r,e)}\bar{f}_x(r,e)\\
\end{aligned}
\end{equation*}
\begin{equation}
\begin{aligned}
& \quad +\frac{1}{2}\tilde{Q}_{(r,e)}\varepsilon_{z 9}(r,e)+\frac{1}{2}\Phi_r\bar{g}_{\tilde{z}}(r)\bar{f}_x(r,e)^\top \tilde{Q}_{(r,e)}\bar{f}_x(r,e)\\
& \quad \left.+\frac{1}{2}\Phi_r\bar{g}_{\tilde{z}}(r)\bar{f}_x(r,e)^\top {P}_r\bar{f}_x(r,e)+\frac{1}{2}\Phi_r\bar{g}_{\tilde{z}}(r)\tilde{Q}_{(r,e)}\right]\nu(de)\bigg\}, \\
 \hat{Z}_r &:=\left\langle q_r,\hat{X}_r\right\rangle + \left\langle p_r,\bar{\sigma}_x(r)\hat{X}_r+\frac{1}{2}\hat{X}_r^\top \bar{\sigma}_{x x}(r)\hat{X}_r+\varepsilon_{z 5}(r)\right\rangle \\
& \quad +\operatorname{tr}\bigg\{\frac{1}{2}\Phi_rQ_r+\frac{1}{2}P_r\Phi_r\bar{\sigma}_x(r)^\top +\frac{1}{2}P_r\bar{\sigma}_x(r)\Phi_r+\frac{1}{2}P_r\varepsilon_{z 8}(r)\bigg\}, \\
 \hat{\tilde{Z}}_{(r,e)} &:=\left\langle \tilde{q}_{(r,e)},\hat{X}_{r-}\right\rangle + \left\langle p_r,\bar{f}_x(r,e)\hat{X}_{r-}+\frac{1}{2}\hat{X}_{r-}^\top \bar{f}_{x x}(r,e)\hat{X}_{r-}+\varepsilon_{z 6}(r,e)\right\rangle \\
& \quad +\left\langle \tilde{q}_{(r,e)},\bar{f}_x(r,e)\hat{X}_{r-}\right\rangle +\left\langle \tilde{q}_{(r,e)},\frac{1}{2}\hat{X}_{r-}^\top \bar{f}_{x x}(r,e)\hat{X}_{r-}+\varepsilon_{z 6}(r,e)\right\rangle \\
& \quad +\operatorname{tr}\bigg\{\frac{1}{2}\Phi_{r-}\tilde{Q}_{(r,e)}+\frac{1}{2}P_r\bar{f}_x(r,e)\Phi_{r-}+\frac{1}{2}P_r\Phi_{r-}\bar{f}_x(r,e)^\top  \\
& \quad +\frac{1}{2}P_r\bar{f}_x(r,e)\Phi_{r-}\bar{f}_x(r,e)^\top +\frac{1}{2}\tilde{Q}_{(r,e)}\bar{f}_x(r,e)\Phi_{r-}+\frac{1}{2}\tilde{Q}_{(r,e)}\Phi_{r-}\bar{f}_x(r,e)^\top \\
& \quad +\frac{1}{2}\tilde{Q}_{(r,e)}\bar{f}_x(r,e)\Phi_{r-}\bar{f}_x(r,e)^\top +\frac{1}{2}P_r\varepsilon_{z 9}(r, e)+\frac{1}{2}\tilde{Q}_{(r,e)}\varepsilon_{z 9}(r, e) \bigg\}.
\end{aligned}
 \label{definition of C and hatZ}
\end{equation}

Step 4. Variational equation for the BSDEP.

For the above $z \in \mathbf{R}^n$, recall that $X^{s, z ; \bar{u}}_\cdot$ is given by (\ref{SDEP with z}) and denote by $\left(Y^{s, z ; \bar{u}}_\cdot, Z^{s, z ; \bar{u}}_\cdot,\tilde{Z}_{(\cdot,\cdot)}^{s,z;\bar{u}}\right)$ the solution to the following BSDEP on $[s, T]$ :
\begin{equation}
\begin{aligned}
Y_r^{s, z ; \bar{u}}&= \phi\left(X_T^{s,z;\bar{u}}\right)
 +\int_r^Tg\left(\alpha, X_\alpha^{s, z ; \bar{u}}, Y_\alpha^{s, z ; \bar{u}}, Z_\alpha^{s, z ; \bar{u}}, \tilde{Z}_{(\alpha,\cdot)}^{s,z;\bar{u}}, \bar{u}_\alpha\right) d \alpha  \\
&\quad -\int_r^TZ_\alpha^{s, z ; \bar{u}} d W_\alpha -\int_r^T\int_{\mathcal{E}}\tilde{Z}_{(\alpha,e)}^{s,z;\bar{u}} \tilde{N}(d e, d \alpha), \\
\end{aligned}
 \label{BSDEP with z}
\end{equation}
and similarly (\ref{BSDEP with z}) is a BSDEP on $\left(\Omega, \mathcal{F},\left\{\mathcal{F}_r^t\right\}_{r \geq t}, \mathbf{P}\left(\cdot | \mathcal{F}_s^t\right)(\omega)\right)$ for $\mathbf{P}$-a.s. $\omega$.
For any $s \leqslant r \leqslant T$, set
\begin{equation}
\begin{aligned}
\Delta{Y}_r:= Y^{s, z ; \bar{u}}_r-\hat{Y}_r, \quad \Delta{Z}_r:= Z^{s, z ; \bar{u}}_r-\hat{Z}_r, \quad \Delta{\tilde{Z}}_{(r,e)}:= \tilde{Z}_{(r,e)}^{s,z;\bar{u}}-\hat{\tilde{Z}}_{(r,e)}.
\end{aligned}
 \label{definition of DeltaYZZ}
\end{equation}
Thus by (\ref{definition of hatY}) and (\ref{BSDEP with z}), we get for all $r\in[s,T]$,
\begin{equation}
\left\{\begin{aligned}
d\left(\Delta{Y}_r-\bar{Y}_r^{t,x;\bar{u}}\right) =& -\left[C(r)+g\left(r, X_r^{s, z ; \bar{u}}, Y_r^{s, z ; \bar{u}}, Z_r^{s, z ; \bar{u}}, \tilde{Z}_{(r,\cdot)}^{s,z;\bar{u}}, \bar{u}_r\right) -\bar{g}(r)\right]d r \\
& +\Big(\Delta{Z}_r-\bar{Z}_r^{t,x;\bar{u}}\Big)dW_r+\int_{\mathcal{E}}\left(\Delta{\tilde{Z}}_{(r,e)}-\bar{\tilde{Z}}_{(r,e)}^{t,x;\bar{u}}\right)\tilde{N}(d e,d r), \\
\Delta{Y}_T-\bar{Y}_T^{t,x;\bar{u}} =& \int_0^1(1-\theta)\hat{X}_T^\top \left[\phi_{x x}(\bar{X}_T^{t,x;\bar{u}}+\theta\hat{X}_T)-\phi_{x x}(\bar{X}_T^{t,x;\bar{u}})\right]\hat{X}_Td\theta.
\end{aligned}\right.
 \label{definition of DeltaY-barY}
\end{equation}
By the modulus continuity of $\phi_{x x}$, we have
$$
\begin{aligned}
& \mathbf{E}\left[\left|\Delta{Y}_T-\bar{Y}_T^{t, x ; \bar{u}}\right| \Big| \mathcal{F}_s^t\right] \\
& \leqslant\left\{\mathbf{E}\left[\left(\int_0^1\left|\phi_{x x}\left(\bar{X}_T^{t, x ; \bar{u}}+\theta \hat{X}_T\right)
 -\phi_{x x}\left(\bar{X}_T^{t, x ; \bar{u}}\right)\right| d \theta\right)^2 \Big| \mathcal{F}_s^t\right]\right\}^{\frac{1}{2}}
 \cdot\left\{\mathbf{E}\left[\big|\hat{X}_T\big|^4 \Big| \mathcal{F}_s^t\right]\right\}^{\frac{1}{2}} \\
& =\delta \left(\left|z-\bar{X}_s^{t, x ; \bar{u}}\right|^2\right),\quad \mathbf{P}\text {-a.s.}.
\end{aligned}
$$
Noting (\ref{definition of DeltaYZZ}), we have for all $r\in[s,T]$,
$$
\begin{aligned}
& g\left(r, X_r^{s, z ; \bar{u}}, Y_r^{s, z ; \bar{u}}, Z_r^{s, z ; \bar{u}}, \tilde{Z}_{(r,\cdot)}^{s,z;\bar{u}}, \bar{u}_r\right) -\bar{g}(r) \\
\equiv &\ g\left(r, X_r^{s, z ; \bar{u}}, Y_r^{s, z ; \bar{u}}, Z_r^{s, z ; \bar{u}}, \tilde{Z}_{(r,\cdot)}^{s,z;\bar{u}}, \bar{u}_r\right)
 -g\left(r, \bar{X}_r^{t, x ; \bar{u}}, \bar{Y}_r^{t, x ; \bar{u}}, \bar{Z}_r^{t, x ; \bar{u}}, \bar{\tilde{Z}}_{(r,\cdot)}^{t,x;\bar{u}},\bar{u}_r\right) \\
=&\ g\left(r,\bar{X}_r^{t, x ; \bar{u}}+\hat{X}_r,\Delta{Y}_r+\hat{Y}_r,\Delta{Z}_r+\hat{Z}_r,\Delta{\tilde{Z}}_{(r,\cdot)}+\hat{\tilde{Z}}_{(r,\cdot)},\bar{u}_r\right) \\
& -g\left(r,\bar{X}_r^{t, x ; \bar{u}}+\hat{X}_r,\bar{Y}_r^{t, x ; \bar{u}}+\hat{Y}_r,\bar{Z}_r^{t, x ; \bar{u}}+\hat{Z}_r,\bar{\tilde{Z}}_{(r,\cdot)}^{t,x;\bar{u}}+\hat{\tilde{Z}}_{(r,\cdot)},\bar{u}_r\right) \\
& +g\left(r,\bar{X}_r^{t, x ; \bar{u}}+\hat{X}_r,\bar{Y}_r^{t, x ; \bar{u}}+\hat{Y}_r,\bar{Z}_r^{t, x ; \bar{u}}+\hat{Z}_r,\bar{\tilde{Z}}_{(r,\cdot)}^{t,x;\bar{u}}+\hat{\tilde{Z}}_{(r,\cdot)},\bar{u}_r\right) \\
& -g\left(r, \bar{X}_r^{t, x ; \bar{u}}, \bar{Y}_r^{t, x ; \bar{u}}, \bar{Z}_r^{t, x ; \bar{u}}, \bar{\tilde{Z}}_{(r,\cdot)}^{t,x;\bar{u}},\bar{u}_r\right) \\
=&\ \tilde{g}_y(r,\theta)\left(\Delta{Y}_r-\bar{Y}_r^{t, x ; \bar{u}}\right)+\tilde{g}_z(r,\theta)\left(\Delta{Z}_r-\bar{Z}_r^{t, x ; \bar{u}}\right) \\
& +\tilde{g}_{\tilde{z}}(r,\theta)\int_{\mathcal{E}}\left(\Delta{\tilde{Z}}_{(r,e)}-\bar{\tilde{Z}}_{(r,e)}^{t, x ; \bar{u}}\right)\nu(de) +\bar{g}_x(r)\hat{X}_r+\bar{g}_y(r)\hat{Y}_r+\bar{g}_z(r)\hat{Z}_r\\
& +\bar{g}_{\tilde{z}}(r)\int_{\mathcal{E}}\hat{\tilde{Z}}_{(r,\cdot)}\nu(d e)
 +\left[\hat{X}_r^\top,\hat{Y}_r,\hat{Z}_r,\int_{\mathcal{E}}\hat{\tilde{Z}}_{(r,\cdot)}\nu(d e)\right]\tilde{D}^2g(r)\left[\hat{X}_r^\top,\hat{Y}_r,\hat{Z}_r,\int_{\mathcal{E}}\hat{\tilde{Z}}_{(r,\cdot)}\nu(d e)\right]^\top,
\end{aligned}
$$
where
$$
\left\{\begin{aligned}
\tilde{g}_y(r,\theta):=& \int_0^1g_y\left(r,\bar{X}_r^{t, x ; \bar{u}}+\hat{X}_r,\bar{Y}_r^{t, x ; \bar{u}}+\hat{Y}_r+\theta\left(\Delta{Y}_r-\bar{Y}_r^{t, x ; \bar{u}}\right),\bar{Z}_r^{t, x ; \bar{u}}+\hat{Z}_r\right. \\
& \qquad \left.+\theta\left(\Delta{Z}_r-\bar{Z}_r^{t, x ; \bar{u}}\right),\bar{\tilde{Z}}_{(r,\cdot)}^{t,x;\bar{u}}+\hat{\tilde{Z}}_{(r,\cdot)}
 +\theta\left(\Delta{\tilde{Z}}_{(r,\cdot)}-\bar{\tilde{Z}}_{(r,\cdot)}^{t, x ; \bar{u}}\right),\bar{u}_r\right)d \theta, \\
\tilde{g}_z(r,\theta):=& \int_0^1g_z\left(r,\bar{X}_r^{t, x ; \bar{u}}+\hat{X}_r,\bar{Y}_r^{t, x ; \bar{u}}+\hat{Y}_r+\theta\left(\Delta{Y}_r-\bar{Y}_r^{t, x ; \bar{u}}\right),\bar{Z}_r^{t, x ; \bar{u}}+\hat{Z}_r\right. \\
& \qquad \left.+\theta\left(\Delta{Z}_r-\bar{Z}_r^{t, x ; \bar{u}}\right),\bar{\tilde{Z}}_{(r,\cdot)}^{t,x;\bar{u}}+\hat{\tilde{Z}}_{(r,\cdot)}
 +\theta\left(\Delta{\tilde{Z}}_{(r,\cdot)}-\bar{\tilde{Z}}_{(r,\cdot)}^{t, x ; \bar{u}}\right),\bar{u}_r\right)d \theta, \\
\tilde{g}_{\tilde{z}}(r,\theta):=& \int_0^1g_{\tilde{z}}\left(r,\bar{X}_r^{t, x ; \bar{u}}+\hat{X}_r,\bar{Y}_r^{t, x ; \bar{u}}+\hat{Y}_r
 +\theta\left(\Delta{Y}_r-\bar{Y}_r^{t, x ; \bar{u}}\right),\bar{Z}_r^{t, x ; \bar{u}}+\hat{Z}_r\right. \\
& \qquad \left.+\theta\left(\Delta{Z}_r-\bar{Z}_r^{t, x ; \bar{u}}\right),\bar{\tilde{Z}}_{(r,\cdot)}^{t,x;\bar{u}}+\hat{\tilde{Z}}_{(r,\cdot)}
 +\theta\left(\Delta{\tilde{Z}}_{(r,\cdot)}-\bar{\tilde{Z}}_{(r,\cdot)}^{t, x ; \bar{u}}\right),\bar{u}_r\right)d \theta, \\
\tilde{D}^2g(r):=& \int_0^1\int_0^1 \lambda D^2g\left(r,\bar{X}_r^{t, x ; \bar{u}}+\lambda\theta\hat{X}_r,\bar{Y}_r^{t, x ; \bar{u}}+\lambda\theta\hat{Y}_r,\right.\\
&\qquad\qquad \left.\bar{Z}_r^{t, x ; \bar{u}}+\lambda\theta\hat{Z}_r,\bar{\tilde{Z}}_{(r,\cdot)}^{t, x ; \bar{u}}+\lambda\theta\hat{\tilde{Z}}_{(r,\cdot)},\bar{u}_r\right)d\lambda d\theta.
\end{aligned}\right.
$$
Using the definition of $\hat{Y}_r, \hat{Z}_r$ and $\hat{\tilde{Z}}_{(r,e)}$, and denoting the ($n+3$)-dimensional random vectors as, for $r\in[s,T]$,
$$
\begin{aligned}
\Psi_1(r):=& \left[\hat{X}_r^\top,\hat{Y}_r,\hat{Z}_r,\int_{\mathcal{E}}\hat{\tilde{Z}}_{(r,\cdot)}\nu(d e)\right], \\
\Psi_2(r):=& \left[\hat{X}_r^\top,\hat{X}_r^\top p_r,\hat{X}_r^\top\left(\bar{\sigma}_x(r)^\top p_r
 +q_r\right),\hat{X}_r^\top \int_{\mathcal{E}}\left(\bar{f}_x(r,e)^\top p_r+\tilde{q}_{(r,e)}+\bar{f}_x(r,e)^\top \tilde{q}_{(r,e)}\right)\nu(de)\right],
\end{aligned}
$$
we have
\begin{equation}
\begin{aligned}
& \Delta{Y}_s-\bar{Y}_s^{t, x ; \bar{u}}= \delta\left(|z-\bar{X}_s^{t, x ; \bar{u}}|^2\right)+\int_s^T\bigg\{\tilde{g}_y(r,\theta)\left(\Delta{Y}_r-\bar{Y}_r^{t, x ; \bar{u}}\right)
 +\tilde{g}_z(r,\theta)\left(\Delta{Z}_r-\bar{Z}_r^{t, x ; \bar{u}}\right) \\
& \quad +\tilde{g}_{\tilde{z}}(r,\theta)\int_{\mathcal{E}}\left(\Delta{\tilde{Z}}_{(r,e)}-\bar{\tilde{Z}}_{(r,e)}^{t, x ; \bar{u}}\right)\nu(de)+\Psi_1(r)\tilde{D}^2g(r)\Psi_1(r)^\top
 -\frac{1}{2}\Psi_2(r)D^2\bar{g}(r)\Psi_2(r)^\top \\
& \quad +C_1(r)\bigg\}dr-\int_s^T\left(\Delta{Z}_r-\bar{Z}_r^{t, x ; \bar{u}}\right)d W_r
 -\int_s^T\int_{\mathcal{E}}\left(\Delta{\tilde{Z}}_{(r,e)}-\bar{\tilde{Z}}_{(r,e)}^{t, x ; \bar{u}}\right)\tilde{N}(d e,d r),\  \mathbf{P}\text{-a.s.}.
\end{aligned}
 \label{DeltaY-barY}
\end{equation}
where
$$
\begin{aligned}
C_1(r):=&\ \langle p_r,\varepsilon_{z 4}(r)\rangle+\langle q_r,\varepsilon_{z 5}(r)\rangle+\bar{g}_z(r)\langle p_r,\varepsilon_{z 5}(r)\rangle
 +\int_{\mathcal{E}}\left\langle \tilde{q}_{(r,e)},\varepsilon_{z 6}(r,e)\right\rangle\nu(de) \\
& +\frac{1}{2}\operatorname{tr}\Big\{P_r\varepsilon_{z 7}(r)+Q_r\varepsilon_{z 8}(r)+\bar{g}_z(r)P_r\varepsilon_{z 8}(r)\Big\}
 +\int_{\mathcal{E}}\frac{1}{2}\operatorname{tr}\Big\{\tilde{Q}_{(r,e)}\varepsilon_{z 9}(r,e)\Big\}\nu(de) \\
& +\int_{\mathcal{E}}\bigg\{\bar{g}_{\tilde{z}}(r)\langle p_r,\varepsilon_{z 6}(r,e)\rangle+\bar{g}_{\tilde{z}}(r)\left\langle \tilde{q}_{(r,e)},\varepsilon_{z 6}(r,e)\right\rangle \\
& \qquad +\frac{1}{2}\operatorname{tr}\left\{\bar{g}_{\tilde{z}}(r)P_r\varepsilon_{z 9}(r,e)+\bar{g}_{\tilde{z}}(r)\tilde{Q}_{(r,e)}\varepsilon_{z 9}(r,e)\right\}\bigg\}\nu(de).
\end{aligned}
$$

Step 5. Estimates of remainder terms of BSDEPs.

Noting that for all $r\in[s,T]$,
$$
\begin{aligned}
& \Psi_1(r)\tilde{D}^2g(r)\Psi_1(r)^\top-\frac{1}{2}\Psi_2(r)D^2\bar{g}(r)\Psi_2(r)^\top \\
\equiv&\ \Psi_1(r)\tilde{D}^2g(r)\Psi_1(r)^\top-\Psi_2(r)\tilde{D}^2g(r)\Psi_2(r)^\top+\Psi_2(r)\tilde{D}^2g(r)\Psi_2(r)^\top-\frac{1}{2}\Psi_2(r)D^2\bar{g}(r)\Psi_2(r)^\top \\
:=&\ \Pi(r)\tilde{D}^2g(r)\Pi(r)^\top +\Psi_2(r)\Big[\tilde{D}^2g(r)-\frac{1}{2}D^2\bar{g}(r)\Big]\Psi_2(r)^\top,
\end{aligned}
$$
where $\Pi(r):=\Psi_1(r)-\Psi_2(r)$, we can get
\begin{equation}
\left\{\begin{aligned}
& \mathbf{E}\left[\left(\int_s^T\left|C_1(r)\right| d r\right)^2 \bigg| \mathcal{F}_s^t\right]\leqslant\delta\left(\left|z-\bar{X}^{t, x ; \bar{u}}_s\right|^4\right),\ \mathbf{P} \text {-a.s.}, \\
& \mathbf{E}\left[\left(\int_s^T\left|\Pi(r) \tilde{D}^2 g(r) \Pi(r)^\top\right| d r\right)^2 \bigg| \mathcal{F}_s^t\right]\leqslant\delta\left(\left|z-\bar{X}^{t, x ; \bar{u}}_s\right|^4\right),\ \mathbf{P} \text {-a.s.}, \\
& \mathbf{E}\left[\left(\int_s^T\left|\Psi_2(r)\left[\tilde{D}^2 g(r)-\frac{1}{2} D^2 \bar{g}(r)\right] \Psi_2(r)^ \top \right| d r\right)^2 \bigg| \mathcal{F}_s^t\right]
 \leqslant \delta\left(\left|z-\bar{X}^{t, x ; \bar{u}}_s\right|^4\right),\ \mathbf{P} \text {-a.s.}.
\end{aligned}\right.
 \label{estimates of remainder terms of BSDEPs}
\end{equation}
Indeed, by the boundedness of $g_z, g_{\tilde{z}}$, we obtain
$$
\begin{aligned}
& \mathbf{E}\left[\left(\int_s^T\left|C_1(r)\right| d r\right)^2 \bigg| \mathcal{F}_s^t\right] \\
\leqslant &\ C\mathbf{E}\left[\left(\int_s^T\langle p_r,\varepsilon_{z 4}(r)\rangle d r\right)^2 \bigg| \mathcal{F}_s^t\right]
 +C\mathbf{E}\left[\left(\int_s^T\langle q_r,\varepsilon_{z 5}(r)\rangle d r\right)^2 \bigg| \mathcal{F}_s^t\right] \\
& +C\mathbf{E}\left[\left(\int_s^T\langle p_r,\varepsilon_{z 5}(r)\rangle d r\right)^2 \bigg| \mathcal{F}_s^t\right]
 +C\mathbf{E}\left[\left(\int_s^T\int_{\mathcal{E}}\left\langle \tilde{q}_{(r,e)},\varepsilon_{z 6}(r,e)\right\rangle\nu(de) d r\right)^2 \bigg| \mathcal{F}_s^t\right] \\
& +C\mathbf{E}\left[\left(\int_s^T\operatorname{tr}\Big\{P_r\varepsilon_{z 7}(r)\Big\} d r\right)^2 \bigg| \mathcal{F}_s^t\right]
 +C\mathbf{E}\left[\left(\int_s^T\operatorname{tr}\Big\{Q_r\varepsilon_{z 8}(r)\Big\} d r\right)^2 \bigg| \mathcal{F}_s^t\right] \\
& +C\mathbf{E}\left[\left(\int_s^T\operatorname{tr}\Big\{P_r\varepsilon_{z 8}(r)\Big\} d r\right)^2 \bigg| \mathcal{F}_s^t\right]
 +C\mathbf{E}\left[\left(\int_s^T\int_{\mathcal{E}}\operatorname{tr}\Big\{\tilde{Q}_{(r,e)}\varepsilon_{z 9}(r,e)\Big\}\nu(de) d r\right)^2 \bigg| \mathcal{F}_s^t\right] \\
& +C\mathbf{E}\left[\left(\int_s^T\int_{\mathcal{E}}\langle p_r,\varepsilon_{z 6}(r,e)\rangle \nu(de)d r\right)^2 \bigg| \mathcal{F}_s^t\right]
 +C\mathbf{E}\left[\left(\int_s^T\int_{\mathcal{E}}\langle \tilde{q}_{(r,e)},\varepsilon_{z 6}(r,e)\rangle \nu(de)d r\right)^2 \bigg| \mathcal{F}_s^t\right] \\
& +C\mathbf{E}\left[\left(\int_s^T\int_{\mathcal{E}}\operatorname{tr}\Big\{P_r\varepsilon_{z 9}(r,e)\Big\} d r\right)^2 \bigg| \mathcal{F}_s^t\right]
 +C\mathbf{E}\left[\left(\int_s^T\int_{\mathcal{E}}\operatorname{tr}\Big\{\tilde{Q}_{(r,e)}\varepsilon_{z 9}(r,e)\Big\} \nu(de)d r\right)^2 \bigg| \mathcal{F}_s^t\right] \\
:= & I_1+I_2+I_3+I_4+I_5+I_6+I_7+I_8+I_9+I_{10}+I_{11}+I_{12}.
\end{aligned}
$$
We give the estimates of $\varepsilon_{z 7}(r),\varepsilon_{z 8}(r)$ and $\varepsilon_{z 9}(r,e)$ defined by (\ref{definition of z 7-9}) in the first. By the continuity and uniformly boundedness of $\sigma_x, f_x$, we have
$$
\begin{aligned}
& \mathbf{E}\left[\int_s^T\left|\varepsilon_{z 7}(r)\right|^2 d r \bigg| \mathcal{F}_s^t\right] \leqslant \int_s^T\mathbf{E}\left[\left|\varepsilon_{z 7}(r)\right|^2 \bigg| \mathcal{F}_s^t\right]d r \\
\leqslant &\ C\int_s^T\bigg\{\mathbf{E}\left[\left|\varepsilon_{z 1}(r)\right|^2|\hat{X}(r)|^2 \bigg| \mathcal{F}_s^t\right]+\mathbf{E}\left[\left|\varepsilon_{z 2}(r)\right|^2|\hat{X}(r)|^2 \Big| \mathcal{F}_s^t\right] \\
&\ + \mathbf{E}\left[\left|\varepsilon_{z 2}(r)\right|^4 \Big| \mathcal{F}_s^t\right]+\mathbf{E}\left[\left\|\varepsilon_{z 3}(r,e)\right\|_{\mathcal{L}^2}^2|\hat{X}(r)|^2 \Big| \mathcal{F}_s^t\right]
 +\mathbf{E}\left[\left\|\varepsilon_{z 3}(r,e)\right\|_{\mathcal{L}^2}^4 \Big| \mathcal{F}_s^t\right]\bigg\}d r \\
\leqslant &\ C\bigg\{\Big\{\mathbf{E}\left[\left|\varepsilon_{z 1}(r)\right|^4 \Big| \mathcal{F}_s^t\right]\Big\}^{\frac{1}{2}}\cdot\Big\{\mathbf{E}\left[|\hat{X}(r)|^4 \Big| \mathcal{F}_s^t\right]\Big\}^{\frac{1}{2}}
  +\Big\{\mathbf{E}\left[\left|\varepsilon_{z 2}(r)\right|^4 \Big| \mathcal{F}_s^t\right]\Big\}^{\frac{1}{2}}\cdot\Big\{\mathbf{E}\left[|\hat{X}(r)|^4 \Big| \mathcal{F}_s^t\right]\Big\}^{\frac{1}{2}} \\
&\ + \mathbf{E}\left[\left|\varepsilon_{z 2}(r)\right|^4 \Big| \mathcal{F}_s^t\right]
 +\Big\{\mathbf{E}\left[\left\|\varepsilon_{z 3}(r,3)\right\|_{\mathcal{L}^2}^4 \Big| \mathcal{F}_s^t\right]\Big\}^{\frac{1}{2}}\cdot\Big\{\mathbf{E}\left[|\hat{X}(r)|^4 \Big| \mathcal{F}_s^t\right]\Big\}^{\frac{1}{2}}\\
&\ +\mathbf{E}\left[\left\|\varepsilon_{z 3}(r,e)\right\|_{\mathcal{L}^2}^4 \Big| \mathcal{F}_s^t\right]\bigg\}d r. \\
\end{aligned}
$$
Note that
$$
\begin{aligned}
&\mathbf{E}\left[\int_s^T\left|\varepsilon_{z 1}(r)\right|^4 d r \bigg| \mathcal{F}_s^t\right]
\leqslant \int_s^T \mathbf{E}\left\{\int_0^1\left|b_x(r, \theta)-\bar{b}_x(r)\right|^4 d \theta \cdot\left|\hat{X}_r\right|^4 \bigg| \mathcal{F}_s^t\right\} d r \\
&  \leqslant C \int_s^T \mathbf{E}\left[\big|\hat{X}_r\big|^8 \Big| \mathcal{F}_s^t\right] d r \leqslant C\left|z-\bar{X}^{t, x ; \bar{u}}_s\right|^8 = \delta \left(\left|z-\bar{X}^{t, x ; \bar{u}}_s\right|^4\right),
\end{aligned}
$$
as well as
$$
\begin{aligned}
&\mathbf{E}\left[\int_s^T\left\|\varepsilon_{z 3}(r, e)\right\|_{\mathcal{L}^2}^4 d r \bigg| \mathcal{F}_s^t\right]
 \leqslant \int_s^T \mathbf{E}\left\{\int_0^1\left\|f_x(r, \theta, e)-\bar{f}_x(r, e)\right\|_{\mathcal{L}^2}^4 d \theta \cdot\big|\hat{X}_{r-}\big|^4 \bigg| \mathcal{F}_s^t\right\} d r \\
& \leqslant C \int_s^T \mathbf{E}\left[\big|\hat{X}_{r-}\big|^8 \Big| \mathcal{F}_s^t\right] d r = C \int_s^T \mathbf{E}\left[\big|\hat{X}_r\big|^8 \Big| \mathcal{F}_s^t\right] d r
 \leqslant C\left|z-\bar{X}^{t, x ; \bar{u}}_s\right|^8
 = \delta \left(\left|z-\bar{X}^{t, x ; \bar{u}}_s\right|^4\right),
\end{aligned}
$$
and similarly for other estimates of remainder terms. Also by (\ref{estimate of sup X}), we get
\begin{equation}
\mathbf{E}\left[\int_s^T\left|\varepsilon_{z 7}(r)\right|^2 d r \bigg| \mathcal{F}_s^t\right] \leqslant \delta \left(\left|z-\bar{X}^{t, x ; \bar{u}}_s\right|^4\right).
 \label{estimate of z 7}
\end{equation}
We can get the following estimates by the similar proof.
\begin{equation}
\begin{aligned}
\mathbf{E}\left[\int_s^T\left|\varepsilon_{z 8}(r)\right|^2 d r \bigg| \mathcal{F}_s^t\right] & \leqslant \delta \left(\left|z-\bar{X}^{t, x ; \bar{u}}_s\right|^4\right), \\
\mathbf{E}\left[\int_s^T\left\|\varepsilon_{z 9}(r,e)\right\|_{\mathcal{L}^2}^2 d r \bigg| \mathcal{F}_s^t\right] & \leqslant \delta \left(\left|z-\bar{X}^{t, x ; \bar{u}}_s\right|^4\right).
\end{aligned}
 \label{estimate of z 8-9}
\end{equation}
In what follows, by (\ref{estimate of z 1-6}), (\ref{estimate of z 7}), (\ref{estimate of z 8-9}), we have
$$
\begin{aligned}
I_1 & \leqslant C \mathbf{E}\left[\sup _{s \leqslant r \leqslant T}|p_r|^2 \int_s^T\left|\varepsilon_{z 4}(r)\right|^2 d r \bigg| \mathcal{F}_s^t\right]
 \leqslant C \mathbf{E}\left[\int_s^T\left|\varepsilon_{z 4}(r)\right|^2 d r \bigg| \mathcal{F}_s^t\right] \\
& =\delta\left(\left|z-\bar{X}^{t, x ; \bar{u}}_s\right|^4\right),\ \mathbf{P} \text {-a.s.},
\end{aligned}
$$
and similarly for $I_3, I_9$; by H\"{o}lder's inequality, we get
$$
\begin{aligned}
I_2 \leqslant &\ C \mathbf{E}\left[\int_s^T|q_r|^2 d r \int_s^T\left|\varepsilon_{z 5}(r)\right|^2 d r \bigg| \mathcal{F}_s^t\right] \\
\leqslant &\ C\left\{\mathbf{E}\left[\left(\int_s^T|q_r|^2 d r\right)^2 \bigg| \mathcal{F}_s^t\right]\right\}^{\frac{1}{2}}
 \cdot\left\{\mathbf{E}\left[\int_s^T\left|\varepsilon_{z 5}(r)\right|^4 d r \bigg| \mathcal{F}_s^t\right]\right\}^{\frac{1}{2}} \\
\leqslant &\ C \left\{\mathbf{E}\left[\int_s^T\left|\varepsilon_{z 5}(r)\right|^4 d r \bigg| \mathcal{F}_s^t\right]\right\}^{\frac{1}{2}}
 = \delta\left(\left|z-\bar{X}^{t, x ; \bar{u}}_s\right|^4\right),\ \mathbf{P}\text {-a.s.},
\end{aligned}
$$
and similarly for $I_4, I_{10}$; by (\ref{estimate of z 7}), we obtain
$$
\begin{aligned}
I_5 & \leqslant C \mathbf{E}\left[\sup _{s \leqslant r \leqslant T}|P_r|^2 \int_s^T\left|\varepsilon_{z 7}(r)\right|^2 d r \bigg| \mathcal{F}_s^t\right]
 \leqslant C \mathbf{E}\left[\int_s^T\left|\varepsilon_{z 7}(r)\right|^2 d r \bigg| \mathcal{F}_s^t\right] \\
& =\delta\left(\left|z-\bar{X}^{t, x ; \bar{u}}_s\right|^4\right),\ \mathbf{P}\text {-a.s.},
\end{aligned}
$$
and similarly for $I_7, I_{11}$; by (\ref{estimate of z 8-9}), we obtain
$$
\begin{aligned}
I_6 \leqslant &\ C \mathbf{E}\left[\int_s^T|Q_r|^2 d r \int_s^T\left|\varepsilon_{z 8}(r)\right|^2 d r \bigg| \mathcal{F}_s^t\right] \\
\leqslant &\ C\left\{\mathbf{E}\left[\left(\int_s^T|Q_r|^2 d r\right)^2 \bigg| \mathcal{F}_s^t\right]\right\}^{\frac{1}{2}}
 \cdot\left\{\mathbf{E}\left[\int_s^T\left|\varepsilon_{z 8}(r)\right|^4 d r \bigg| \mathcal{F}_s^t\right]\right\}^{\frac{1}{2}} \\
\leqslant &\ C \left\{\mathbf{E}\left[\int_s^T\left|\varepsilon_{z 8}(r)\right|^4 d r \bigg| \mathcal{F}_s^t\right]\right\}^{\frac{1}{2}}
 = \delta\left(\left|z-\bar{X}^{t, x ; \bar{u}}_s\right|^4\right),\ \mathbf{P}\text {-a.s.},
\end{aligned}
$$
and similarly for $I_8, I_{12}$. Combining the above estimates, we obtain that the first inequality of (\ref{estimates of remainder terms of BSDEPs}) holds.

For the second one, since
$$
\begin{aligned}
 \Pi(r):= &\ \Psi_1(r)-\Psi_2(r) \\
:= &\ \Big[O_{n\times n},\hat{Y}_r-\hat{X}_r^\top p_r,\hat{Z}_r-\hat{X}_r^\top \left(\bar{\sigma}_x(r)^\top p_r+q_r\right), \\
& \quad \hat{\tilde{Z}}_r-\hat{X}_r^\top\int_{\mathcal{E}}\left(\bar{f}_x(r,e)^\top p_r+\tilde{q}_{(r,e)}+\bar{f}_x(r,e)^\top \tilde{q}_{(r,e)}\right)\nu(de)\Big] \\
= &\ \bigg[O_{n\times n},\frac{1}{2}\left\langle P_r\hat{X}_r,\hat{X}_r\right\rangle, \Big\langle p_r,\frac{1}{2}\hat{X}_r^\top \bar{\sigma}_{x x}(r)\hat{X}_r+\varepsilon_{z 5}(r)\Big\rangle \\
& \quad +\frac{1}{2}\operatorname{tr}\Big\{\Phi_rQ_r+P_r\Phi_r\bar{\sigma}_x(r)^\top+P_r\bar{\sigma}_x(r)\Phi_r+P_r\varepsilon_{z 8}(r)\Big\}, \\
& \quad \int_{\mathcal{E}}\bigg\{\Big\langle p_r,\frac{1}{2}\hat{X}_{r-}^\top \bar{f}_{x x}(r,e)\hat{X}_{r-}+\varepsilon_{z 6}(r,e)\Big\rangle
 +\Big\langle \tilde{q}_{(r,e)},\frac{1}{2}\hat{X}_{r-}^\top \bar{f}_{x x}(r,e)\hat{X}_{r-}+\varepsilon_{z 6}(r,e)\Big\rangle \nu(de) \\
& \qquad +\frac{1}{2}\operatorname{tr}\Big\{\Phi_{r-}\tilde{Q}_{(r,e)}+P_r\bar{f}_x(r,e)\Phi_{r-}+P_r\Phi_{r-}\bar{f}_x(r,e)^\top \\
& \qquad +P_r\bar{f}_x(r,e)\Phi_{r-}\bar{f}_x(r,e)^\top +\tilde{Q}_{(r,e)}\bar{f}_x(r,e)\Phi_{r-}+\tilde{Q}_{(r,e)}\Phi_{r-}\bar{f}_x(r,e)^\top \\
& \qquad +\tilde{Q}_{(r,e)}\bar{f}_x(r,e)\Phi_{r-}\bar{f}_x(r,e)^\top +P_r\varepsilon_{z 9}(r, e)+\tilde{Q}_{(r,e)}\varepsilon_{z 9}(r, e)\Big\}\bigg\}\nu(de)\Big], \\
\end{aligned}
$$
by the definitions of $\varepsilon_{z 5}(r), \varepsilon_{z 6}(r,e), \varepsilon_{z 8}(r), \varepsilon_{z 9}(r,e)$, the boundedness of $\bar{\sigma}_x, \bar{\sigma}_{x x}, \tilde{D}^2g$, and the square-integrability of $P_\cdot, Q_\cdot, \tilde{Q}_{(\cdot,e)}$, we have
$$
\begin{aligned}
& \mathbf{E}\left[\left(\int_s^T\left|\Pi(r) \tilde{D}^2 g(r) \Pi(r)^\top\right| d r\right)^2 \bigg| \mathcal{F}_s^t\right]
 \leqslant C \mathbf{E}\left[\left(\int_s^T\left|\Pi(r)^2\right| d r\right)^2 \bigg| \mathcal{F}_s^t\right] \\
& \leqslant C \mathbf{E}\Bigg[\bigg(\int _s^T\Big[|\hat{X}_r|^4\left(|P_r|^2+|Q_r|^2+\|\tilde{Q}_{(r,e)}\|_{\mathcal{L}^2}^2+|p_r|^2+\|q_{(r,e)}\|_{\mathcal{L}^2}^2\right)\\
& \quad +|p_r|^2\left(\left|\varepsilon_{z 5}(r)\right|^2+\|\varepsilon_{z 6}(r,e)\|_{\mathcal{L}^2}^2\right) +\|\tilde{q}_{(r,e)}\|_{\mathcal{L}^2}^2\|\varepsilon_{z 6}(r,e)\|_{\mathcal{L}^2}^2\\
& \quad +|P_r|^2\left(\left|\varepsilon_{z 8}(r)\right|^2+\|\varepsilon_{z 9}(r,e)\|_{\mathcal{L}^2}^2\right)
 +\|\tilde{Q}_{(r,e)}\|_{\mathcal{L}^2}^2\|\varepsilon_{z 9}(r,e)\|_{\mathcal{L}^2}^2\Big] d r\bigg)^2 \bigg| \mathcal{F}_s^t\Bigg] \\
& \leqslant C \mathbf{E}\left[\left(\int_s^T\left[|\hat{X}_r|^4\left(|P_r|^2+|Q_r|^2+\|\tilde{Q}_{(r,e)}\|_{\mathcal{L}^2}^2+|p_r|^2
 +\|q_{(r,e)}\|_{\mathcal{L}^2}^2\right)\right] d r\right)^2 \bigg| \mathcal{F}_s^t\right]\\
& \quad +\delta\left(\left|z-\bar{X}^{t, x ; \bar{u}}_s\right|^4\right) \\
& \leqslant C\left(\mathbf{E}\left[\sup _{s \leqslant r \leqslant T}|\hat{X}_r|^{16} \bigg| \mathcal{F}_s^t\right]\right)^{\frac{1}{2}}+\delta\left(\left|z-\bar{X}^{t, x ; \bar{u}}_s\right|^4\right)
 \leqslant\delta\left(\left|z-\bar{X}^{t, x ; \bar{u}}_s\right|^4\right),\ \mathbf{P}\text {-a.s.}.
\end{aligned}
$$
Thus the second inequality of (\ref{estimates of remainder terms of BSDEPs}) holds. And then, we prove the last one.
$$
\begin{aligned}
\mathbf{E} & {\left[\left(\int_s^T\left|\Psi_2(r)\left[\tilde{D}^2 g(r)-\frac{1}{2} D^2 \bar{g}(r)\right] \Psi_2(r)^\top \right| d r\right)^2 \bigg| \mathcal{F}_s^t\right] } \\
\leqslant &\ C \mathbf{E}\left[\left(\int_s^T|\hat{X}_r|^2\left(1+|p_r|^2+\left|\bar{\sigma}_x^\top (r) p_r+q_r\right|^2+\left\|\left(\bar{f}_x(r,e)^\top p_r+\tilde{q}_{(r,e)}\right.\right.\right.\right.\right.\\
&\qquad \left.\left.\left.\left.\left.+\bar{f}_x(r,e)^\top \tilde{q}_{(r,e)}\right)\right\|_{\mathcal{L}^2}^2\right)\left|\tilde{D}^2 g(r)-\frac{1}{2} D^2 \bar{g}(r) \right| d r\right)^2 \bigg| \mathcal{F}_s^t\right] \\
\leqslant &\ C \mathbf{E}\left[\sup_{s \leqslant r \leqslant T}|\hat{X}_r|^4\left(\int_s^T\left(1+|p_r|^2+\left|q_r\right|^2+\left\|\tilde{q}_{(r,e)}\right\|_{\mathcal{L}^2}^2\right)\left|\tilde{D}^2 g(r)-\frac{1}{2} D^2 \bar{g}(r) \right| d r\right)^2 \bigg| \mathcal{F}_s^t\right] \\
\leqslant &\ C\left\{\mathbf{E}\left[\sup _{s \leqslant r \leqslant T}|\hat{X}_r|^8 \bigg| \mathcal{F}_s^t\right]\right\}^{\frac{1}{2}}\cdot
 \left\{\mathbf{E}\left[\left(\int_s^T\left(1+|p_r|^2+\left|q_r\right|^2+\left\|\tilde{q}_{(r,e)}\right\|_{\mathcal{L}^2}^2\right)\right.\right.\right.\\
&\qquad \left.\left.\left.\cdot\left|\tilde{D}^2 g(r)-\frac{1}{2} D^2 \bar{g}(r)\right| d r\right)^4 \bigg| \mathcal{F}_s^t\right]\right\}^{\frac{1}{2}} \\
\leqslant &\ C\left\{\mathbf{E}\left[\sup _{s \leqslant r \leqslant T}|\hat{X}_r|^8 \bigg| \mathcal{F}_s^t\right]\right\}^{\frac{1}{2}} \cdot \left\{\mathbf{E}\left[\left(\int_s^T\left(1+|p_r|^2+\left|q_r\right|^2+\left\|\tilde{q}_{(r,e)}\right\|_{\mathcal{L}^2}^2\right)^2 d r\right)^4 \bigg| \mathcal{F}_s^t\right]\right\}^{\frac{1}{4}} \\
&\quad \cdot\left\{\mathbf{E}\left[\left(\int_s^T\left|\tilde{D}^2 g(r)-\frac{1}{2} D^2 \bar{g}(r)\right|^2 d r\right)^4 \bigg| \mathcal{F}_s^t\right]\right\}^{\frac{1}{4}},\ \mathbf{P}\text {-a.s.}.
\end{aligned}
$$
Since $p_\cdot, q_\cdot, \tilde{q}_{(\cdot,e)}$ are square-integrable, by the definition of $\tilde{D}^2 g, \hat{Y}_\cdot, \hat{Z}_\cdot$ and $\hat{\tilde{Z}}_{(\cdot,e)}$ and the modulus continuity of $D^2 g$, we obtain the last inequality of (\ref{estimates of remainder terms of BSDEPs}).

By (\ref{DeltaY-barY}), (\ref{estimates of remainder terms of BSDEPs}), and Lemma \ref{Lemma 4.3}, we have
\begin{equation}
\begin{aligned}
& \left|\Delta{Y}_s -\bar{Y}^{t, x ; \bar{u}}_s \right|^2 \leqslant C \mathbf{E}\left[\left(\int_s^T\left|C_1(r)\right| d r\right)^2 \bigg| \mathcal{F}_s^t\right] \\
&\quad +C \mathbf{E}\left[\left(\int_s^T\left|\Psi_1(r) \tilde{D}^2 g(r) \Psi_1(r)^\top -\frac{1}{2} \Psi_2(r) D^2 \bar{g}(r) \Psi_2(r)^\top \right| d r\right)^2 \bigg| \mathcal{F}_s^t\right] \\
&\quad +\delta\left(\left|z-\bar{X}^{t, x ; \bar{u}}_s\right|^4\right) \leqslant \delta\left(\left|z-\bar{X}^{t, x ; \bar{u}}_s\right|^4\right),\ \mathbf{P}\text {-a.s.}.
\end{aligned}
 \label{estimate of DeltaY-barY}
\end{equation}

Step 6. Completion of the proof.

Since the set of all rational vectors $z \in \mathbf{R}^n$ is countable, we can find a subset $\Omega_0 \subseteq \Omega$ with $\mathbf{P}\left(\Omega_0\right)=1$ such that for any $\omega_0 \in \Omega_0$,
$$
\left\{\begin{array}{l}
V\left(s, \bar{X}^{t, x ; \bar{u}}_{(s, \omega_0)}\right)=-\bar{Y}^{t, x ; \bar{u}}_{(s, \omega_0)},\ (\ref{estimate of sup X}),(\ref{estimate of z 1-6}), (\ref{definition of hatY}), (\ref{DeltaY-barY}), (\ref{estimates of remainder terms of BSDEPs}), (\ref{estimate of DeltaY-barY}) \text{ hold for}\\
 \text { any rational vector } z,\ \left(\Omega, \mathcal{F}, \mathbf{P}(\cdot | \mathcal{F}_s^t)(\omega_0), W_\cdot-W_s, \tilde{N}(\cdot,\cdot) ; u_\cdot|_{[s, T]}\right) \in \mathcal{U}^w[s, T],\\
 \text { and } \sup\limits_{s \leqslant r \leqslant T}\left[|p_{(r, \omega_0)}|+|P_{(r, \omega_0)}|\right]<\infty.
\end{array}\right.
$$
The first relation of the above is obtained by the DPP of \cite{LiPeng2009}. Let $\omega_0 \in \Omega_0$ be fixed, and then for any rational vector $z \in \mathbf{R}^n$, by (\ref{estimate of DeltaY-barY}), we have
$$
 \Delta{Y}_{(s, \omega_0)}-\bar{Y}^{t, x ; \bar{u}}_{(s, \omega_0)} \leqslant \delta \left(\left|z-\bar{X}^{t, x ; \bar{u}}_{(s, \omega_0)}\right|^2\right), \quad \text{for all} \ s \in[t, T].
$$
By the definition of $\Delta{Y}_\cdot$ (see (\ref{definition of DeltaYZZ})), we have
$$
\begin{aligned}
Y^{s, z ; \bar{u}}_{(s, \omega_0)}-\bar{Y}^{t, x ; \bar{u}}_{(s, \omega_0)}& =  \left\langle p_{(s, \omega_0)}, (z-\bar{X}^{t, x ; \bar{u}}_{(s, \omega_0)})\right\rangle \\
&\quad +\frac{1}{2}\left\langle P_{(s, \omega_0)} (z-\bar{X}^{t, x ; \bar{u}}_{(s, \omega_0)}), (z-\bar{X}^{t, x ; \bar{u}}_{(s, \omega_0)})\right\rangle+\delta\left(\left|z-\bar{X}^{t, x ; \bar{u}}_{(s, \omega_0)}\right|^2\right),
\end{aligned}
$$
for all $s \in[t, T]$. Thus for all $s \in[t, T]$,
\begin{equation}
\begin{aligned}
& V\left(s, z\right)-V(s, \bar{X}^{t, x ; \bar{u}}_{(s, \omega_0)})\leqslant-Y^{s, z ; \bar{u}}_{(s, \omega_0)}+\bar{Y}^{t, x ; \bar{u}}_{(s, \omega_0)}
 = -\left\langle p_{(s, \omega_0)}, (z-\bar{X}^{t, x ; \bar{u}}_{(s, \omega_0)})\right\rangle\\
&\quad -\frac{1}{2}\left\langle P_{(s, \omega_0)} (z-\bar{X}^{t, x ; \bar{u}}_{(s, \omega_0)}), (z-\bar{X}^{t, x ; \bar{u}}_{(s, \omega_0)})\right\rangle
 +\delta \left(\left|z-\bar{X}^{t, x ; \bar{u}}_{(s, \omega_0)}\right|^2\right).
\end{aligned}
 \label{estimate of V-V}
\end{equation}
Note that the term $\delta(|z-\bar{X}^{t, x ; \bar{u}}_{(s, \omega_0)}|^2)$ in the above depends only on the size of $\left|z-\bar{X}^{t, x ; \bar{u}}_{(s, \omega_0)}\right|$, and it is independent of $z$. Therefore, by the continuity of $V(s, \cdot)$, we see that (\ref{estimate of V-V}) holds for all $z \in \mathbf{R}^n$ (for more similar details see \cite{Zhou1990}), which by definition (\ref{definition of partial super-subjets to x}) proves
$$
\{-p_s\} \times[-P_s, \infty) \in \mathcal{D}_x^{2,+} V\left(s, \bar{X}^{t, x ; \bar{u}}_s\right),\quad \forall s \in[t, T],\ \mathbf{P} \text {-a.s.}.
$$

Finally, fix an $\omega \in \Omega$ such that (\ref{estimate of V-V}) holds for any $z \in \mathbf{R}^n$. For any $(\hat{p}, \hat{P}) \in$ $\mathcal{D}_x^{2,-} V\left(s, \bar{X}^{t, x ; \bar{u}}_s\right)$, by definition (\ref{definition of partial super-subjets to x}) we have
$$
\begin{aligned}
0 \leqslant & \liminf _{z \rightarrow \bar{X}^{t, x ; \bar{u}}_s}  \left\{\frac{V\left(s, z\right)-V\left(s, \bar{X}^{t, x ; \bar{u}}_s\right)
 -\left\langle\hat{p}, z-\bar{X}^{t, x ; \bar{u}}_s\right\rangle-\frac{1}{2}\left\langle\hat{P}(z-\bar{X}^{t, x ; \bar{u}}_s), z-\bar{X}^{t, x ; \bar{u}}_s\right\rangle}{\left|z-\bar{X}^{t, x ; \bar{u}}_s\right|}\right\} \\
\leqslant & \liminf _{z \rightarrow \bar{X}^{t, x ; \bar{u}}_s}  \left\{-\frac{\left\langle p_s+\hat{p}, z-\bar{X}^{t, x ; \bar{u}}_s\right\rangle
 +\frac{1}{2}\left\langle\left(P_s+\hat{P}\right)\left(z-\bar{X}^{t, x ; \bar{u}}_s\right), z-\bar{X}^{t, x ; \bar{u}}_s\right\rangle}{\left|z-\bar{X}^{t, x ; \bar{u}}_s\right|}\right\}.
\end{aligned}
$$
Then, it is necessary that
$$
\hat{p}=-p_s, \ \hat{P} \leqslant-P_s, \quad \forall s \in[t, T], \ \mathbf{P} \text {-a.s.}.
$$
Thus, (\ref{relation of p,P and V}) holds. (\ref{relation of p and DV}) is immediate. The proof is complete.
\end{proof}

The following result characterizes the super- and subjets of the value function in the time variable $t$ along an optimal trajectory, with the help of an additional $\mathcal{H}$-function.

\begin{mythm}\label{nonsmooth time theorem}
Suppose (H1)-(H3) hold and let $(t,x)\in [0,T)\times \mathbf{R}^n$ be fixed. Let $(\bar{X}_\cdot^{t,x;\bar{u}},\bar{u}_\cdot)$ be an optimal pair of {\bf Problem (SROCPJ)}. Let $(p_\cdot,q_\cdot,\tilde{q}_{(\cdot,\cdot)})\in L_{\mathcal{F}}^2([t,T];\mathbf{R}^n)\times L_{\mathcal{F},p}^2([t,T];\mathbf{R}^n)\times F_p^2([t,T]\times\mathcal{E};\mathbf{R}^n)$ and $(P_\cdot,Q_\cdot,\tilde{Q}_{(\cdot,\cdot)})\in L_{\mathcal{F}}^2([t,T];\mathcal{S}^n)\times L_{\mathcal{F},p}^2([t,T];\mathcal{S}^n)\times F_p^2([t,T]\times\mathcal{E};\mathcal{S}^n)$ be the first-order and second-order adjoint processes satisfying (\ref{first-order adjoint equation}) and (\ref{second-order adjoint equation}), respectively. Then
\begin{equation}
\begin{aligned}
\left[\mathcal{H}(s,\bar{X}_s^{t,x;\bar{u}},\bar{u}_s),\infty)\right. \subseteq \mathcal{D}_{t+}^{1,+}V(s,\bar{X}_s^{t,x;\bar{u}}),\quad &\text{a.e.}\ s\in[t,T],\ \mathbf{P}\text{-a.s.}, \\
\mathcal{D}_{t+}^{1,-}V(s,\bar{X}_s^{t,x;\bar{u}}) \subseteq \left.(-\infty,\mathcal{H}(s,\bar{X}_s^{t,x;\bar{u}},\bar{u}_s)\right],\quad &\text{a.e.}\ s\in[t,T],\ \mathbf{P}\text{-a.s.},
\end{aligned}
 \label{relation of H and DV}
\end{equation}
where $\mathcal{H}:[0,T] \times \mathbf{R}^n \times \mathbf{U} \rightarrow \mathbf{R}$ is defined as
\begin{equation}
\begin{aligned}
\mathcal{H}(s,x,u):=&\ H\left(s,x,-V(s,x),\bar{\sigma}(s)^\top p_s,-\int_{\mathcal{E}}\left[V(s,x+\bar{f}(s,e))-V(s,x)\right]\nu(de);p_s,q_s,P_s,u\right)\\
&\ +\int_{\mathcal{E}}\left\langle\tilde{q}_{(s,e)},\bar{f}(s,e)\right\rangle\nu(de)-\frac{1}{2}\left\langle P_s\bar{\sigma}(s),\bar{\sigma}(s) \right\rangle \\
&\ -\frac{1}{2}\int_{\mathcal{E}} \left\langle P_s\bar{f}(s,e),\bar{f}(s,e) \right\rangle \nu(de)-\frac{1}{2}\int_{\mathcal{E}} \left\langle \tilde{Q}_{(s,e)}\bar{f}(s,e),\bar{f}(s,e) \right\rangle \nu(de)\\
\equiv&\ \left\langle p_s,b(s,x,u)\right\rangle +\left\langle q_s-P_s\bar{\sigma}(s),\sigma(s,x,u)\right\rangle +\frac{1}{2}\left\langle P_s(\sigma(s,x,u),\sigma(s,x,u)\right\rangle \\
&\ +g\left(s,x,-V(s,x),\sigma(s,x,u)^\top p_s,-\int_{\mathcal{E}}\left[V(s,x+\bar{f}(s,e))-V(s,x)\right]\nu(de),u\right) \\
&\ +\int_{\mathcal{E}}\left\langle\tilde{q}_{(s,e)}-\frac{1}{2}(P_s+\tilde{Q}_{(s,e)})\bar{f}(s,e),\bar{f}(s,e)\right\rangle \nu(de)-\frac{1}{2}\left\langle P_s\bar{\sigma}(s),\bar{\sigma}(s) \right\rangle \\
\equiv&\ G\left(s,x,-V(t,x),p_s,P_s,u\right)+\left\langle q_s-P_s\bar{\sigma}(s),\sigma(s,x,u)\right\rangle-\frac{1}{2}\left\langle P_s\bar{\sigma}(s),\bar{\sigma}(s)\right\rangle \\
&\ +\int_{\mathcal{E}}\left[V(s,x+f(s,x,u,e))-V(s,x)+\left\langle p_s,f(s,x,u,e)\right\rangle\right]\nu(de) \\
&\ +\int_{\mathcal{E}}\left\langle\tilde{q}_{(s,e)}-\frac{1}{2}(P_s+\tilde{Q}_{(s,e)})\bar{f}(s,e),\bar{f}(s,e)\right\rangle \nu(de).
\end{aligned}
 \label{definition of HH}
\end{equation}
\end{mythm}

\begin{proof}
For any $s\in(t,T)$, choose $\tau \in (s,T]$. Denote by $X_\cdot^{\tau,\bar{X}_s^{t,x;\bar{u}};\bar{u}}$ the solution to the following SDEP on $[\tau,T]$:
\begin{equation}
\begin{aligned}
X_r^{\tau,\bar{X}_s^{t,x;\bar{u}};\bar{u}}=&\ \bar{X}_s^{t,x;\bar{u}}+ \int_\tau^r b\left(\alpha,X_\alpha^{\tau,\bar{X}_s^{t,x;\bar{u}};\bar{u}},\bar{u}_\alpha\right)d \alpha+\int_\tau^r \sigma \left(\alpha,X_\alpha^{\tau,\bar{X}_s^{t,x;\bar{u}};\bar{u}},\bar{u}_\alpha\right)d W_\alpha \\
& +\int_\tau^r \int_{\mathcal{E}} f\left(\alpha,X_\alpha^{\tau,\bar{X}_s^{t,x;\bar{u}};\bar{u}},\bar{u}_\alpha,e\right)\tilde{N}(d e,d \alpha).
\end{aligned}
 \label{SDEP with tau}
\end{equation}
Set $\hat{X}_r^\tau:= X_r^{\tau,\bar{X}_s^{t,x;\bar{u}};\bar{u}}-\bar{X}_r^{t,x;\bar{u}}$, $\tau \leqslant r \leqslant T$. We have the following estimate for any $k=2^i,i=1,2\cdots$,
$$
\mathbf{E}\left[\sup_{\tau \leqslant r \leqslant T}|\hat{X}_r^\tau|^{k} \bigg| \mathcal{F}_\tau^t\right] \leqslant C|\bar{X}_\tau^{t,x;\bar{u}}-\bar{X}_s^{t,x;\bar{u}}|^{k},\quad  \mathbf{P} \text{-a.s.}.
$$
Taking $\mathbf{E}[\cdot | \mathcal{F}_s^t]$ on both sides, by a standard argument we get
\begin{equation}
\mathbf{E}\left[\sup_{\tau \leqslant r \leqslant T}|\hat{X}_r^\tau|^{k} \bigg| \mathcal{F}_s^t\right] \leqslant C|\tau-s|^{\frac{k}{2}}, \quad  \mathbf{P} \text{-a.s.}.
 \label{estimate of sup X with tau}
\end{equation}
The process $\hat{X}_\cdot^\tau$ satisfies the following variational equations:
\begin{equation}
\left\{\begin{aligned}
d \hat{X}_r^\tau& = \left[\bar{b}_x(r) \hat{X}_r^\tau+\varepsilon_{\tau 1}(r)\right] d r+\left[\bar{\sigma}_x(r) \hat{X}_r^\tau+\varepsilon_{\tau 2}(r)\right] d W_r \\
&\quad +\int_{\mathcal{E}} \left[\bar{f}_x(r, e) \hat{X}_{r-}^\tau+\varepsilon_{\tau 3}(r, e)\right] \tilde{N}(d e, d r) , \quad r \in[\tau, T], \\
\hat{X}_\tau^\tau& =-\int_s^\tau \bar{b}(r)d r -\int_s^\tau \bar{\sigma}(r)d W_r -\int_s^\tau\int_{\mathcal{E}}\bar{f}(r,e)\tilde{N}(d e,d r),
\end{aligned}\right.
 \label{first-order variational equation with tau}
\end{equation}
where
\begin{equation}
\left\{\begin{aligned}
\varepsilon_{\tau 1}(r):=& \int_0^1\left[b_x\left(r, \bar{X}^{t, x ; \bar{u}}_r+\theta \hat{X}_r^\tau, \bar{u}_r\right)-\bar{b}_x(r)\right] \hat{X}_r^\tau d \theta, \\
\varepsilon_{\tau 2}(r):=& \int_0^1\left[\sigma_x\left(r, \bar{X}^{t, x ; \bar{u}}_r+\theta \hat{X}_r^\tau, \bar{u}_r\right)-\bar{\sigma}_x(r)\right] \hat{X}_r^\tau d \theta, \\
\varepsilon_{\tau 3}(r, e):=& \int_0^1\left[f_x\left(r, \bar{X}^{t, x ; \bar{u}}_{r-}+\theta \hat{X}_{r-}^\tau, \bar{u}_r, e\right)-\bar{f}_x(r,e)\right] \hat{X}_{r-}^\tau d \theta,
\end{aligned}\right.
 \label{definition of tau 1-3}
\end{equation}
and
\begin{equation}
\left\{\begin{aligned}
d \hat{X}_r^\tau= & \left[\bar{b}_x(r) \hat{X}_r^\tau+\frac{1}{2} \hat{X}_r^{\tau\top} \bar{b}_{x x}(r) \hat{X}_r^\tau+\varepsilon_{\tau 4}(r)\right] d r \\
& +\left[\bar{\sigma}_x(r) \hat{X}_r^\tau+\frac{1}{2} \hat{X}_r^{\tau\top} \bar{\sigma}_{x x}(r) \hat{X}_r^\tau+\varepsilon_{\tau 5}(r)\right] d W_r \\
& +\int_{\mathcal{E}}\left[\bar{f}_x(r, e) \hat{X}_{r-}^\tau+\frac{1}{2} \hat{X}_{r-}^{\tau\top} \bar{f}_{x x}(r, e) \hat{X}_{r-}^\tau+\varepsilon_{\tau 6}(r, e)\right] \tilde{N}(d e,d r), \quad r \in[\tau, T], \\
\hat{X}_\tau^\tau= & -\int_s^\tau \bar{b}(r)d r -\int_s^\tau \bar{\sigma}(r)d W_r -\int_s^\tau\int_{\mathcal{E}}\bar{f}(r,e)\tilde{N}(d e,d r),
\end{aligned}\right.
 \label{second-order variational equation with tau}
\end{equation}
where
\begin{equation}
\left\{\begin{aligned}
\varepsilon_{\tau 4}(r):= & \int_0^1(1-\theta)\hat{X}_r^{\tau\top}\left[b _ { x x } \left(r, \bar{X}^{t, x ; \bar{u}}_r+\theta \hat{X}_r^\tau, \bar{u}_r\right)-\bar{b}_{x x}(r)\right] \hat{X}_r^\tau d \theta, \\
\varepsilon_{\tau 5}(r):= & \int_0^1(1-\theta) \hat{X}_r^{\tau\top}\left[\sigma _ { x x } \left(r, \bar{X}^{t, x ; \bar{u}}_r+\theta \hat{X}_r^\tau, \bar{u}_r\right)-\bar{\sigma}_{x x}(r)\right] \hat{X}_r^\tau d \theta, \\
\varepsilon_{\tau 6}(r, e):= & \int_0^1(1-\theta) \hat{X}_{r-}^{\tau\top}\left[f _ { x x } \left(r, \bar{X}^{t, x ; \bar{u}}_{r-}+\theta \hat{X}_{r-}^\tau, \bar{u}_r, e\right)-\bar{f}_{x x}(r, e)\right] \hat{X}_{r-}^\tau d \theta .
\end{aligned}\right.
 \label{definition of tau 4-6}
\end{equation}
Similar to the proof of (\ref{estimate of z 1-6}), using (\ref{estimate of sup X with tau}) we obtain, $\mathbf{P}$-a.s.,
\begin{equation}
\left\{\begin{aligned}
& \mathbf{E}\left[\int_s^T\left|\varepsilon_{\tau 1}(r)\right|^k d r \bigg| \mathcal{F}_s^t\right] \leqslant \delta\left(|\tau-s|^{\frac{k}{2}}\right), \quad
 \mathbf{E}\left[\int_s^T\left|\varepsilon_{\tau 2}(r)\right|^k d r \bigg| \mathcal{F}_s^t\right] \leqslant \delta\left(|\tau-s|^{\frac{k}{2}}\right),  \\
& \mathbf{E}\left[\int_s^T\left\|\varepsilon_{\tau 3}(r,e)\right\|_{\mathcal{L}^2}^k d r \bigg| \mathcal{F}_s^t\right] \leqslant \delta\left(|\tau-s|^{\frac{k}{2}}\right), \\
& \mathbf{E}\left[\int_s^T\left|\varepsilon_{\tau 4}(r)\right|^k d r \bigg| \mathcal{F}_s^t\right] \leqslant \delta\left(|\tau-s|^k\right), \quad
 \mathbf{E}\left[\int_s^T\left|\varepsilon_{\tau 5}(r)\right|^k d r \bigg| \mathcal{F}_s^t\right] \leqslant \delta\left(|\tau-s|^k\right),  \\
& \mathbf{E}\left[\int_s^T\left\|\varepsilon_{\tau 6}(r,e)\right\|_{\mathcal{L}^2}^k d r \bigg| \mathcal{F}_s^t\right] \leqslant \delta\left(|\tau-s|^k\right).
\end{aligned}\right.
 \label{estimate of tau 1-6}
\end{equation}
Denote by $(Y_\cdot^{\tau,\bar{X}_s^{t,x;\bar{u}};\bar{u}},Z_\cdot^{\tau,\bar{X}_s^{t,x;\bar{u}};\bar{u}},\tilde{Z}_{(\cdot,\cdot)}^{\tau,\bar{X}_s^{t,x;\bar{u}};\bar{u}})$ the solution to the following BSDEP on $\left(\Omega,\mathcal{F}\right.$, $\left.\{\mathcal{F}_\tau^t\}_{\tau\geq t},\mathbf{P}(\cdot | \mathcal{F}_\tau^t)\right)$ for $r \in [\tau,T]$:
\begin{equation}
\begin{aligned}
Y_r^{\tau,\bar{X}_s^{t,x;\bar{u}};\bar{u}}=&\ \phi\left(X_T^{\tau,\bar{X}_s^{t,x;\bar{u}};\bar{u}}\right)
 +\int_\tau^T g\left(\alpha,X_\alpha^{\tau,\bar{X}_s^{t,x;\bar{u}};\bar{u}},Y_\alpha^{\tau,\bar{X}_s^{t,x;\bar{u}};\bar{u}},
 Z_\alpha^{\tau,\bar{X}_s^{t,x;\bar{u}};\bar{u}},\tilde{Z}_{(\alpha,e)}^{\tau,\bar{X}_s^{t,x;\bar{u}};\bar{u}},\bar{u}\right)d \alpha \\
& -\int_\tau^T Z_\alpha^{\tau,\bar{X}_s^{t,x;\bar{u}};\bar{u}}d W_\alpha -\int_\tau^T \int_{\mathcal{E}}\tilde{Z}_{(\alpha,e)}^{\tau,\bar{X}_s^{t,x;\bar{u}};\bar{u}} \tilde{N}(d e,d \alpha).
\end{aligned}
 \label{BSDEP with tau}
\end{equation}
For any $\tau \leqslant r \leqslant T$, set $\hat{Y}_r^\tau:=\left\langle p_r, \hat{X}_r^\tau \right\rangle +\frac{1}{2} \left\langle P_r\hat{X}_r^\tau, \hat{X}_r^\tau \right\rangle$ and $\Phi_r^\tau:=\hat{X}_r^\tau\hat{X}_r^{\tau\top}$. Applying It\^{o}'s formula, we have
\begin{equation}
d\hat{Y}_r^\tau=C^\tau(r)d r+\hat{Z}_r^\tau d W_r + \int_{\mathcal{E}}\hat{\tilde{Z}}_{(r,e)}^\tau \tilde{N}(d e,d r),\quad r\in[\tau,T],
 \label{definition of hatY with tau}
\end{equation}
where
$$
\begin{aligned}
 C^\tau(r)&:= -\Big\langle \hat{X}_r^\tau,\bar{g}_x(r)+\bar{g}_y(r)p_r+\bar{g}_z(r)\left[\bar{\sigma}_x(r)^\top p_r+q_r\right]+\bar{g}_{\tilde{z}}(r)\int_{\mathcal{E}}\Big[\bar{f}_x(r,e)^\top p_r +\tilde{q}_{(r,e)}\\
& \quad +\bar{f}_x(r,e)^\top \tilde{q}_{(r,e)}\Big]\nu (d e)\Big\rangle+ \langle p_r,\varepsilon_{\tau 4}(r)\rangle +\langle q_r,\varepsilon_{\tau 5}(r)\rangle
 +\int_{\mathcal{E}}\left\langle \tilde{q}_{(r,e)},\varepsilon_{\tau 6}(r,e)\right\rangle \nu(d e)\\
& \quad  +\operatorname{tr}\bigg\{-\frac{1}{2}\Phi_r^\tau\bar{g}_y(r)P_r -\frac{1}{2}\Phi_r^\tau\bar{g}_z(r)\bar{\sigma}_x(r)^\top P_r -\frac{1}{2}\Phi_r^\tau P_r\bar{g}_z(r)\bar{\sigma}_x(r)
 -\frac{1}{2}\Phi_r^\tau\bar{g}_z(r)Q_r\\
& \quad -\frac{1}{2}\Phi_r^\tau\bar{\sigma}_{x x}(r)^\top\bar{g}_z(r)p_r+\frac{1}{2}P_r\varepsilon_{\tau 7}(r) +\frac{1}{2}Q_r\varepsilon_{\tau 8}(r) -\frac{1}{2}\Phi_r^\tau \Psi(r)D^2\bar{g}(r) \Psi(r)^\top\\
& \quad -\int_{\mathcal{E}}\left[\frac{1}{2}\Phi_r^\tau\bar{f}_{x x}(r,e)^\top \bar{g}_{\tilde{z}}(r)p_r +\frac{1}{2}\Phi_r^\tau\bar{f}_{x x}(r,e)^\top \bar{g}_{\tilde{z}}(r)\tilde{q}_{(r,e)}
 +\frac{1}{2}\Phi_r^\tau\bar{g}_{\tilde{z}}(r)\bar{f}_x(r,e)^\top P_r\right.\\
& \quad +\frac{1}{2}\Phi_r^\tau P_r\bar{g}_{\tilde{z}}(r)\bar{f}_x(r,e)+\frac{1}{2}\Phi_r^\tau\bar{g}_{\tilde{z}}(r)\bar{f}_x(r,e)^\top \tilde{Q}_{(r,e)}
 +\frac{1}{2}\Phi_r^\tau\bar{g}_{\tilde{z}}(r)\tilde{Q}_{(r,e)}\bar{f}_x(r,e)\\
& \quad +\frac{1}{2}\tilde{Q}_{(r,e)}\varepsilon_{\tau 9}(r,e)+\frac{1}{2}\Phi_r^\tau\bar{g}_{\tilde{z}}(r)\bar{f}_x(r,e)^\top \tilde{Q}_{(r,e)}\bar{f}_x(r,e) \\
& \quad \left.+\frac{1}{2}\Phi_r^\tau\bar{g}_{\tilde{z}}(r)\bar{f}_x(r,e)^\top P_r\bar{f}_x(r,e)+\frac{1}{2}\Phi_r^\tau\bar{g}_{\tilde{z}}(r)\tilde{Q}_{(r,e)}\right]\nu(de)\bigg\}, \\
 \hat{Z}_r^\tau&:= \left\langle q_r,\hat{X}_r^\tau\right\rangle + \Big\langle p_r,\bar{\sigma}_x(r)\hat{X}_r^\tau+\frac{1}{2}\hat{X}_r^{\tau\top} \bar{\sigma}_{x x}(r)\hat{X}_r^\tau+\varepsilon_{\tau 5}(r)\Big\rangle \\
& \quad +\operatorname{tr}\bigg\{\frac{1}{2}\Phi_r^\tau Q_r+\frac{1}{2}P_r\Phi_r^\tau \bar{\sigma}_x(r)^\top+\frac{1}{2}P_r\bar{\sigma}_x(r)\Phi_r^\tau+\frac{1}{2}P_r\varepsilon_{\tau 8}(r)\bigg\}, \\
 \hat{\tilde{Z}}_{(r,e)}^\tau&:= \left\langle \tilde{q}_{(r,e)},\hat{X}_{r-}^\tau\right\rangle +\left\langle p_r,\bar{f}_x(r,e)\hat{X}_{r-}^\tau+\frac{1}{2}\hat{X}_{r-}^{\tau\top} \bar{f}_{x x}(r,e)\hat{X}_{r-}^\tau+\varepsilon_{\tau 6}(r,e)\right\rangle\\
& \quad +\left\langle \tilde{q}_{(r,e)},\bar{f}_x(r,e)\hat{X}_{r-}^\tau\right\rangle +\left\langle \tilde{q}_{(r,e)},\frac{1}{2}\hat{X}_{r-}^{\tau\top} \bar{f}_{x x}(r,e)\hat{X}_{r-}^\tau+\varepsilon_{\tau 6}(r,e)\right\rangle \\
& \quad +\operatorname{tr}\bigg\{\int_{\mathcal{E}}\left[\frac{1}{2}\Phi_{r-}^\tau\tilde{Q}_{(r,e)}+\frac{1}{2}P_r\bar{f}_x(r,e)\Phi_{r-}^\tau +\frac{1}{2}P_r\Phi_{r-}^\tau\bar{f}_x(r,e)^\top \right.\\
& \quad +\frac{1}{2}P_r\bar{f}_x(r,e)\Phi_{r-}^\tau\bar{f}_x(r,e)^\top+\frac{1}{2}\tilde{Q}_{(r,e)}\bar{f}_x(r,e)\Phi_{r-}^\tau +\frac{1}{2}\tilde{Q}_{(r,e)}\Phi_{r-}^\tau\bar{f}_x(r,e)^\top \\
& \quad \left.+\frac{1}{2}\tilde{Q}_{(r,e)}\bar{f}_x(r,e)\Phi_{r-}^\tau\bar{f}_x(r,e)^\top+\frac{1}{2}P_r\varepsilon_{\tau 9}(r, e)+\frac{1}{2}\tilde{Q}_{(r,e)}\varepsilon_{\tau 9}(r, e)\right]\nu(de) \bigg\},
\end{aligned}
$$
and
$$
\left\{\begin{aligned}
\varepsilon_{\tau 7}(r):= &\ \varepsilon_{\tau 1}(r) \hat{X}_r^{\tau\top}+\hat{X}_r^\tau \varepsilon_{\tau 1}(r)^\top +\bar{\sigma}_x(r) \hat{X}_r^\tau \varepsilon_{\tau 2}(r)^\top
 +\varepsilon_{z 2}(r) \hat{X}_r^{\tau\top} \bar{\sigma}_x(r)^\top \\
& +\varepsilon_{\tau 2}(r) \varepsilon_{\tau 2}(r)^\top +\int_{\mathcal{E}}\left\{\bar{f}_x(r, e) \hat{X}_r^\tau \varepsilon_{\tau 3}(r, e)^\top \right.\\
& \left.+\varepsilon_{\tau 3}(r, e) \hat{X}_r^{\tau\top} \bar{f}_x(r, e)^\top +\varepsilon_{\tau 3}(r, e) \varepsilon_{\tau 3}(r, e)^\top \right\} \nu(d e), \\
\varepsilon_{\tau 8}(r):= &\ \varepsilon_{\tau 2}(r) \hat{X}_r^{\tau\top} +\hat{X}_r^\tau \varepsilon_{\tau 2}(r)^\top, \\
\varepsilon_{\tau 9}(r, e):= &\ \bar{f}_x(r, e) \hat{X}_{r-}^\tau \varepsilon_{\tau 3}(r, e)^\top +\varepsilon_{\tau 3}(r, e) \hat{X}_{r-}^{\tau\top} \bar{f}_x(r, e)^\top +\hat{X}_{r-}^\tau \varepsilon_{\tau 3}(r, e)^\top \\
& +\varepsilon_{\tau 3}(r, e) \hat{X}_{r-}^{\tau\top} +\varepsilon_{\tau 3}(r, e) \varepsilon_{\tau 3}(r, e)^\top .
\end{aligned}\right.
$$
Define
\begin{equation}
\Delta{Y}_r^\tau:= Y_r^{\tau,\bar{X}_s^{t,x;\bar{u}};\bar{u}}-\hat{Y}_r^\tau, \
\Delta{Z}_r^\tau:= Z_r^{\tau,\bar{X}_s^{t,x;\bar{u}};\bar{u}}-\hat{Z}_r^\tau, \
\Delta{\tilde{Z}}_{(r,e)}^\tau:= \tilde{Z}_{(r,e)}^{\tau,\bar{X}_s^{t,x;\bar{u}};\bar{u}}-\hat{\tilde{Z}}_{(r,e)}^\tau.
 \label{definition of DeltaYZZ with tau}
\end{equation}
Similar to the Step 4 and Step 5 in the proof of Theorem \ref{nonsmooth state theorem}, we obtain
\begin{equation}
\left|\Delta{Y}_\tau^\tau-\bar{Y}_\tau^{t,x;\bar{u}}\right|^2 \leqslant \delta(|\tau-s|^2).
 \label{estimate of DeltaY-barY with tau}
\end{equation}

Note that $\left(\Omega, \mathcal{F}, \mathbf{P}\left(\cdot | \mathcal{F}_\tau^t\right), W_\cdot-W_\tau ; \left.u(\cdot)\right|_{[\tau, T]}\right) \in \mathcal{U}^w[\tau, T]$, $\mathbf{P}$-a.s.. Thus by the definition of the value function $V$, we have
$$
V\left(\tau, \bar{X}_s^{t, x ; \bar{u}}\right) \leqslant-Y^{\tau, \bar{X}^{t, x ; \bar{u}}_s ; \bar{u}}_\tau, \quad \mathbf{P} \text {-a.s.}.
$$
Taking $\mathbf{E}\left(\cdot \mid \mathcal{F}_s^t\right)$ on both sides and noting that $\bar{X}^{t, x ; \bar{u}}_s$ is $\mathcal{F}_s^t$-measurable, we have
\begin{equation}
V\left(\tau, \bar{X}^{t, x ; \bar{u}}_s\right) \leqslant \mathbf{E}\left[-Y^{\tau, \bar{X}_s^{t, x ; \bar{u}} ; \bar{u}}_\tau \Big| \mathcal{F}_s^t\right], \quad \mathbf{P}\text {-a.s.}.
 \label{estimate of V with tau}
\end{equation}

By the DPP of \cite{LiPeng2009}, choose a common subset $\Omega_0 \subseteq \Omega$ with $\mathbf{P}\left(\Omega_0\right)=1$ such that for any $\omega_0 \in \Omega_0$, the following holds:
$$
\left\{\begin{array}{l}
V\left(s, \bar{X}^{t, x ; \bar{u}}_{\left(s, \omega_0\right)}\right)=-\bar{Y}^{t, x ; \bar{u}}_{(s, \omega_0)}, \text{ and }(\ref{estimate of sup X with tau}), (\ref{estimate of tau 1-6}), (\ref{estimate of V with tau})
 \text { are satisfied for any } \\
\text{ rational}\ \tau>s,\ \left(\Omega, \mathcal{F}, \mathbf{P}(\cdot | \mathcal{F}_s^t)(\omega_0), W_\cdot-W_s, \tilde{N}(\cdot,\cdot) ; \left.u(\cdot)\right|_{[s, T]}\right) \in \mathcal{U}^w[s, T], \\
\text { and } \sup\limits_{s \leqslant r \leqslant T}\left[|p_{\left(r, \omega_0\right)}|+|P_{\left(r, \omega_0\right)}|\right]<\infty.
\end{array}\right.
$$
Let $\omega_0 \in \Omega_0$ be fixed, and then for any rational number $\tau>s$, by (\ref{estimate of DeltaY-barY with tau}) we have
\begin{equation}
\begin{aligned}
& V\left(\tau, \bar{X}^{t, x ; \bar{u}}_{\left(s, \omega_0\right)}\right)-V\left(s, \bar{X}^{t, x ; \bar{u}}_{\left(s, \omega_0\right)}\right)
 \leqslant \mathbf{E}\left[-Y^{\tau, \bar{X}^{t, x ; \bar{u}}_s ; \bar{u}}_\tau+\bar{Y}^{t, x ; \bar{u}}_s \Big| \mathcal{F}_s^t\right]\left(\omega_0\right) \\
& \leqslant \mathbf{E}\left[-Y^{\tau, \bar{X}^{t, x ; \bar{u}}_s ; \bar{u}}_\tau+\bar{Y}^{t, x ; \bar{u}}_\tau-\bar{Y}^{t, x ; \bar{u}}_\tau+\bar{Y}^{t, x ; \bar{u}}_s \Big| \mathcal{F}_s^t\right]\left(\omega_0\right) \\
&= \mathbf{E}\left[-(\Delta{Y}_\tau^\tau-\bar{Y}_\tau^{t,x;\bar{u}})-\hat{Y}_\tau^\tau \Big| \mathcal{F}_s^t \right]\left(\omega_0\right)
 +\mathbf{E}\left[-\bar{Y}^{t, x ; \bar{u}}_\tau+\bar{Y}^{t, x ; \bar{u}}_s \Big| \mathcal{F}_s^t\right]\left(\omega_0\right) \\
&= \mathbf{E}\left[-\langle p_\tau, \hat{X}_\tau^\tau \rangle -\frac{1}{2} \langle P_\tau\hat{X}_\tau^\tau, \hat{X}_\tau^\tau \rangle \Big| \mathcal{F}_s^t\right]\left(\omega_0\right)
 +\delta(|\tau-s|)+\mathbf{E}\left[-\bar{Y}^{t, x ; \bar{u}}_\tau+\bar{Y}^{t, x ; \bar{u}}_s \Big| \mathcal{F}_s^t\right]\left(\omega_0\right) \\
&= \mathbf{E}\left[\int_s^\tau g\bigg(r,\bar{X}_r^{t,x;\bar{u}},-V(r,\bar{X}_r^{t,x;\bar{u}}),\sigma(r,\bar{X}_r^{t,x;\bar{u}},\bar{u})^\top
 p_r,-\int_{\mathcal{E}}\Big[V\left(r,\bar{X}_r^{t,x;\bar{u}}+\bar{f}(r,e)\right)\right. \\
& \qquad \qquad \left.-V(r,\bar{X}_r^{t,x;\bar{u}})\Big]\nu(de),\bar{u}\bigg) d r-\langle p_\tau, \hat{X}_\tau^\tau \rangle -\frac{1}{2} \langle P_\tau\hat{X}_\tau^\tau, \hat{X}_\tau^\tau \rangle \Big| \mathcal{F}_s^t\right]\left(\omega_0\right)+\delta(|\tau-s|).
\end{aligned}
 \label{V-V with tau}
\end{equation}

Now let us estimate the terms on the right-hand side of (\ref{V-V with tau}). To this end, noting that $\omega_0 \in \Omega_0$ is fixed, thus for any square-integrable functions $\phi(\cdot), \hat{\phi}(\cdot), \varphi(\cdot), \tilde{\varphi}(\cdot,e)\in L^2([0,T];\mathbf{R}^n), \forall e\in\mathcal{E}$, we have
$$
\begin{aligned}
& \mathbf{E}\left[\left\langle \int_s^\tau \phi(r)d r, \int_s^\tau \hat{\phi}(r)d r \right\rangle \bigg| \mathcal{F}_s^t\right](\omega_0) \\
\leqslant &\ \left\{\mathbf{E}\left[\left|\int_s^\tau \phi(r)d r\right|^2 \bigg| \mathcal{F}_s^t\right](\omega_0)\right\}^{\frac{1}{2}}
 \cdot \left\{\mathbf{E}\left[\left|\int_s^\tau \hat{\phi}(r)d r\right|^2 \bigg| \mathcal{F}_s^t\right](\omega_0)\right\}^{\frac{1}{2}} \\
\leqslant &\ (\tau-s)\left\{\int_s^\tau \mathbf{E}\left[|\phi(r)|^2 \big| \mathcal{F}_s^t\right](\omega_0)d r
 \cdot \int_s^\tau \mathbf{E}\left[|\hat{\phi}(r)|^2 \big| \mathcal{F}_s^t\right](\omega_0)d r\right\}^{\frac{1}{2}} \\
\leqslant &\ \delta(|\tau-s|),\quad \text{as } \tau \downarrow s, \ \forall s \in [t,T),
\end{aligned}
$$
$$
\begin{aligned}
& \mathbf{E}\left[\left\langle \int_s^\tau \phi(r)d r, \int_s^\tau \varphi(r)d W_r \right\rangle \bigg| \mathcal{F}_s^t\right](\omega_0) \\
\leqslant &\ \left\{\mathbf{E}\left[\left|\int_s^\tau \varphi(r)d r\right|^2 \bigg| \mathcal{F}_s^t\right](\omega_0)\right\}^{\frac{1}{2}}
 \cdot \left\{\mathbf{E}\left[\left|\int_s^\tau \psi(r)d W_r\right|^2 \bigg| \mathcal{F}_s^t\right](\omega_0)\right\}^{\frac{1}{2}} \\
\leqslant &\ (\tau-s)^{\frac{1}{2}}\left\{\int_s^\tau \mathbf{E}\left[|\varphi(r)|^2 \big| \mathcal{F}_s^t\right](\omega_0)d r
 \cdot \int_s^\tau \mathbf{E}\left[|\psi(r)|^2 \big| \mathcal{F}_s^t\right](\omega_0)d r\right\}^{\frac{1}{2}} \\
\leqslant &\ \delta(|\tau-s|),\quad \text{as } \tau \downarrow s, \ \text{a.e.} \ s \in [t,T),
\end{aligned}
$$
and
$$
\begin{aligned}
& \mathbf{E}\left[\left\langle \int_s^\tau \phi(r)d r, \int_s^\tau\int_{\mathcal{E}} \tilde{\varphi}(r,e)\tilde{N}(d e,d r) \right\rangle \bigg| \mathcal{F}_s^t\right](\omega_0) \\
\leqslant &\ \left\{\mathbf{E}\left[\left|\int_s^\tau \phi(r)d r\right|^2 \bigg| \mathcal{F}_s^t\right](\omega_0)\right\}^{\frac{1}{2}}
 \cdot\left\{\mathbf{E}\left[\left|\int_{\mathcal{E}}\int_s^\tau \tilde{\varphi}(r,e)\tilde{N}(d e,d r)\right|^2 \bigg| \mathcal{F}_s^t\right](\omega_0)\right\}^{\frac{1}{2}} \\
\leqslant &\ (\tau-s)^{\frac{1}{2}}\left\{\int_s^\tau \mathbf{E}\left[|\phi(r)|^2 \big| \mathcal{F}_s^t\right](\omega_0)d r
 \cdot \int_{\mathcal{E}}\int_s^\tau \mathbf{E}\left[|\xi(r,e)|^2 \big| \mathcal{F}_s^t\right](\omega_0)\nu (d e)d r\right\}^{\frac{1}{2}} \\
\leqslant &\ \delta(|\tau-s|),\quad \text{as } \tau \downarrow s, \ \text{a.e.} \ s \in [t,T).
\end{aligned}
$$
All the last inequalities in the above three inequalities is due to the fact that the sets of right Lebesgue points have full Lebesgue measures for integrable functions by Lemma \ref{Lebesgue point}, and $s \mapsto \mathcal{F}_s^t$ is right continuous in $s$. Thus by (\ref{first-order variational equation with tau}) and (\ref{first-order adjoint equation}), we obtain that
\begin{equation*}
\begin{aligned}
&\ \mathbf{E}\left[\langle p_\tau,\hat{X}_\tau^\tau \rangle \big| \mathcal{F}_s^t\right](\omega_0) \\
=&\ \mathbf{E}\left[\left\langle p_s,\hat{X}_\tau^\tau \right\rangle+ \left\langle p_\tau-p_s,\hat{X}_\tau^\tau\right\rangle \Big| \mathcal{F}_s^t\right](\omega_0) \\
=&\ \mathbf{E}\left[\left\langle p_s,-\int_s^\tau \bar{b}(r)d r -\int_s^\tau \bar{\sigma}(r)d W_r -\int_{\mathcal{E}}\int_s^\tau\bar{f}(r,e)\tilde{N}(d e,d r)\right\rangle \right.\\
& \quad +\left\langle -\int_s^\tau \bigg\{\bar{b}_x(r)^\top p_r+\bar{\sigma}_x(r)^\top q_r+\bar{g}_x(r)+\bar{g}_y(r)p_r+\bar{g}_z(r)\left[\bar{\sigma}_x(r)^\top p_r+q_r\right] \right.\\
\end{aligned}
\end{equation*}
\begin{equation}
\begin{aligned}
& \quad +\bar{g}_{\tilde{z}}(r)\int_{\mathcal{E}}\left[\bar{f}_x(r,e)^\top p_r+\tilde{q}_{(r,e)}+\bar{f}_x(r,e)^\top \tilde{q}_{(r,e)}\right]\nu(de)+\int_{\mathcal{E}}\bar{f}_x(r,e)^\top \tilde{q}_{(r,e)}\nu(de)\bigg\}dr \\
& \quad +\int_s^\tau q_r d W_r+\int_{\mathcal{E}}\int_s^\tau \tilde{q}_{(r,e)}\tilde{N}(d e,d r), \\
& \qquad \left.\left.-\int_s^\tau \bar{b}(r)d r -\int_s^\tau \bar{\sigma}(r)d W_r -\int_{\mathcal{E}}\int_s^\tau\bar{f}(r,e)\tilde{N}(d e,d r)\right\rangle \bigg| \mathcal{F}_s^t\right](\omega_0) \\
\leqslant&\ \mathbf{E}\left[\left\langle p_s,-\int_s^\tau \bar{b}(r)d r \right\rangle-\int_s^\tau \left\langle q_r,\bar{\sigma}(r)\right\rangle d r
 -\int_{\mathcal{E}}\int_s^\tau \left\langle \tilde{q}_{(r,e)},\bar{f}(r,e)\right\rangle \nu(d e)dr \bigg| \mathcal{F}_s^t \right](\omega_0) \\
& +\delta(|\tau-s|),\quad \text{as } \tau \downarrow s, \ \text{a.e.} \ s \in [t,T).
\end{aligned}
 \label{estimate of pX with tau}
\end{equation}
Similarly, by (\ref{second-order variational equation with tau}) and (\ref{second-order adjoint equation}), we obtain
\begin{equation}
\begin{aligned}
&\ \mathbf{E}\left[\left\langle P_\tau \hat{X}_\tau^\tau,\hat{X}_\tau^\tau \right\rangle \Big| \mathcal{F}_s^t\right](\omega_0)
 =\mathbf{E}\left[\left\langle P_s \hat{X}_\tau^\tau,\hat{X}_\tau^\tau \right\rangle +\left\langle (P_\tau-P_s) \hat{X}_\tau^\tau,\hat{X}_\tau^\tau \right\rangle \Big| \mathcal{F}_s^t\right](\omega_0) \\
\leqslant&\ \mathbf{E}\bigg[\int_s^\tau \left\langle P_s\bar{\sigma}(r),\bar{\sigma}(r) \right\rangle dr +\int_{\mathcal{E}}\int_s^\tau \left\langle P_s\bar{f}(r,e),\bar{f}(r,e) \right\rangle \nu(de)dr \\
&\ +\int_{\mathcal{E}}\int_s^\tau \left\langle \tilde{Q}_{(r,e)}\bar{f}(r,e),\bar{f}(r,e) \right\rangle \nu(de)dr \bigg| \mathcal{F}_s^t\bigg](\omega_0)\\
&\ +\delta(|\tau-s|),\quad \text{as } \tau \downarrow s, \ \text{a.e.} \ s \in [t,T).
\end{aligned}
 \label{estimate of PXX with tau}
\end{equation}
It follows from (\ref{V-V with tau}), (\ref{estimate of pX with tau}) and (\ref{estimate of PXX with tau}) that for any rational $\tau \in (s,T]$,
\begin{equation}
\begin{aligned}
&\ V\left(\tau, \bar{X}^{t, x ; \bar{u}}_{\left(s, \omega_0\right)}\right)-V\left(s, \bar{X}^{t, x ; \bar{u}}_{\left(s, \omega_0\right)}\right) \\
\leqslant&\ \mathbf{E}\left[\int_s^\tau g\bigg(r,\bar{X}_r^{t,x;\bar{u}},-V(r,\bar{X}_r^{t,x;\bar{u}}),\sigma(r,\bar{X}_r^{t,x;\bar{u}},\bar{u})^\top p_r,
 -\int_{\mathcal{E}}\Big[V\left(r,\bar{X}_r^{t,x;\bar{u}}+\bar{f}(r,e)\right)\right. \\
& \qquad \quad \left.-V(r,\bar{X}_r^{t,x;\bar{u}})\Big]\nu(de),\bar{u}\bigg) d r-\langle p_\tau, \hat{X}_\tau^\tau \rangle
 -\frac{1}{2} \langle P_\tau\hat{X}_\tau^\tau, \hat{X}_\tau^\tau \rangle \bigg| \mathcal{F}_s^t\right]\left(\omega_0\right) +\delta(|\tau-s|) \\
\leqslant&\ \mathbf{E}\left[\int_s^\tau g\bigg(r,\bar{X}_r^{t,x;\bar{u}},-V(r,\bar{X}_r^{t,x;\bar{u}}),
 \sigma(r,\bar{X}_r^{t,x;\bar{u}},\bar{u})^\top p_r,-\int_{\mathcal{E}}\Big[V\left(r,\bar{X}_r^{t,x;\bar{u}}+\bar{f}(r,e)\right)\right. \\
& \left.-V(r,\bar{X}_r^{t,x;\bar{u}})\Big]\nu(de),\bar{u}\bigg) d r+\left\langle p_s,\int_s^\tau \bar{b}(r)d r \right\rangle +\int_s^\tau \left\langle q_r,\bar{\sigma}(r)\right\rangle d r\right.\\
& \quad +\int_{\mathcal{E}}\int_s^\tau \left\langle \tilde{q}_{(r,e)},\bar{f}(r,e)\right\rangle \nu(d e)dr-\frac{1}{2}\int_s^\tau \left\langle P_s\bar{\sigma}(r),\bar{\sigma}(r) \right\rangle dr  \\
& \quad -\frac{1}{2}\int_{\mathcal{E}}\int_s^\tau \left\langle (P_s+\tilde{Q}_{(r,e)})\bar{f}(r,e),\bar{f}(r,e) \right\rangle \nu(de)dr \bigg| \mathcal{F}_s^t\bigg](\omega_0)+\delta(|\tau-s|) \\
=&\ (\tau-s)\mathcal{H}\left(s,\bar{X}_s^{t,x;\bar{u}},\bar{u}_s\right)+\delta(|\tau-s|).
\end{aligned}
 \label{estimate of V-V with tau}
\end{equation}
By definition (\ref{definition of partial super-subjets to t}), we obtain that the first relation of (\ref{relation of H and DV}) holds for any (not only rational numbers) $\tau \in (s,T]$. Finally, fix an $\omega \in \Omega$ such that (\ref{estimate of V-V with tau}) holds for any $\tau \in (s,T]$. Then for any $\hat{q} \in \mathcal{D}_{t+}^{1,-}V(s,\bar{X}_s^{t,x;\bar{u}})$, by definition (\ref{definition of partial super-subjets to t}) and (\ref{estimate of V-V with tau}) we have
$$
\begin{aligned}
0 & \leqslant \liminf _{\tau \downarrow s}\left\{\frac{V\left(\tau, \bar{X}_s^{t, x ; \bar{u}}\right)-V\left(s, \bar{X}_s^{t, x ; \bar{u}}\right)-\hat{q}(\tau-s)}{|\tau-s|}\right\} \\
& \leqslant \liminf _{\tau \downarrow s}\left\{\frac{\left(\mathcal{H}\left(s, \bar{X}_s^{t, x ; \bar{u}}, \bar{u}_s\right)-\hat{q}\right)(\tau-s)}{|\tau-s|}\right\} .
\end{aligned}
$$
Then, it is necessary that $\hat{q}\leqslant \mathcal{H}\left(s, \bar{X}_s^{t, x ; \bar{u}}, \bar{u}_s\right)$. Thus, the second relation of (\ref{relation of H and DV}) holds. The proof is complete.
\end{proof}

\begin{Remark}
\cite{ShiWu2011} studied the relationship between the general MP and DPP in the framework of viscosity solutions for systems which satisfies SDEP, and obtained Theorem 3.3, which is similar to our Theorem 4.2. However, (63) in \cite{ShiWu2011} lacks a term $\int_{\mathcal{E}}\int_s^\tau \left\langle \tilde{Q}_{(r,e)}\bar{f}(r,e),\bar{f}(r,e) \right\rangle \nu(de)dr$ compared to (\ref{estimate of PXX with tau}).

After we complete this term, (64) in \cite{ShiWu2011} will turn into the following inequality:
\begin{equation*}
\begin{aligned}
&V(\tau,\left.\bar{x}^{s, y ; \bar{u}}\left(t, \omega_0\right)\right)-V\left(t, \bar{x}^{s, y ; \bar{u}}\left(t, \omega_0\right)\right)\\
 \leq  &(\tau-t)\mathcal{H}^{\prime}(t,\bar{x}^{s,y;\bar{u}}(t,\omega_0),\bar{u})+o(|\tau-t|), \ \text{as}\ \tau \downarrow t,\ \text{a.e.}\ t\in[s,T),
\end{aligned}
 \label{new inequality in ShiWu2011}
\end{equation*}
which $\mathcal{H}^{\prime}$ is defined by (21) in \cite{ShiWu2011}. From definition (31) in \cite{ShiWu2011}, we can get the revised result of Theorem 3.3, which is
\begin{equation}
\begin{aligned}
\mathcal{H}^{\prime}(t,\bar{x}^{s,y;\bar{u}}(t),\bar{u})\in \mathcal{P}_{t +}^{1,+}V(t,\bar{x}^{s,y;\bar{u}}(t)),\ \text{a.e.}\ t\in [s,T], \mathbf{P}\text{-a.s.},
\end{aligned}
 \label{revised result in ShiWu2011}
\end{equation}
here $\mathcal{P}_{t +}^{1,+}$ is defined by (32) in \cite{ShiWu2011}.
\end{Remark}

\begin{mythm}
Under the condition of Theorem \ref{nonsmooth state theorem} or Theorem \ref{nonsmooth time theorem}, for a.e. $s \in [t,T],\mathbf{P}$-a.s.,
\begin{equation}
\begin{aligned}
& [\mathcal{H}(s,\bar{X}_s^{t,x;\bar{u}},\bar{u}_s),\infty)\times \{-p_s\}\times [-P_s,\infty)\subseteq \mathcal{D}_{t+, x}^{1,2,+} V(s, \bar{X}_s^{t,x;\bar{u}}), \\
& \mathcal{D}_{t+, x}^{1,2,-} V(s, \bar{X}_s^{t,x;\bar{u}}) \subseteq (-\infty,\mathcal{H}(s,\bar{X}_s^{t,x;\bar{u}},\bar{u}_s)]\times \{-p_s\}\times (-\infty,-P_s].
\end{aligned}
 \label{relation of p,P,H and V}
\end{equation}
\end{mythm}

\begin{proof}
It clear that (\ref{relation of p,P,H and V}) can be proved by combining the proofs of (\ref{relation of p,P and V}) and (\ref{relation of H and DV}).
\end{proof}

\section{Some examples }

\qquad In this section, we give some examples to illustrate the main results (Theorem \ref{smooth theorem}, Theorem \ref{nonsmooth state theorem} and Theorem \ref{nonsmooth time theorem}) of this paper.

$\mathbf{Example\ 5.1.}$ \ Given $t\in[0,T)$, consider the following controlled FBSDEP for $s\in[t,T]$:
\begin{equation}
\left\{\begin{aligned}
dX_s^{t,x;u}&=u_sd s + u_sd W_s + u_s \tilde{N}(d s), \\
-dY_s^{t,x;u}&=\left(u_sX_s^{t,x;u}-u_s-Z_s^{t,x;u}-\tilde{Z}_s^{t,x;u}\right)d s-Z_s^{t,x;u}d W_s-\tilde{Z}_s^{t,x;u} \tilde{N}(d s), \\
X_t^{t,x;u}&=x, \quad Y_T^{t,x;u}=\frac{1}{2}\left(X_T^{t,x;u}\right)^2,
\end{aligned}\right.
 \label{state equation of example 1}
\end{equation}
with $\mathbf{U}=[-3,-2]\cup [1,2]$. Here $n=1$ and $\tilde{N}(\cdot)$ is a standard Poisson process. The cost functional is defined as (\ref{cost functional}).

Then the HJB equation (\ref{HJB equation}) writes as
\begin{equation}
\left\{\begin{aligned}
&-v_t(t,x)+\sup_{u\in\mathbf{U}} \left\{-\frac{1}{2}v_{x x}(t,x)u^2+uv_x(t,x)+ux-u\right\}=0, \\
&\quad v(T,x)=-\frac{1}{2}x^2.
\end{aligned}\right.
 \label{HJB equation of example 1}
\end{equation}
We conjecture the unique solution is $V(t,x)=-\frac{1}{2}x^2$. We can verify that (\ref{HJB equation of example 1}) holds by $V_t(t,x)\equiv 0, V_x(t,x)=-x$, and $V_{x x}(t,x)=-1$. Moreover, the first-order and second-order adjoint equations (\ref{first-order adjoint equation}) and (\ref{second-order adjoint equation}) can be written as
$$
\left\{\begin{aligned}
-dp_s&=\left(-\bar{u}_s-q_s-\tilde{q}_s\right)ds-q_sd W_s-\tilde{q}_s\tilde{N}(ds), \\
p_T&=x,
\end{aligned}\right.
$$
$$
\left\{\begin{aligned}
-dP_s&=\left(-Q_s-\tilde{Q}_s\right)d s-Q_sd W_s-\tilde{Q}_s\tilde{N}(ds), \\
P_T&=1,
\end{aligned}\right.
$$
which, by It\^{o}'s formula, admits a unique solution $(p_s,q_s,\tilde{q}_s)=\left(\bar{X}_s^{t,x;\bar{u}},1,1\right)$ and $(P_s,Q_s,\tilde{Q}_s)=(1,0,0)$, where the optimal control $\bar{u}_s\equiv1$, for $s\in[t,T]$. Thus, the (\ref{relation of p and V}) and (\ref{relation of P and V}) holds. In fact, we can find another example that the strict inequality (\ref{relation of P and V}) happens.

$\mathbf{Example\ 5.2.}$ \ Given $t\in[0,T)$, consider the following controlled FBSDEP for $s\in[t,T]$ ($n=1$):
\begin{equation}
\left\{\begin{aligned}
dX_s^{t,x;u}&=u_sd s + u_sd W_s + u_s \tilde{N}(d s), \\
-dY_s^{t,x;u}&=\Big\{\operatorname{ln}X_s^{t,x;u}+\left({X_s^{t,x;u}}\right)^{-1}u_s+\left({X_s^{t,x;u}}\right)^{-3}u_s-\frac{1}{2}\left({X_s^{t,x;u}}\right)^{-4}\\
              &\quad -Y_s^{t,x;u}-Z_s^{t,x;u}-\tilde{Z}_s^{t,x;u}\Big\}d s -Z_s^{t,x;u}d W_s-\tilde{Z}_s^{t,x;u} \tilde{N}(d s), \\
X_t^{t,x;u}&=x>0, \quad Y_T^{t,x;u}=\operatorname{ln}X_T^{t,x;u},
\end{aligned}\right.
 \label{state equation of example 2}
\end{equation}
with $\mathbf{U}=\mathbf{R}^+$. The cost functional is defined as (\ref{cost functional}) and the corresponding generalized HJB equation (\ref{HJB equation}) reads as
\begin{equation}
\left\{\begin{aligned}
&-v_t(t,x)+\sup_{u\in\mathbf{U}} \left\{-\frac{1}{2}v_{x x}(t,x)u^2+v(t,x)+v_x(t,x)u+\operatorname{ln}x+x^{-1}u+x^{-3}u-\frac{1}{2}x^{-4}\right\}=0, \\
&\quad v(T,x)=-\operatorname{ln}x.
\end{aligned}\right.
 \label{HJB equation of example 2}
\end{equation}

Let us verify the solution to (\ref{HJB equation of example 2}) with $V(t,x)=-\operatorname{ln}x, V_x(t,x)=-x^{-1}, V_{x x}(t,x)=x^{-2}, V_t(t,x)\equiv0$. We have
\begin{equation*}
\sup_{u\in\mathbf{U}} \left\{-\frac{1}{2}x^{-2}u^2+x^{-3}u-\frac{1}{2}x^{-4}\right\}=\ \frac{1}{2}x^{-4}\sup_{u\in\mathbf{U}} \left\{-(xu-1)^2\right\}=0,
\end{equation*}
thus $\bar{u}_s=(\bar{X}_s^{t,x;\bar{u}})^{-1}$.

Next, we only consider the second-order adjoint equation (\ref{second-order adjoint equation}) which write as
\begin{equation}
\left\{\begin{aligned}
-dP_s=&\Big\{-P_s-Q_s-\tilde{Q}_s-\left(\bar{X}_s^{t,x;u}\right)^{-2}+2\left(\bar{X}_s^{t,x;u}\right)^{-3}\bar{u}_s+12\left(\bar{X}_s^{t,x;u}\right)^{-5}\bar{u}_s\\
      &\quad -10\left(\bar{X}_s^{t,x;u}\right)^{-6}\Big\}d s -Q_sd W_s-\tilde{Q}_s\tilde{N}(d s)\\
P_T=&-(\bar{X}_T^{t,x;u})^{-2}.
\end{aligned}\right.
 \label{second-order adjoint equation of example 2}
\end{equation}
Then, applying It\^{o}'s formula to $(\bar{X}_s^{t,x;\bar{u}})^{-2}$, we obtain that
\begin{equation}
\begin{aligned}
-d(-(\bar{X}_s^{t,x;\bar{u}})^{-2})=&\ \Big\{-2(\bar{X}_s^{t,x;\bar{u}})^{-3}\bar{u}_s+3(\bar{X}_s^{t,x;\bar{u}})^{-4}\bar{u}_s^2\Big\}d s-2(\bar{X}_s^{t,x;\bar{u}})^{-3}\bar{u}_sd W_s \\
&\ -\left\{-(\bar{X}_{s-}^{t,x;\bar{u}}+\bar{u}_s)^{-2}+(\bar{X}_{s-}^{t,x;\bar{u}})^{-2}\right\}\tilde{N}(d s) \\
&\ -\Big\{-(\bar{X}_s^{t,x;\bar{u}}+\bar{u}_s)^{-2}+(\bar{X}_s^{t,x;\bar{u}})^{-2}-2(\bar{X}_s^{t,x;\bar{u}})^{-3}\bar{u}_s\Big\}d s\\
=&\ \Big\{-(\bar{X}_s^{t,x;\bar{u}})^{-2}+3(\bar{X}_s^{t,x;\bar{u}})^{-4}\bar{u}_s^2+(\bar{X}_s^{t,x;\bar{u}}+\bar{u}_s)^{-2}\Big\}d s\\
&\ -2(\bar{X}_s^{t,x;\bar{u}})^{-3}\bar{u}_sd W_s-\left\{-(\bar{X}_{s-}^{t,x;\bar{u}}+\bar{u}_s)^{-2}+(\bar{X}_{s-}^{t,x;\bar{u}})^{-2}\right\}\tilde{N}(d s)\\
=&\ \Big\{(\bar{X}_s^{t,x;\bar{u}})^{-2}-2(\bar{X}_s^{t,x;\bar{u}})^{-3}\bar{u}_s-\left[-(\bar{X}_s^{t,x;\bar{u}}+\bar{u}_s)^{-2}+(\bar{X}_s^{t,x;\bar{u}})^{-2}\right]\\
&\qquad +\left[-(\bar{X}_s^{t,x;\bar{u}})^{-2}+2(\bar{X}_s^{t,x;\bar{u}})^{-3}\bar{u}_s+3(\bar{X}_s^{t,x;\bar{u}})^{-4}\bar{u}_s^2\right]\Big\}d s\\
&\ -2(\bar{X}_s^{t,x;\bar{u}})^{-3}\bar{u}_sd W_s-\left\{-(\bar{X}_{s-}^{t,x;\bar{u}}+\bar{u}_s)^{-2}+(\bar{X}_{s-}^{t,x;\bar{u}})^{-2}\right\}\tilde{N}(d s).
\end{aligned}
 \label{Ito formula of example 2}
\end{equation}

Noting $\bar{u}_s=(\bar{X}_s^{t,x;\bar{u}})^{-1}$, we have
$$
\begin{aligned}
&-(\bar{X}_s^{t,x;u})^{-2}+2(\bar{X}_s^{t,x;u})^{-3}\bar{u}_s+12(\bar{X}_s^{t,x;u})^{-5}\bar{u}_s-10(\bar{X}_s^{t,x;u})^{-6}\\
<& -(\bar{X}_s^{t,x;\bar{u}})^{-2}+2(\bar{X}_s^{t,x;\bar{u}})^{-3}\bar{u}_s+3(\bar{X}_s^{t,x;\bar{u}})^{-4}\bar{u}_s^2.
\end{aligned}
$$

Finally, $P_s<-V_{x x}(s,\bar{X}_s^{t,x;\bar{u}})$ holds by the comparison theorem of BSDEP (see \cite{LiPeng2009}).

$\mathbf{Example \ 5.3.}$ \ Given $t\in[0,T)$, consider the following controlled FBSDEP for $s\in[t,T]$ ($n=1$):
\begin{equation}
\left\{\begin{aligned}
dX_s^{t,x;u}&=\left(X_s^{t,x;u}+2X_s^{t,x;u}u_s\right)d s + X_s^{t,x;u}u_sd W_s + X_{s-}^{t,x;u}u_s \tilde{N}(d s), \\
-dY_s^{t,x;u}&=\left(-Z_s^{t,x;u}u_s-\tilde{Z}_s^{t,x;u}\right)d s-Z_s^{t,x;u}d W_s-\tilde{Z}_s^{t,x;u} \tilde{N}(d s), \\
X_t^{t,x;u}&=x, \quad Y_T^{t,x;u}=X_T^{t,x;u},
\end{aligned}\right.
 \label{state equation of example 3}
\end{equation}
with $\mathbf{U}=[-1,0] \cup [1,2]$. The cost functional is defined as (\ref{cost functional}) and the corresponding generalized HJB equation (\ref{HJB equation}) reads as
\begin{equation}
\left\{\begin{aligned}
&-v_t(t,x)+\sup_{u\in\mathbf{U}} \left\{-\frac{1}{2}v_{x x}(t,x)x^2u^2+v_x(t,x)xu^2-v_x(t,x)xu-v_x(t,x)x\right\}=0, \\
&\quad v(T,x)=-x.
\end{aligned}\right.
 \label{HJB equation of example 3}
\end{equation}
It is easy to verify that the viscosity solution to (\ref{HJB equation of example 3}) is given by
\begin{equation}
V(t, x)=
\begin{cases}
 -e^{t-T} x, & \text { if } x \leqslant 0, \\
 -e^{T-t} x, & \text { if } x>0.
\end{cases}
 \label{value function of example 3}
\end{equation}
Moreover, the first-order and second-order adjoint equations (\ref{first-order adjoint equation}) and (\ref{second-order adjoint equation}) are
\begin{equation}
\left\{\begin{aligned}
-d p_s & =\left\{\left[-(\bar{u}_s)^2+\bar{u}_s+1\right] p_s-\tilde{q}_s\right\} d s-q_s d W_s-\tilde{q}_s\tilde{N}(ds), \\
p_T & =1,
\end{aligned}\right.
\end{equation}
\begin{equation}
\left\{\begin{aligned}
-d P_s & =\left\{\left[-(\bar{u}_s)^2+2\bar{u}_s+2 \right] P_s+\bar{u}_s Q_s-\tilde{Q}_s\right\} d s-Q(s) d W_s-\tilde{Q}_s\tilde{N}(ds), \\
 P_T & =0,
\end{aligned}\right.
\end{equation}
respectively. Note that $(P_s,Q_s,\tilde{Q}_s)=(0,0,0)$ for any $\bar{u}_\cdot \in \mathcal{U}^\omega[t,T]$.

If the initial state $x < 0$, from the comparison theorem of SDEP (\cite{IkedaWatanabe1981}), we have $\bar{X}_s^{t,x;\bar{u}}\leqslant 0$ for all $s\in[t,T]$. We get that $(p_s,q_s,\tilde{q}_s)=(e^{s-T},0,0)$ for $\bar{u}_s=-1$ or 2, and by (\ref{definition of partial super-subjets to x}) we can figure that
$$
\mathcal{D}_x^{2,+} V\left(s,\bar{X}_s^{t,x;u}\right)=\left\{-e^{s-T}\right\}\times [0,\infty),\quad \forall\ s\in[t,T].
$$

If $x > 0$, from again the comparison theorem of SDEP, we have $\bar{X}_s^{t,x;\bar{u}}\geqslant 0$ for all $s\in[t,T]$. We get that $(p_s,q_s,\tilde{q}_s)=(e^{T-s},0,0)$ for $\bar{u}_s=0$ or 1, and we can also figure that
$$
\mathcal{D}_x^{2,+} V\left(s,\bar{X}_s^{t,x;u}\right)=\left\{-e^{T-s}\right\}\times [0,\infty),\quad \forall\ s\in[t,T].
$$

If $x=0$, we have $\bar{X}_\cdot^{t,x;\bar{u}}=0$ for any $u_\cdot \in \mathcal{U}^\omega[t,T]$. From (\ref{value function of example 3}), we have
$$
\mathcal{D}_x^{2,-} V\left(s,\bar{X}_s^{t,x;u}\right)=\emptyset,\quad \mathcal{D}_x^{2,+} V\left(s,\bar{X}_s^{t,x;u}\right)=\left[-e^{T-s},-e^{s-T}\right]\times [0,\infty).
$$
Conversely, (\ref{relation of p,P and V}) always holds.

\section{Concluding remarks }

In this paper, we have investigated the relationship between the general maximum principle and dynamic programming principle for the stochastic optimal control problem of jump diffusions, where the control domain is not necessarily convex. We have first discussed the relationship among the adjoint process, generalized Hamiltonian function and value function when value function is smooth. Then, under the framework of viscosity solutions, we have further researched the case of the value function being not necessarily smooth. The results obtained in this paper cover the case without random jumps (\cite{YongZhou1999}, \cite{NieShiWu2017}), and generalize the case with jumps but the control domain is forced to assume to be a convex set and the value function to be smooth (\cite{Shi2014}).

A challenging problem is when the controlled FSBDEP (\ref{FBSDEP equation}) is generalized to a the following fully coupled one (see \cite{Wu1999}, \cite{Shi2012-}, \cite{LiWei2014}, \cite{YangMoon2023}):
\begin{equation*}
\left\{\begin{aligned}
d X_s^{t, x ; u}= & \ b\left(s, X_s^{t, x ; u},  Y_s^{t, x ; u}, Z_s^{t, x ; u},\tilde{Z}_{(s,\cdot)}^{t,x;u}, u_s\right) d s\\
 & +\sigma\left(s, X_s^{t, x ; u},  Y_s^{t, x ; u}, Z_s^{t, x ; u},\tilde{Z}_{(s,\cdot)}^{t,x;u}, u_s\right) d W_s\\
 & +\int_{\mathcal{E}}f(s,X_{s-}^{t,x;u}, Y_{s-}^{t, x ; u}, Z_{s-}^{t, x ; u},\tilde{Z}_{(s-,\cdot)}^{t,x;u}, u_s,e)\tilde{N}(d e,d s), \\
-d Y_s^{t, x ; u}= & \ g\left(s, X_s^{t, x ; u}, Y_s^{t, x ; u}, Z_s^{t, x ; u},\tilde{Z}_{(s,\cdot)}^{t,x;u}, u_s\right) d s-Z_s^{t, x ; u} d W_s-\int_{\mathcal{E}}\tilde{Z}_{(s,e)}^{t,x;u}\tilde{N}(d e,d s), \\
X_t^{t, x ; u}= & \ x, \ \ \ \ \ Y_T^{t, x ; u}=\phi\left(X_T^{t, x ; u}\right),
\end{aligned}\right.
 \label{fully coupled FBSDEP equation}
\end{equation*}
how to characterize the relationship between the general MP and DPP? A nice literature is \cite{HuJiXue2020}, where the authors dealt with the problem without jumps. We will research this topic in the near future.


\end{CJK}

\begin{thebibliography}{0}

\bibitem{Antonelli1993}F. Antonelli, Backward-forward stochastic differential equations, \emph{Anna. Appl. Proba.}, 3, 777-793, 1993.

\bibitem{BBP1997}G. Barles, R. Buckdahn, and E. Pardoux, Backward stochastic differential equations and integral-partial differential equations, \emph{Stoch. $\&$ Stoch. Reports}, 60, 57-83, 1997.

\bibitem{BarlesImbert2008}G. Barles, C. Imbert, Second-order elliptic integro-differential equations: viscosity solutions' theory revisited, \emph{Ann. I. H. Poincar\'{e} - AN}, 25, 567-585, 2008.

\bibitem{Bismut1978}J.M. Bismut, An introductory approach to duality in optimal stochastic control, \emph{SIAM Review}, 20, 62-78, 1978.

\bibitem{ChenLv2023}L.Y. Chen, Q. Lv, Relationships between the maximum principle and dynamic programming for infinite dimensional stochastic control systems, \emph{J. Differential Equations}, 358, 103-146, 2023.

\bibitem{ChighoubMezerdi2014}F. Chighoub, B. Mezerdi, The relationship between the stochastic maximum principle and the dynamic programming in singular control of jump diffusions, \emph{Int. J. Stoch. Anal.}, 2014, Article ID 201491, 17 pages, 2014.

\bibitem{Duffie1992}D. Duffie, L.G. Epstein, Stochastic differential utility, \emph{Econometrica}, 60, 353-394, 1992.

\bibitem{FOS2004}N.C. Framstad, B. \O ksendal, and A. Sulem, A sufficient stochastic maximum principle for optimal control of jump diffusions and applications to finance, \emph{J. Optim. Theory Appl.},  121, 77-98, 2004. (Erratum: J. Optim. Theory Appl., 124, 511-512, 2005.)

\bibitem{Hu2017}M.S. Hu, Stochastic global maximum principle for optimization with recursive utilities, \emph{Proba. Uncer. Quan. Risk}, 2, 1-20, 2017.

\bibitem{HuJiXue2018}M.S. Hu, S.L. Ji, and X.L. Xue, A global stochastic maximum principle for fully coupled forward-backward stochastic systems, \emph{SIAM J. Control Optim.}, 56, 4309-4335, 2018.

\bibitem{HuJiXue2019}M.S. Hu, S.L. Ji, and X.L. Xue, The existence and uniqueness of viscosity solution to a kind of Hamilton-Jacobi-Bellman equation, \emph{SIAM J. Control Optim.}, 57, 3911-3938, 2019.

\bibitem{HuJiXue2020}M.S. Hu, S.L. Ji, and X.L. Xue, Stochastic maximum principle, dynamic programming principle, and their relationship for fully coupled forward-backward stochastic controlled systems, \emph{ESAIM Control Optim. Calc. Var.}, 26, Article no. 81, 2020.

\bibitem{EPQ1997}N. El Karoui, S.G. Peng, and M.C. Quenez, Backward stochastic differential equations in finance, \emph{Math. Finan.}, 7, 1-71, 1997.

\bibitem{IkedaWatanabe1981}N. Ikeda, S. Watanabe, \emph{Stochastic Differential Equations and Diffusion Processes}, North Holland, 1981.

\bibitem{LiPeng2009}J. Li, S.G. Peng, Stochastic optimization theory of backward stochastic differential equations with jumps and viscosity solutions of Hamilton-Jacobi-Bellman equations, \emph{Non. Anal.: Theory, Methods Appl.}, 70, 1776-1796, 2009.

\bibitem{LiWei2014}J. Li, Q.M. Wei, $L^p$ estimates for fully coupled FBSDEs with jumps, \emph{Stoch. Proc. Appl.}, 124, 1582-1611, 2014.

\bibitem{Li2023}X.J. Li, Relationship between maximum principle and dynamic programming principle for stochastic recursive optimal control problem under volatility uncertainty, \emph{Optimal Control Appl. Methods}, 44, 2457-2475, 2023.

\bibitem{MoonBasar2022}J. Moon, T, Ba\c{s}ar, Dynamic programming and a verification theorem for the recursive stochastic control problem of jump-diffusion models with random coefficients, \emph{IEEE Trans. Autom. Control}, 67, 6474-6488, 2022.

\bibitem{NieShiWu2017}T.Y. Nie, J.T. Shi, and Z. Wu, Connection between MP and DPP for stochastic recursive optimal control problems: viscosity solution framework in the general case, \emph{SIAM J. Control Optim.}, 55, 3258-3294, 2017.

\bibitem{OksendalSulem2006}B. \O ksendal and A. Sulem, \emph{Applied Stochastic Control of Jump Diffusions}, 2nd ed., Springer, Berlin, 2006.

\bibitem{OksendalSulem2009}B. \O ksendal, A. Sulem, Maximum principles for optimal control of forward-backward stochastic differential equations with jumps, \emph{SIAM J. Control Optim.}, 48, 2945-2976, 2009.

\bibitem{PardouxPeng1990}E. Pardoux, S.G. Peng, Adapted solution of a backward stochastic differential equation, \emph{Syst. $\&$ Control Lett.}, 14, 55-61, 1990.

\bibitem{Pham1998}H. Pham, Optimal stopping of controlled jump diffusion processes: A viscosity solution approach, \emph{J. Math. Syst. Estimation Control}, 8, 1-27, 1998.

\bibitem{Peng1990}S.G. Peng, A general stochastic maximum principle for optimal control problems, \emph{SIAM J. Control Optim.}, 28, 966-979, 1990.

\bibitem{Peng1992}S.G. Peng, A generalized dynamic programming principle and Hamilton-Jacobi-Bellman equation, \emph{Stoch. $\&$ Stoch. Reports}, 38, 119-134, 1992.

\bibitem{Peng1993}S.G. Peng, Backward stochastic differential equations and applications to optimal control, \emph{Appl. Math. $\&$ Optim.}, 27, 125-144, 1993.

\bibitem{Shi2012}J.T. Shi, Global maximum principle for the forward-backward stochastic optimal control problem with Poisson jumps, \emph{Asian J. Control}, 14, 1355-1365, 2012.

\bibitem{Shi2012-}J.T. Shi, Necessary conditions for optimal control of forward-backward stochastic systems with random jumps, \emph{Inter. J. Stoch. Anal.}, 2012, Article ID 258674, 50 pages, 2012.

\bibitem{Shi2014}J.T. Shi, Relationship between maximum principle and dynamic programming principle for stochastic recursive optimal control problems of jump diffusions, \emph{Optim. Control Appl. Methods}, 35, 61-76, 2014.

\bibitem{ShiWu2010}J.T. Shi, Z. Wu, Maximum principle for forward-backward stochastic control systems with random jumps and applications to finance, \emph{J. Syst. Sci. Complex.}, 23, 219-231, 2010.

\bibitem{ShiWu2011}J.T. Shi, Z. Wu, Relationship between MP and DPP for the stochastic optimal control problem of jump diffusions, \emph{Appl. Math. $\&$ Optim.}, 63, 151-189, 2011.

\bibitem{ShiYu2013}J.T. Shi, Z.Y. Yu, Relationship between maximum principle and dynamic programming for stochastic recursive optimal control problems and applications, \emph{Math. Probl. Engin.}, 2013, Article ID 285241, 12 pages, 2013.

\bibitem{Situ1991}R. Situ, A maximum principle for optimal controls of stochastic systems with random jumps, in \emph{Proc. National Conference on Control Theory and Applications}, Qingdao, China, 1991.

\bibitem{SongTangWu2020}Y.Z. Song, S.J. Tang, and Z. Wu, The maximum principle for progressive optimal stochastic control problems with random jumps, \emph{SIAM J. Control Optim.}, 58, 2171-2187, 2020.

\bibitem{SongWu2023}Y.Z. Song, Z. Wu, The maximum principle for stochastic control problem with jumps in progressive structure, \emph{J. Optim. Theory Appl.}, 199, 415-438, 2023.

\bibitem{SunGuoZhang2018}Z.Y. Sun, J.Y. Guo, and X. Zhang, Maximum principle for Markov regime-switching forward-backward stochastic control system with jumps and relation to dynamic programming, \emph{J. Optim. Theory Appl.}, 176, 319-350, 2018.

\bibitem{TangLi1994}S.J. Tang, X.J. Li, Necessary conditions for optimal control of stochastic systems with random jumps, \emph{SIAM J. Control Optim.}, 32, 1447-1475, 1994.

\bibitem{Wu1999}Z. Wu, Forward-backward stochastic differential equations with Brownian motion and Poisson processes, \emph{Acta Math. Appl. Sinica}, 15, 433-443, 1999.

\bibitem{Wu2013}Z. Wu, A general maximum principle for optimal control problems of forward-backward stochastic control systems, \emph{Automatica}, 49, 1473-1480, 2013.

\bibitem{WuYu2008}Z. Wu, Z.Y. Yu, Dynamic programming principle for one kind of stochastic recursive optimal control problem and Hamilton-Jacobi-Bellman equation, \emph{SIAM J. Control Optim.}, 47, 2616-2641, 2008.

\bibitem{YangMoon2023}H.J. Yang, J. Moon, A sufficient condition for optimal control problem of fully coupled forward-backward stochastic systems with jumps: A state-constrained control approach, \emph{Optim. Control Appl. Meth.}, 44, 1936-1971, 2023.

\bibitem{Yong2010}J.M. Yong, Optimality variational principle for controlled forward-backward stochastic differential equations with mixed initial-terminal conditions, \emph{SIAM J. Control Optim.}, 48, 4119-4156, 2010.

\bibitem{YongZhou1999}J.M. Yong, X.Y. Zhou, \emph{Stochastic Controls: Hamiltonian Systems and HJB Equations}, Springer-Verlag, New York, 1999.

\bibitem{ZhengShi2023}Y.Y. Zheng, J.T. Shi, The global maximum principle for progressive optimal control of partially observed forward-backward stochastic systems with random jumps, \emph{SIAM J. Control Optim.}, 61, 1063-1094, 2023.

\bibitem{Zhou1990}X.Y. Zhou, The connection between the maximum principle and dynamic programming in stochastic control, \emph{Stoch. $\&$ Stoch. Reports}, 31, 1-13, 1990.

\bibitem{Zhu2010}X.H. Zhu, Backward stochastic viability property with jumps and applications to the comparison theorem for multidimensional BSDEs with jumps, \emph{arXiv:1006.1453}.

\end{thebibliography}
\end{document}